\documentclass[a4paper, 10pt, american]{amsart}

\usepackage{times,latexsym,amssymb}
\usepackage{amsmath,amsthm,bm}
\usepackage{xcolor}
\usepackage[colorlinks,pdfpagelabels,pdfstartview = FitH,bookmarksopen
= true,bookmarksnumbered = true,linkcolor = blue,plainpages =
false,hypertexnames = false,citecolor = red,pagebackref=false]{hyperref}

\usepackage{amsbsy}
\usepackage{amstext}
\usepackage{amssymb}
\usepackage{esint}
\usepackage{stmaryrd}
\usepackage[pdftex]{graphicx}
\usepackage{floatflt}
\usepackage{comment}

\usepackage{tikz,pgfplots}
\pgfplotsset{compat=1.10}
\usepgfplotslibrary{fillbetween}

\allowdisplaybreaks
\newtheorem{theorem}{Theorem}
\newtheorem{lemma}[theorem]{Lemma}
\newtheorem{definition}[theorem]{Definition}
\newtheorem{proposition}[theorem]{Proposition}
\newtheorem{corollary}[theorem]{Corollary}
\newtheorem{remark}[theorem]{Remark}
\numberwithin{theorem}{section}
\numberwithin{equation}{section}

\newcommand{\mint}{- \mskip-19,5mu \int}
\newcommand{\tmint}{- \mskip-16,5mu \int}
\def\N{\mathbb{N}}
\def\R{\mathbb{R}}

\newcommand{\dx}{\mathrm{d}x}
\newcommand{\dy}{\mathrm{d}y}
\newcommand{\dz}{\mathrm{d}z}
\newcommand{\dt}{\mathrm{d}t}

\renewcommand{\epsilon}{\varepsilon}

\newcommand{\al}{\alpha}

\newcommand{\be}{\beta}

\newcommand{\gm}{\gamma}
\newcommand{\dl}{\delta}

\newcommand{\lm}{\lambda}

\newcommand{\varep}{\varepsilon}

\newcommand{\sig}{\sigma}

\newcommand{\Om}{\Omega}

\DeclareMathOperator{\spt}{spt}

\DeclareMathOperator{\dist}{dist}
\DeclareMathOperator{\loc}{loc}

\DeclareMathOperator{\Tail}{Tail}
\renewcommand{\epsilon}{\varepsilon}
\newcommand{\eps}{\varepsilon}
\renewcommand{\rho}{\varrho}

\def\eqn#1$$#2$${\begin{equation}\label#1#2\end{equation}}

\newcommand{\btau}{\boldsymbol{\tau}}

\def\Xint#1{\mathchoice
    {\XXint\displaystyle\textstyle{#1}}%
    {\XXint\textstyle\scriptstyle{#1}}%
    {\XXint\scriptstyle\scriptscriptstyle{#1}}%
    {\XXint\scriptscriptstyle\scriptscriptstyle{#1}}%
    \!\int}
\def\XXint#1#2#3{\setbox0=\hbox{$#1{#2#3}{\int}$}
    \vcenter{\hbox{$#2#3$}}\kern-0.5\wd0}
\def\bint{\Xint-}
\def\dashint{\Xint{\raise4pt\hbox to7pt{\hrulefill}}}
\def\tmint{\Xint{\raise0pt\hbox to6pt{\hrulefill}}}

\def\XXiint#1#2#3{\setbox0=\hbox{$#1{#2#3}{\iint}$}
    \vcenter{\hbox{$#2#3$}}\kern-0.5\wd0}


\author[V. B\"ogelein]{Verena B\"{o}gelein}
\address{Verena B\"ogelein\\
Fachbereich Mathematik, Universit\"at Salzburg\\
Hellbrunner Str. 34, 5020 Salzburg, Austria}
\email{verena.boegelein@plus.ac.at}

\author[F. Duzaar]{Frank Duzaar}
\address{Frank Duzaar\\
Fachbereich Mathematik, Universit\"at Salzburg\\
Hellbrunner Str. 34, 5020 Salzburg, Austria}
\email{frankjohannes.duzaar@plus.ac.at}

\author[N. Liao]{Naian Liao}
\address{Naian Liao\\
Fachbereich Mathematik, Universit\"at Salzburg\\
Hellbrunner Str. 34, 5020 Salzburg, Austria}
\email{naian.liao@plus.ac.at}

\author[G. Molica Bisci]{Giovanni Molica Bisci}
\address{Giovanni Molica Bisci\\ 
Dipartimento di Scienze Pure e Applicate (DiSPeA), University of Urbino Carlo Bo\\
Piazza della Repubblica, 13, 61029 Urbino, Italy}
\email{giovanni.molicabisci@uniurb.it}

\author[R. Servadei]{Raffaella Servadei}
\address{Raffaella Servadei\\ 
Dipartimento di Scienze Pure e Applicate (DiSPeA), University of Urbino Carlo Bo\\
Piazza della Repubblica, 13, 61029 Urbino, Italy}
\email{raffaella.servadei@uniurb.it}

\begin{document}

\title[Gradient regularity for $(s,p)$-harmonic functions]{Gradient regularity for $(s,p)$-harmonic functions}

\date{\today}

\begin{abstract}
We study the local regularity properties of $(s,p)$-harmonic functions, i.e.~local weak solutions to the fractional $p$-Laplace equation of order $s\in (0,1)$ in the case $p\in (1,2]$. It is shown that $(s,p)$-harmonic functions are weakly differentiable and that the weak gradient is locally integrable to any power $q\ge 1$.
As a result, $(s,p)$-harmonic  functions are H\"older continuous to arbitrary H\"older exponent in $(0,1)$.
In addition, the weak gradient of $(s,p)$-harmonic functions has certain fractional differentiability. All estimates are stable when $s$ reaches $1$, and the known regularity properties of $p$-harmonic functions are formally recovered, in particular the local $W^{2,2}$-estimate.

\end{abstract}
\maketitle
\tableofcontents

\section{Introduction}
In this manuscript we complete our program on the higher Sobolev regularity of solutions to the fractional $p$-Laplace equation with order $s$:
\begin{equation}\label{PDE}
    (-\Delta_p)^s u = 0 \quad\text{in}\>\>\Om.
\end{equation}
Here, we denoted a bounded domain $\Om$ in $\R^N$, $N\ge2$.
Local solutions to \eqref{PDE} are termed $(s,p)$-harmonic functions in $\Om$; see~Definition~\ref{def:loc-sol}.
In our previous work \cite{BDLMS}, we considered the case $p\in[2,\infty)$, whereas in this manuscript we focus on the case $p\in(1,2]$.

The \textbf{main result} states that when $p\in(1,2]$ and $s\in(0,1)$, the gradient of $(s,p)$-harmonic functions in $\Om$ exists and belongs to $L^q_{\loc}(\Om)$ for any $q\in[p,\infty)$. As a result, $(s,p)$-harmonic functions are locally H\"older continuous in $\Om$ with an arbitrary H\"older exponent $\gm\in(0,1)$. In addition, we establish the fractional differentiability of the gradient in any $L^q$-scale, namely, $\nabla u\in W^{\al,q}_{\loc}(\Om)$ for any $q\in[2,\infty)$ and any $\al\in(0,\max\{\frac{sp}{q}, \frac{1-(1-s)p}{q-1}\})$. In particular, letting $q=2$ and $s\uparrow1$, we formally recover the well-known $W^{2,2}$-regularity of $p$-harmonic functions, cf.~\cite{Manfredi-Weitsman, Acerbi-Fusco}. All claimed regularity properties are confirmed by explicit local estimates that are stable as $s\uparrow1$.

A first decisive result on the gradient regularity for $(s,p)$-harmonic functions belongs to Brasco \& Lindgren \cite{Brasco-Lindgren}. They established $\nabla u\in L^p_{\loc}(\Om)$ in the case $p\in[2,\infty)$ and $s\in(\frac{p-1}{p},1)$ for globally bounded $(s,p)$-harmonic functions. Recently, we have substantially improved this result in \cite{BDLMS}. As a matter of fact, the same gradient regularity holds for a wider range of $s$, namely, $s\in(\frac{p-2}{p},1)$, and furthermore, global boundedness can also be dispensed with. More importantly, what really sets our contribution apart from others is that the integrability exponent of $\nabla u$ has been upgraded to any finite number $q$. 

In this work we make an important step forward and extend the gradient regularity for $(s,p)$-harmonic functions to the whole range $p\in(1,2]$ and $s\in(0,1)$. Therefore, summarizing the achievement in \cite{BDLMS} and in this work, we have obtained a rather complete picture regarding the higher Sobolev regularity for $(s,p)$-harmonic functions, see the ranges of $s$ and $p$ shown by the blue part in Figure~\ref{figure:new}. In addition, for the same ranges of $s$ and $p$, we have concluded the ``almost" Lipschitz regularity. Prior to our works, such regularity was only known for $s\in (\frac{p-1}{p},1)$; see~\cite{Brasco-Lindgren-Schikorra} for $p\in[2,\infty)$ and \cite{Lindgren-Garain} for $p\in(1,2]$.

\begin{figure}[h]
\begin{tikzpicture}[scale=1.3
](160,155)(0,0)
    \draw[dashdotted] (2,0) -- (2,1.5);
    \draw[dashdotted] (1,0) -- (1,1.5);
    \draw (2,-0.3) node[anchor=mid] {$2$};
    \draw (-0.3,1) node[anchor=mid] {$1$};
    \draw (-0.3,0) node[anchor=mid] {$0$};
    \draw (1,-0.3) node[anchor=mid] {$1$};
    \draw (1,-0.1) -- (1,0.1);
    \draw (-0.1,1) -- (0.1,1);
    \draw (2,-0.1) -- (2,0.1);
    \draw[->] (0, 0) -- (7, 0) node[right] {$p$};
    \draw[->] (0, 0) -- (0, 1.5) node[above] {$s$};
    \draw[dashdotted] (0, 1) -- (2, 1);
    \draw[dashdotted, name path=A] (2, 1) -- (7, 1);
    \draw[scale=1, thick, domain=2:7, smooth, variable=\p, blue, name path=B] plot ({\p}, {(\p-2)/\p});
    \tikzfillbetween[of=A and B]{blue, opacity=0.1};
    \filldraw [fill=blue, opacity=0.1] (1,0) rectangle (2,1);
    \draw[name path=C] (2, 0) -- (7, 0);
    \tikzfillbetween[of=B and C]{green, opacity=0.05};
\end{tikzpicture}
\caption{ Threshold: $s=\frac{(p-2)_+}{p}$ }\label{figure:new}
\end{figure}
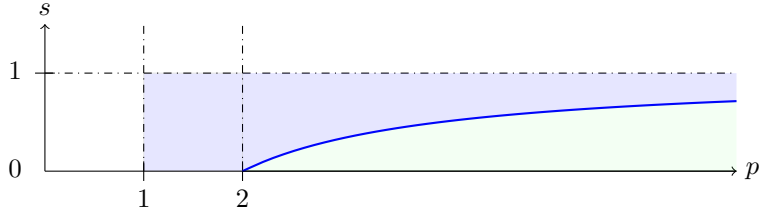

For simplicity, we choose to present and prove all our results for locally bounded $(s,p)$-harmonic functions. However, once this is done, we can employ proper freezing-and-comparison arguments, consider general inhomogeneous terms on the right-hand side of \eqref{PDE}, and also inserting appropriate coefficients to the equation. The estimates could also be modified in such a way that the $L^\infty$-norm is replaced by the $L^p$-norm of $u$. Conceivably, our higher Sobolev regularity for $(s,p)$-harmonic functions opens the way for a Calder\'on-Zygmund-type theory of $\nabla u$.

As in our previous work, we explicitly trace the dependence of the constants in estimates. We believe this is important and necessary in order to justify the stability of our regularity estimates as $s\uparrow1$. Likewise, we have also traced how constants grow in various iteration schemes. This task often involves great care. However, the reader is advised to skip these details on first reading.

Independent of us, the interesting preprint \cite{DKLN} established that $\nabla u\in L_{\loc}^{p}(\Om)$ for the same range of $s$ and $p$ as ours. In addition, it is shown that when $p\in(1,2)$, $\nabla u$ possesses any order of fractional differentibility less than $\max\{ \frac12 sp, 1-(1-s)p\}$ in the $L^p$-scale.

\subsection{Results}

We state the main results here. The first one concerns the existence of the gradient of $(s,p)$-harmonic functions in $L^p$-scale.

\begin{theorem}[$L^{p}$-gradient regularity]\label{thm:W1p-p<2}
Let $p\in (1, 2]$ and $s\in(0,1)$. Then, for any locally bounded $(s,p)$-harmonic function $u$ in the sense of Definition~\ref{def:loc-sol}, we have
$$
    u\in W^{1,p}_{\rm loc}(\Omega).
$$
Moreover, there exists a  constant $C=C(N,p,s)$, such that for any ball $B_{R}\equiv B_{R}(x_o)\Subset \Omega$ we have
\begin{align*}
    \|\nabla u\|_{L^p(B_{\frac12 R})}
    &\le 
    \frac{C}{R} \Big[  R^{s}(1-s)^{\frac1p}[u]_{W^{s,p}(B_{R})} +
    R^{\frac{N}{p}} \big(\|u\|_{L^\infty(B_{R})} + \Tail(u;R)\big)\Big].
\end{align*}
The constant $C$ is stable as  $s\uparrow 1$ and blows up as $p\downarrow 1$. 
\end{theorem}

The second result upgrades the integrability of $\nabla u$ to any exponent $q>p$.

\begin{theorem}[$L^q$-gradient regularity]\label{thm:W1q-p<2}
Let $p\in (1,2]$ and  $s\in(0,1)$. Then, for any locally bounded $(s,p)$-harmonic function $u$ in the sense of Definition~\ref{def:loc-sol}, we have
$$
    u\in W^{1,q}_{\rm loc}(\Omega)\qquad \mbox{for any $q\in [p,\infty)$.}
$$
Moreover, there exists a  constant  $C=C(N,p,s,q)$, such that for any ball $B_{R}\equiv B_{R}(x_o)\Subset \Omega$ the quantitative $L^q$-gradient estimate 
\begin{align*}
    &\|\nabla u\|_{L^q(B_{\frac12 R})} \le
    C R^{\frac{N}{q}-1} 
    \Big[R^{s-\frac{N}{p}} (1-s)^{\frac1p}[u]_{W^{s,p}(B_{R})} +
    \|u\|_{L^{\infty}(B_{R})}+\mathrm{Tail}(u;R)\Big]
\end{align*}
holds true. The constant $C$ is stable in the limit  $s\uparrow 1$  and blows up as $p\downarrow 1$.
\end{theorem}

The classical  Morrey-type embedding states that any $W^{1,q}$ function with $q>N$ is locally Hölder continuous with exponent $1-\frac{N}{q}$. Therefore, a direct consequence of Theorem~\ref{thm:W1q-p<2} is that any $(s,p)$-harmonic function is H\"older continuity for any H\"older exponent in $ (0,1)$. 

\begin{theorem}[Almost Lipschitz continuity]\label{thm:Hoelder-p<2}
Let $p\in (1,2]$ and $s\in(0,1)$.  Then,  for any locally bounded $(s,p)$-harmonic function $u$ in the sense of Definition~\ref{def:loc-sol}, we have
$$
    u\in C_{\loc}^{0,\gamma}(\Omega)\qquad\mbox{for any $\gamma\in(0,1)$.}\
$$
Moreover,  there exists a constant  $C=C(N,p,s,\gamma)$, such for any ball $B_R\equiv B_R(x_o)\Subset\Omega$ we have
\begin{align*}
     [u]_{C^{0,\gamma}(B_{\frac12 R})}
     \le 
     \frac{C}{R^{\gamma}}
    \Big[ R^{s-\frac{N}p}(1-s)^{\frac1p}[u]_{W^{s,p} (B_{R})} +
    \|u\|_{L^{\infty}(B_{R})} +
    \mathrm{Tail}(u;R)\Big] .
\end{align*}
The constant $C$ is stable as $s\uparrow1$ and blows up as $p\downarrow 1$.  
\end{theorem}

The gradient of $(s,p)$-harmonic functions also possesses certain fractional differentiability.

\begin{theorem}[Fractional differentiability of the gradient]\label{thm:grad-frac}
Let $p\in (1,2]$ and  $s\in(0,1)$ and $q\in[2,\infty)$. Then, for any locally bounded $(s,p)$-harmonic function $u$ in the sense of Definition~\ref{def:loc-sol}, we have
$$
    \nabla u\in W^{\alpha,q}_{\rm loc}(\Omega)\qquad 
    \mbox{for any $\alpha\in(0,\beta)$, where $\beta:=\max\Big\{\frac{sp}{q}, \frac{1-(1-s)p}{q-1}\Big\}$.}
$$
Moreover, there exists a constant $C=C(N,p,s,q,\alpha)$ such that for any ball $B_{R}\equiv B_{R}(x_o)\Subset \Omega$, we have
\begin{align*}
    [u]_{W^{\al,q}(B_{\frac12 R})}
    &\le 
    \frac{C R^{\frac{N}{q}}}{R^{1+\alpha}}
    \Big[R^{s-\frac{N}{p}}(1-s)^{\frac1p}[u]_{W^{s,p}(B_{R})} +
    \|u\|_{L^{\infty}(B_{R})}+\mathrm{Tail}(u; R)\Big].
\end{align*}
The constant $C$ is stable as $s\uparrow1$ and blows up as $p\downarrow 1$.
\end{theorem}

\begin{remark}\label{rem:diff-p}\upshape
The application of Theorem \ref{thm:grad-frac} with $q=2$ ensures that any locally bounded $(s,p)$-harmonic function satisfies
$$
    \nabla u\in W^{\alpha,2}_{\rm loc}(\Omega)\qquad 
    \mbox{for any $\alpha\in(0,\beta)$, where $\beta:=\max\Big\{\frac{sp}{2}, 1-(1-s)p\Big\}$.}
$$
In turn, we conclude that  
$$
    \nabla u\in W^{\alpha,q}_{\rm loc}(\Omega)\qquad 
    \mbox{for any $\alpha\in(0,\beta)$ and any $q\in[1,2]$.}
$$
\end{remark}

\subsection{Brief review of the literature}\label{state-of-art}
Nonlocal equations have attracted a lot of attention in recent years, and interest in them continues to grow. Regarding regularity theory, the investigations started with linear fractional differential operators; see \cite{Bass-1,Bass-2,Cozzi-1,Dipiero-Savin-Valdinoci,Dipiero-Ros Oton-Serra-Valdinoci,Dyda-Kassmann,Fall,Imbert-Silvestre,Jin-Xiong,Kassmann,Ros-Oton-Serra,Serra-1,Serra-2,Silvestre-06} and the references therein. For global and boundary regularity we refer to \cite{Abatangelo-Ros-Oton, Abels-Grubb, Grubb-1,Grubb-2,Ros-Oton-Serra:bd1,Ros-Oton-Serra:bd2}. In the framework of fractional Sobolev spaces, self-improving properties of Gehring-type were studied in \cite{Auscher-Bortz-Saari, Bass-Ren, Kuusi-Mingione-Sire-1, Kuusi-Mingione-Sire-2}, while estimates of Calder\'on-Zygmund-type were obtained in \cite{Fall-Mengesha-Schikorra-Yeepo, Mengesha-Schikorra-Yeepo}. Measure data problems and pointwise gradient potential estimates were considered in \cite{Diening-Kim-Lee-Nowak, Kuusi-Nowak-Sire}.

The regularity theory for $(s,p)$-harmonic fuctions is still fragmented, and many of the fundamental questions are unanswered. 
Boundedness and Hölder continuity of weak solutions in the interior and at the boundary were proved in \cite{Cozzi-1, DiCastro-Kuusi-Palatucci, Iannizzotto-Mosconi, Iannizzotto-Mosconi-Squassina-2016, Iannizzotto-Mosconi-Squassina, Korvenpää-Kuusi-Palatucci}; see \cite{Lindgren-1} for a similar result in the framework of viscosity solutions. Harnack's inequality was established in \cite{DKP-2}; see also~\cite{Cozzi-1}. 
A Wiener-type criterion for boundary regularity was obtained in \cite{Kim-Lee-Lee}. Regularity at the gradient level is less developed. However, there are some results on level of fractional derivatives. A self-improving Gehring-type property was established for $p\ge2$ in \cite{Schikorra-1}, while a Calder\'on-Zygmund theory was developed in \cite{Byun-Kim, Diening-Nowak}. Finally, in \cite{Kuusi-Mingione-Sire-3} measure data problems and potential estimates were considered.
Recently, the study of mixed local and non-local problems has gained much interest; see the non-exhaustive list \cite{BDVV-mixed, Byun-Song-mixed, DM-mixed, GL-mixed} and the references therein. 

The above mentioned results guarantee Hölder continuity of $(s,p)$-harmonic functions with a qualitative Hölder exponent.
The proofs in \cite{Cozzi-1, DiCastro-Kuusi-Palatucci} are based on either De Giorgi's or Moser's approach. Similar to the local case, Hölder continuity with a quantitative H\"older exponent or gradient regularity requires other techniques such as second-order finite difference estimates in combination with a Moser-type iteration scheme. Roughly speaking, the second-order finite difference estimates in the fractional context replace the second weak derivative estimates from the local case. The latter  can be derived by the difference quotient method.
This technique has proven to be very efficient for local operators; we refer to  \cite{Acerbi-Fusco, Uhlenbeck, Uraltseva}. 

For $(s,p)$-harmonic functions, Brasco \& Lindgren \cite{Brasco-Lindgren} established  higher Sobolev regularity for the case $p\ge2$ when $s>\frac{p-1}{p}$. 
Later, higher H\"older continuity was proved in \cite{Brasco-Lindgren-Schikorra} by Brasco \& Lindgren \& Schikorra. More recently, in~\cite{BDLMS} we were able to improve the threshold for $s$ to $s>\frac{p-2}{p}$.  At the same time we improved the level of integrability for the gradient to any exponent $q\ge p$, together with the fractional differentiability $W^{\alpha,q}_{\rm loc}$ for any $\alpha \in\big(0,\tfrac{sp-(p-2)}{q}\big)$. As a consequence, we also improved the higher H\"older regularity. Shortly afterwards, Diening \& Kim \& Lee \& Nowak~\cite{DKLN} were able to establish the same amount of fractional differentiability for the inhomogeneous problem at the level $q=p$.

For the case $p<2$,   Garain \& Lindgren obtained a quantitative higher Hölder estimate in~\cite{Lindgren-Garain}. When it comes to higher Sobolev regularity, \cite{DKLN} showed that the weak gradient exists in $L^p_{\rm loc}$, and meanwhile admits a certain fractional differentiability. However, the quantitative integrability of the weak gradient does not go beyond $L^p$.
The present paper -- independent of \cite{DKLN} -- establishes the existence of the weak gradient, its higher integrability  at any $L^q$ level with $q\ge 1$ as well as a certain fractional differentiability.

\subsection{Master plan}
The so-called difference quotient technique, even though known to be part of the repertoire in the analysis of partial differential equations for decades, is fairly recent on the stage of integro-differential equations. For $(s,p)$-harmonic functions, it was implemented by Brasco \& Lindgren \cite{Brasco-Lindgren}, Brasco  \& Lindgren \& Schikorra \cite{Brasco-Lindgren-Schikorra}, and Garain \& Lindgren \cite{Lindgren-Garain}; see also \cite{Caffarelli-Silvestre, Cozzi-2} under different settings. Our approach also relies on this barehanded technique.

The thrust is quite simple, namely, we differentiate the equation at a discrete level, coupled with a careful choice of testing functions that cope with the discretized nonlinear structure. However, in contrast to the classical treatment of $p$-harmonic functions, some challenging, new features appear in the fractional setting. To successfully implement the technique to analyze regularity properties of $(s,p)$-harmonic functions, one needs to skillfully balance their local differentiability and integrability, and to precisely capture the long-range behavior of their finite differences. Once this is done, we employ various iteration schemes to upgrade the differentiability and subsequently the integrability.

In this program, the first important step lies in raising the preset $W^{s,p}$-regularity to $W^{1,p}$-regularity as stated in Theorem~\ref{thm:W1p-p<2}.  
This step was also carried out in~\cite{DKLN}. 

The next step is to improve the integrability of the gradient via an iteration scheme of Moser-type. Recall that for $p$-harmonic functions, one first derives proper second order estimates (i.e. energy estimate for $|\nabla u|$) and exploits the Sobolev embedding to acquire a reverse H\"older estimate:
\[
\bigg(\bint_{B_{1}}|\nabla u|^{\frac{Nq}{N-2}}\,\dx\bigg)^{\frac{N-2}{Nq}}\le C(q)  \bigg(\bint_{B_{2}}|\nabla u|^{q}\,\dx\bigg)^{\frac1q}
\]
provided $\nabla u\in L^{q}_{\loc}$ for some $q\ge p$. Then, starting from $\nabla u\in L^{p}_{\loc}$ one iteratively improves the integrability of $\nabla u$.  
 When it comes to $(s,p)$-harmonic functions, we exploit a similar idea at a discrete level, derive $\nabla u\in W^{\al, q}_{\loc}$ for some quantified $\al=\al(s,p,q)$ assuming that $\nabla u\in L^{q}_{\loc}$ for some $q\ge p$, and use the fractional Sobolev embedding to obtain a reverse H\"older estimate:
 \[
\bigg(\bint_{B_{1}}|\nabla u|^{\frac{Nq}{N-\al q}}\,\dx\bigg)^{\frac{N-\al q}{Nq}}\le C(q)  \bigg[\bigg(\bint_{B_{2}}|\nabla u|^{q}\,\dx\bigg)^{\frac1q} +\boldsymbol{\mathfrak{T}}(u;2)\bigg].
\]
The quantity $\boldsymbol{\mathfrak{T}}$, which encodes the nonlocal behavior of $u$, turns out to be harmless in the iteration scheme.  Whereas the most notable difference from the case of $p$-harmonic functions lies in the way the integral exponent increments. Arguably, such difference becomes magnified when one runs the iteration in a quantified manner.

The higher Sobolev regularity acquired in the previous step can be translated into quantitative H\"older estimates thanks to the Morrey embedding. We stress that, this way of obtaining higher H\"older regularity is different from previous works as all of them circumvented the higher Sobolev regularity. 

In the pursuit of the higher Sobolev regularity, we are also clued in on a gain of fractional differentiability for $\nabla u$. Conceivably, one should be able to formally recover the classical $W^{2,2}$-regularity of $p$-harmonic functions as $s\uparrow 1$. Indeed, we improve and establish that $\nabla u\in W^{\al,q}_{\loc}$ for any $q\ge2$ and $\al<\be:=\max\big\{\frac{sp}{q}, \frac{1-(1-s)p}{q-1}\big\}$.
This task is taken on in the final part of the paper. In particular, we exploit the Hölder continuity established in Theorem~\ref{thm:Hoelder-p<2} 
and perform a {\it discrete  integration by parts}; see Lemma \ref{lem:int-parts}. This is inspired by an argument from the $p$-harmonic functions; see \cite{Acerbi-Fusco}.

We emphasize that throughout the proofs we take great care on the structure of the constants in order to justify the stability as $s\uparrow 1$.  
Finally, we note that in~\cite{DKLN} the authors established $\nabla u\in W^{\alpha,p}_{\rm loc}(\Omega)$ for any $\alpha\in(0,\max\{\frac{sp}{2}, 1-(1-s)p\})$,  
which follows from Theorem~\ref{thm:grad-frac} as described in Remark~\ref{rem:diff-p}.

\medskip
 
\noindent
{\bf Acknowledgments.}  N.~Liao is supported by the FWF-project P36272-N \emph{On the Stefan type problems.} 
G.~Molica Bisci and ~R.~Servadei have
been funded by the European Union - NextGenerationEU within the framework
of PNRR  Mission 4 - Component 2 - Investment 1.1 under the Italian
Ministry of University and Research (MUR) program PRIN 2022 - grant number
2022BCFHN2 - Advanced theoretical aspects in PDEs and their applications -
CUP: H53D23001960006 and partially supported by the INdAM-GNAMPA Research
Project 2024: Aspetti geometrici e analitici di alcuni problemi locali e
non-locali in mancanza di compattezza - CUP E53C23001670001.

\section{Preliminaries}

\subsection{Notation}
Throughout the paper, $C$ stands for a generic constant that can change from line to line. In the statements and also in the proofs, we trace the dependencies of the constants with respect to the data.    We indicate the dependencies by writing, for example, $C=C(N,p,s)$ if $C$ depends on $N,p$ and $s$. Next, we denote $B_R(x_o)\Subset \R^N$ as a ball with radius $R$  centered at $x_o$ in $\R^N$. We also define
$$
    K_R(x_o):=B_R(x_o)\times B_R(x_o)\subset \R^N\times\R^N.
$$
We use this notation at various points to give the occurring double integrals a more compact form.

Let $\Om\subset\R^N$ be a bounded  open set and $\al\in(0,1)$. For a function $w\colon\Om\to \R$ we understand by $w\in C_{\loc}^{0,\al}(\Om)$ that for any ball $B_R(x_o)
\Subset \Om$ we have that $w\in C^{0,\al}\big(\overline{B_R(x_o)}\big)$ holds. 
Moreover, we denote the semi-norm
\[
[w]_{C^{0,\al}(B_R(x_o))}:=\sup_{x\not=y\in B_R(x_o)}\frac{|w(x)-w(y)|}{|x-y|^\al}.
\]
We also introduce the {\bf fractional Sobolev space} $W^{\gamma,q}(\Omega,\R^k)$, $k\in\N$, with some $q\in [1,\infty )$ and
$\gamma\in (0,1)$. A measurable function $w\colon\Omega\to\R^k$ belongs to the fractional Sobolev space $W^{\gamma, q}(\Omega,\R^k)$ if and only if
\begin{align*}
    \| w\|_{W^{\gamma ,q}(\Omega,\R^k)}
    &:=
    \| w\|_{L^q(\Omega,\R^k)} +
    [w]_{W^{\gamma ,q}(\Omega,\R^k)}
    <\infty,
\end{align*}
where the semi-norm is defined as 
\begin{align*}
    [w]_{W^{\gamma ,q}(\Omega,\R^k)}
    &:=
    \bigg[
    \iint_{\Omega\times\Omega}
    \frac{|w(x)-w(y)|^q}{|x-y|^{N+\gamma q}}\,\dx\dy
    \bigg]^\frac{1}{q} .
\end{align*}
If the dimension $k$ is clear from the context, we write $[w]_{W^{\gamma ,q}(\Omega)}$ instead of $[w]_{W^{\gamma ,q}(\Omega,\R^k)}$ for the sake of simplicity. In particular, if $w=\nabla v$ is the gradient of a scalar function $v\colon\Om\to \R$, we will write $[\nabla w]_{W^{\gamma ,q}(\Omega)}$ instead of 
$[\nabla w]_{W^{\gamma ,q}(\Omega,\R^N)}$.
Some useful results concerning fractional Sobolev spaces are collected in \S ~\ref{sec:fractional};
for more information we refer to \cite{Hitchhikers-guide}.

\begin{definition}[$(s,p)$-harmonic functions]\label{def:loc-sol}
Let $\Omega\subset\mathbb R^N$ be a bounded open set, $p\in(1,\infty)$ and $s\in(0,1)$. A function $u\in W^{s,p}_{\rm loc}(\Omega)$ is called $(s,p)$-harmonic in $\Omega$ if and only if
\begin{equation}\label{Lebesgue-weight}
\int_{\mathbb R^N}\frac{|u(x)|^{p-1}}{(1+|x|)^{N+sp}}\, \mathrm{d}x <\infty,
\end{equation}
and
\begin{equation}\label{weak-sol}
        \iint_{\mathbb R^N\times\mathbb R^N}\frac{|u(x)-u(y)|^{p-2}(u(x)-u(y))(\varphi(x)-\varphi(y))}{|x-y|^{N+sp}}\,\mathrm{d}x \mathrm{d}y =0
\end{equation}
for every $\varphi\in W^{s,p}(\Omega)$ compactly supported in $\Omega$ and extended to $0$ outside $\Omega$.
\end{definition}
Whenever $u$ satisfies \eqref{Lebesgue-weight}, we say $u$ belongs to the weighted Lebesgue space $L^{p-1}_{sp}(\R^N)$. For $x_o\in\Omega$ and $\rho>0$ such that $B_R(x_o)\subset\Omega$ use the nonlocal quantity 
\begin{equation*}
{\rm Tail}(u; x_o, R):= \bigg(R^{sp}\int_{\R^N\setminus B_R(x_o)}\frac{|u(x)|^{p-1}}{|x-x_o|^{N+sp}}\,\dx \bigg)^{\frac1{p-1}}.
\end{equation*}
Moreover, we abbreviate
\begin{equation}\label{def-T}
    \boldsymbol{\mathfrak T}(u; x_o, R)
    :=
    \|u\|_{L^{\infty}(B_{R}(x_o))}+\mathrm{Tail}(u; x_o, R).
\end{equation}
When $x_o=0$ or when the context is clear, we omit it from the notation.

\subsection{Algebraic inequalities}
For $\gamma\in(0,\infty)$ define
$$
    V_\gamma(a)
    :=
    |a|^{\gamma-1}a,
    \qquad\mbox{for $a\in\R$.}
$$
If $a=0$ we set $V_\gamma(a)=0$ also for $\gamma\in(0,1)$.
The first algebraic lemma  concerning bounds from above and below
for the difference of $V_\gamma$ can be retrieved from \cite[Lemma~2.1]{Acerbi-Fusco} for the case $\gamma\in(0,1)$, respectively in \cite[Lemma~2.2]{GiaquintaModica:1986} for the case $\gamma\in(1,\infty)$. 

\begin{lemma}\label{lem:Acerbi-Fusco}
For any $\gamma>0$, and for all $a,b\in\R$, we have
\begin{align*}
	C_1(|a| + |b|)^{\gamma-1}|b-a|
	\le
	|V_\gamma(b) - V_\gamma(a)|
	\le
	C_2(|a| + |b|)^{\gamma-1}|b-a|,
\end{align*}
where 
\begin{equation*}
    C_1
    =
    \left\{
    \begin{array}{ll}
        \gamma, & \mbox{if $\gamma\in(0,1)$,} \\[5pt]
        2^{1-\gamma}, & \mbox{if $\gamma\in(1,\infty)$,}
    \end{array}
    \right.
    \qquad
    C_2
    =
    \left\{
    \begin{array}{ll}
         2^{1-\gamma} & \mbox{if $\gamma\in(0,1]$} \\[5pt]
         \gamma & \mbox{if $\gamma\in[1,\infty)$.}
    \end{array}
    \right.
\end{equation*}
\end{lemma}

\begin{lemma}\label{lem:algebraic-1}
Let $p\in (1,2]$ and $\delta\ge1$. Then,  for any    $a,b,c,d\in \R$ and any $e,f\in\R_{\ge 0}$, we have
\begin{align*}
    &\big( V_{p-1}(a-b)- V_{p-1}(c-d)\big)
    \big( V_\delta (a-c)e^p -V_\delta (b-d)f^p\big)\\
    &\quad\ge 
    \tfrac{p-1}{2^{\delta+1}} 
    \big( |a-b|+|c-d|\big)^{p-2}
    \big( |a-c|+|b-d|\big)^{\delta-1}
    \big| (a-c)-(b-d)\big|^2 \big(e^p+f^p\big) \\
    &\qquad -
    \big(\tfrac{2^{\delta+1}}{p-1}\big)^{p-1}
    \big( |a-c| +|b-d|\big)^{p+\delta-1} |e-f|^p.
\end{align*}
\end{lemma}

\begin{proof}
We first rewrite the left-hand side of the desired inequality in the form
\begin{align*}
    &
    \tfrac12
    \big( V_{p-1}(a-b)- V_{p-1}(c-d)\big) \big( V_\delta (a-c)-V_\delta (b-d)\big)
    \big(e^p+f^p\big)\\
    &\qquad
    +\tfrac12
    \big( V_{p-1}(a-b)- V_{p-1}(c-d)\big) \big( V_\delta (a-c)+V_\delta (b-d)\big)
    \big(e^p-f^p\big)
    =:
    \mathbf I + \mathbf{II}.
\end{align*}
The first summand  is non-negative due to the monotonicity; cf.~\cite[Lemma A.5]{Brasco-Lindgren-Schikorra}. This means that the multiplicative factors $ V_{p-1}(a-b)- V_{p-1}(c-d)$ and $V_\delta (a-c)-V_\delta (b-d)$ must have the same sign. Without loss of generality, we assume that both are non-negative. Indeed, if they are both negative, we simply switch the summands in the two factors. To each factor we apply Lemma \ref{lem:Acerbi-Fusco} (first with exponent $p-1$ and then with exponent $\delta$) to conclude 
\begin{align}\label{alg-1}
    &(p-1) \big( |a-b|+|c-d|\big)^{p-2} | (a-c)-(b-d)| 
    \le   V_{p-1}(a-b)- V_{p-1}(c-d)\nonumber\\
    &\qquad 
    \le  2^{2-p} \big( |a-b|+|c-d|\big)^{p-2}| (a-c)-(b-d)|
\end{align}
and
\begin{align}\label{alg-2}
    &2^{1-\delta} \big( |a-c|+|b-d|\big)^{\delta-1}| (a-c)-(b-d)| 
    \le 
     V_{\delta}(a-c)- V_{\delta}(b-d)
     \nonumber\\
     &\qquad \le 
     \delta \big( |a-c|+|b-d|\big)^{\delta-1}| (a-c)-(b-d)|. 
\end{align}
This leads to the following lower bound for the  first term
\begin{align*}
   \mathbf I 
    &\ge
    \tfrac{p-1}{2^\delta}
    \big( |a-b|+|c-d|\big)^{p-2}
    \big( |a-c|+|b-d|\big)^{\delta-1}
    \big| (a-c)-(b-d)\big|^2 \big(e^p+f^p\big) \\
    &=:
    \tfrac{p-1}{2^\delta}\, \widetilde{\mathbf I}.
\end{align*}
In order to bound the second term from above, we use the upper bound for $V_{p-1}$ and obtain
\begin{align*}
    |\mathbf{II}|
    &\le
     2^{1-p} \big( |a-b|+|c-d|\big)^{p-2}| (a-c)-(b-d)|
     \big( |a-c|^\delta +|b-d|^\delta\big)|e^p-f^p|\\
     & \le
     \big( |a-b|+|c-d|\big)^{p-2}| (a-c)-(b-d)|
     \big( |a-c| +|b-d|\big)^\delta|e^p-f^p|.
\end{align*}
By Lemma \ref{lem:Acerbi-Fusco} applied with $\gamma =p$ we have
\begin{align*}
    |e^p-f^p|
    \le 
    p (e+f)^{p-1}|e-f|
    &\le 
    p (e^p+f^p)^{\frac{p-1}{p}}|e-f|.
\end{align*}
We use this inequality and rewrite the exponents appearing in the right-hand side of the preceding inequality and subsequently apply Young's inequality with exponents $\frac{p}{p-1}$ and $p$. In this way we obtain
\begin{align*}
    |\mathbf{II}|
    &\le 
    p\big( |a-b|+|c-d|\big)^{(p-2)(\frac{p-1}{p} +\frac1{p})}
    | (a-c)-(b-d)|^{2\frac{p-1}{p}+\frac{2-p}{p}}\\
    &\qquad\qquad\cdot 
    \big( |a-c| +|b-d|\big)^{(\delta -1)\frac{p-1}{p}+\frac{p+\delta-1}{p}}
    (e^p+f^p)^{\frac{p-1}{p}}|e-f| \\
    &\le 
    \eps\, \widetilde{\mathbf I} +
    \eps^{-(p-1)} \big( |a-b|+|c-d|\big)^{p-2}| (a-c)-(b-d)|^{2-p}\\
    &\qquad\qquad\qquad\cdot \big( |a-c| +|b-d|\big)^{p+\delta-1}|e-f|^p\\
    &\le 
    \eps\, \widetilde{\mathbf I} + 
    \eps^{-(p-1)}\big( |a-c| +|b-d|\big)^{p+\delta-1}|e-f|^p.
\end{align*}
Choosing $\epsilon=\tfrac{p-1}{2^{\delta+1}}$ we conclude the claimed inequality. 
\end{proof}

\begin{lemma}\label{lem:algebraic-1-2}
Let $p\in (1,2]$ and $\delta\ge1$. Then,  for any    $a,b,c,d\in \R$ and any $e,f\in\R_{\ge 0}$, we have
\begin{align*}
    \big( V_{p-1}&(a-b)- V_{p-1}(c-d)\big)
    \big( V_\delta (a-c)e^2 -V_\delta (b-d)f^2\big)\\
    &\ge 
    \tfrac{2(p-1)}{2^\delta} \big( |a-b|+|c-d|\big)^{p-2}
    \big( |a-c|+|b-d|\big)^{\delta-1}
    \big| (a-c)-(b-d)\big|^2 ef \\
    &\quad +
    \big( V_{p-1}(a-b)- V_{p-1}(c-d)\big) \big( V_\delta (a-c)e+V_\delta (b-d)f\big)
    (e-f)
\end{align*}
\end{lemma}

\begin{proof}
We re-write the left-hand side of the claimed inequality and subsequently use inequalities~\eqref{alg-1} and~\eqref{alg-2} from the proof of the preceding lemma 
\begin{align*}
    \big( V_{p-1}&(a-b)- V_{p-1}(c-d)\big)
    \big( V_\delta (a-c) -V_\delta (b-d)\big)ef\\
    &\quad +
    \big( V_{p-1}(a-b)- V_{p-1}(c-d)\big)
    \big( V_\delta (a-c)e +V_\delta (b-d)f\big)(e-f) \\
    &\ge 
    \tfrac{p-1}{2^{\delta-1}} \big( |a-b|+|c-d|\big)^{p-2}
    \big( |a-c|+|b-d|\big)^{\delta-1}
    \big| (a-c)-(b-d)\big|^2 ef \\
    &\quad +
    \big( V_{p-1}(a-b)- V_{p-1}(c-d)\big) \big( V_\delta (a-c)e+V_\delta (b-d)f\big)
    (e-f).
\end{align*}
This is the claimed inequality.
\end{proof}

\subsection{Some integral estimates}
The first result ensures that a certain integral exists.
\begin{lemma}\label{int-sing}
Let $0<\beta<N$.  Then, for any $B_R(x_o)\subset\R^N$ and any $x\in\R^N$ we have
\begin{equation*}
    \int_{B_R(x_o)}\frac{1}{|x-y|^{N-\beta}}\dy
    \le 
    \frac{N|B_1|}{\beta}R^\beta.
\end{equation*}
If $\beta\ge N$,  then for any $x\in B_R(x_o)$
 we have
\begin{align*}
        \int_{B_R(x_o)}\frac{1}{|x-y|^{N-\beta}}\dy
        \le
        \frac{N|B_1|}{\beta} (2R)^{\beta}.
\end{align*}
\end{lemma}

The following lemma can be inferred from \cite[Lemma 2.3]{Brasco-Lindgren-Schikorra}; for the precise statement see~\cite[Lemma~2.7]{BDLMS}.
\begin{lemma}\label{lem:t}
Let $p\in(1,\infty)$ and $s\in(0,1)$. Then, for any $u\in L^{p-1}_{sp}(\R^N)$, any ball $B_{R}\equiv B_{R}(x_o)$, and any $r\in(0,R)$, we have 
\begin{align*}
    \Tail (u;r)^{p-1}
    \le
    C(N)\Big(\frac{R}{r}\Big)^N
    \big( \Tail (u, R)+ \| u\|_{L^\infty (B_R)}\big)^{p-1}.
\end{align*}
\end{lemma}

\begin{remark}\label{rem:t}\upshape
Let the assumptions of Lemma~\ref{lem:t} be satisfied. Recalling the definition of $\boldsymbol{\mathfrak T}(\cdot)$ in \eqref{def-T}, Lemma~\ref{lem:t} ensures that 
$$
    \boldsymbol{\mathfrak T}(u;r)^{p-1}
    \le 
    C(N)\Big(\frac{R}{r}\Big)^N 
    \boldsymbol{\mathfrak T}(u;R)^{p-1}.
$$    
\end{remark}

\subsection{Fractional Sobolev spaces}\label{sec:fractional}
In the following we summarize some statements about fractional Sobolev spaces. We avoid introducing further functional spaces, such as Nikol'skii and Besov spaces.  Rather, we restrict ourselves to the main functional estimates. For more information on this topic, we refer to \cite{Adams, Brasco-Lindgren, Brasco-Lindgren-Schikorra}. We start with the \emph{embedding} $W^{1,q}\hookrightarrow W^{\gamma,q}$.
\begin{lemma}
\label{lem:FS-S}
Let $q\ge 1$ and $\gamma\in (0,1)$. Then  for any  $w\in W^{1,q}(B_R)$ we have
\begin{align*}
    (1-\gamma)\iint_{B_R\times B_R} 
    \frac{|w(x)-w(y)|^q}{|x-y|^{N+\gamma q}} \,\dx\dy
    \le 
    \frac{8\omega_N}{q}R^{(1-\gamma)q} 
    \int_{B_R} |\nabla w|^q \,\dx. 
\end{align*}
\end{lemma}
Next, we provide a \emph{fractional Sobolev-Poincar\'e inequality}, which can be retrieved from \cite[Theorem 6.7]{Hitchhikers-guide}.  To trace the stability with $\gamma\uparrow 1$, the exact dependence of $C$ with respect to $\gamma$ is crucial.  
The result can be found in \cite[Theorem 1]{BBM-1}.

\begin{lemma}\label{lem:frac-Sob-2}
Let $q\ge 1$, $\gamma\in (0,1)$, such that $\gamma q<N$.
Then, for any  $w\in W^{\gamma,q}(B_R)$ we have 
\begin{align*}
    \bigg[\mint_{B_R} & |w|^\frac{Nq}{N-q\gamma} \,\dx
    \bigg]^\frac{N-q\gamma}{Nq}
    \!\!\le
    2^{q-1}\bigg[ C^q R^{q\gamma}  \int_{B_R}\mint_{B_R}
    \frac{|w(x)-w(y)|^q}{|x-y|^{N+q\gamma}} \dx\dy +
    \mint_{B_R} |w|^q \dx\bigg]^\frac1q
\end{align*}
where  $C=C(N,q,\gamma)$. Moreover, for $\gamma\in [\frac12 ,1)$ we have
$$
    C(N,q,\gamma)= \bigg[\frac{C(N)(1-\gamma)}{(N-\gamma q)^{q-1}}\bigg]^\frac1q.
$$
\end{lemma}

\subsection{Finite differences and fractional Sobolev spaces}
For an open set $\Omega\subset \R^N$, and a  vector $h\in\R^N$, define $\Omega_{|h|}:=\{ x\in\Omega\colon {\rm dist}(x,\partial\Omega)>|h|\}$. By $\btau_h\colon  L^1(\Omega,\R^k)\to L^1(\Omega_{|h|},\R^k) $ we denote the finite difference operator
\begin{equation*}
    \btau_hw(x):=w(x+h)-w(x)\equiv w_h(x)-w(x)\quad\mbox{for $x\in \Omega_{|h|}$.}
\end{equation*}
If the direction is a fixed  unit  vector
$e\in\R^N$, we write 
$$
    \btau_h^{(e)}w(x):=w(x+he)-w(x),
$$ 
where $h\in\R$ now is a real number. At several points we will use two elementary properties of finite differences. These are summarized in \cite[Lemma 7.23 \& 7.24]{Gilbarg-Trudinger}
and \cite[Chap.~5.8.2]{Evans:book}. 

\begin{lemma}\label{lem:diff-quot-1}
Let $1<q<\infty$, $M>0$, and $0<d<R$.  Then, any $w\in L^q(B_R)$ that satisfies 
\begin{equation*}
    \int_{B_{R-d}} \big|\btau_h^{(e)}w\big|^q\,\dx
    \le
    M^q|h|^q \quad\mbox{for any $0<|h|\le d$}
\end{equation*}
is weakly differentiable in direction $e$ on $B_{R-d}$.  Moreover, we have
$$
    \int_{B_{R-d}} |D_ew|^q\,\dx\le M^q.
$$
\end{lemma}

\begin{lemma}\label{lem:diff-quot-2}
Let  $1\le q<\infty$ and $0<d<R$. Then, for any 
$w\in W^{1,q}(B_R)$, and any $0<|h|\le d$, we have
$$
    \big\| \btau_h^{(e)} w\big\|_{L^q(B_{R-d})} 
    \le 
    |h|\,\| D_e w\|_{L^q(B_R)}.
$$
\end{lemma}
Functions $w$ belonging to a fractional Sobolev space $W^{\gamma ,q}$ fulfill  an estimate for finite differences similar to the one from Lemma \ref{lem:diff-quot-2}; see \cite[Proposition\,2.6]{Brasco-Lindgren}.

\begin{lemma}\label{lem:N-FS}
Let $q\in(1,\infty)$, $\gamma \in(0,1)$, and $0<d<R$. Then, there exists a constant $C=C(N,q)$ such that for any $w\in W^{\gamma,q}(B_{R})$, we have
\begin{align*}
    \int_{B_{R-d}} |\btau_h w|^q \,\dx 
    &\le
    C\,|h|^{\gamma q} 
    \Bigg[(1-\gamma)[w]^q_{W^{\gamma,q}(B_{R})} +
    \bigg(\frac{R^{(1-\gamma )q}}{d^q} + \frac{1}{\gamma d^{\gamma q}}\bigg)\|w\|^q_{L^q(B_{R})}
    \Bigg]
\end{align*}
for any  $h\in\R^N\setminus\{0\}$ that satisfies  $|h|\le d$.
\end{lemma}
As in the case of Sobolev spaces, quantitative properties  of finite differences on the $L^q$-scale in terms of a power of the step size $h$ can be coverted  into fractional differentiability;  
cf.~\cite[7.73]{Adams}. The version given here is taken from \cite[Lemma 3.1]{DeFilippis-Mingione-Invent}.

\begin{lemma}\label{lem:FS-N}
Let $q\in [1,\infty)$, $\gamma \in(0,1]$, $M\in[ 1,\infty)$, and $d\in(0,R)$. Then, there exists a constant $C=C(N,q)$ such that whenever $w\in L^{q}(B_{R+d})$ satisfies 
\begin{align*}
    \int_{B_R} |\btau_h w|^q \,\dx 
    \le 
    M^q|h|^{\gamma q} 
    \qquad \mbox{for any $h\in\R^N\setminus\{0\}$ with $|h|\le d$,}
\end{align*}
then $w\in W^{\beta ,q}(B_R)$ whenever $\be\in(0,\gamma)$. Moreover, we have
\begin{align*}
      \iint_{B_R\times B_R}  
    \frac{|w(x)-w(y)|^q}{|x-y|^{N+\beta q}} \,\dx\dy
    \le 
    C\bigg[\frac{d^{(\gamma-\beta)q}}{\gamma-\beta} M^q +
    \frac{1}{\beta d^{\beta q}} \|w\|_{L^q(B_R)}^q\bigg].
\end{align*}
\end{lemma}

The following lemma proves useful estimates in a the fractional context when dealing with second-order finite differences.  Essentially, the proof can be found in \cite[Chapter 5]{Stein}; see also \cite[Proposition 2.4]{Brasco-Lindgren}.  An alternative proof  has been given in \cite[Lemma 2.2.1]{Domokos-1};see also \cite[Theorem 1.1]{Domokos-2}.
\begin{lemma} \label{lem:Domokos}
Let $q\in [1,\infty)$, $\gamma>0$, $M\ge 0$, $0<r<R$, and $0<d\le \tfrac1{2}(R-r)$. Then, there exists a constant $C=C(q)$ such that whenever  $w\in L^q(B_R)$ satisfies
\begin{align}\label{ass:Domokos}
    \int_{B_r} |\btau_h(\btau_h w)|^q \,\dx 
    \le 
    M^q|h|^{\gamma q},
    \qquad \mbox{for any $h\in\R^N\setminus\{0\}$ with $|h|\le d$,}
\end{align}
then in the  case $\gamma\in (0,1)$ we have for any $ 0<|h|\le \tfrac12 d$ that
\begin{align}\label{est-1st-diffquot<1}
     \int_{B_r} |\btau_hw|^q  \,\dx
     &\le
     C(q) |h|^{q\gamma}
     \Bigg[\Big( \frac{M}{1-\gamma}\Big)^q +\frac{1}{d^{q\gamma }}
     \int_{B_R}|w|^q\,\dx
     \Bigg]
     ,
\end{align}
while in the case $\gamma >1$ there holds
\begin{align}\label{est-1st-diffquot>1}  
     \int_{B_r} |\btau_hw|^q  \,\dx
     &\le
     C(q)|h|^q\bigg[
      \Big(\frac{M }{\gamma -1}\Big)^qd^{(\gamma-1)q}
      +
        \frac1{d^q}
        \int_{B_R}|w|^q\,\dx
      \bigg].
\end{align}
In the limiting case $\gamma =1$  we have for any $0<\beta<1$ that
\begin{align*}
     \int_{B_r} |\btau_hw|^q  \,\dx
     &\le
      C(q) |h|^{q\beta}
     \bigg[
     \Big(\frac{ M}{1-\beta}\Big)^q d^{(1-\beta)q}+\frac{1}{d^{\beta q}} \int_{B_R}|w|^q\,\dx\bigg].
\end{align*}
\end{lemma}

In the case $\gamma >1$, Lemma \ref{lem:Domokos} in combination with Lemma \ref{lem:diff-quot-1} guarantees that functions $w$ that fulfill \eqref{ass:Domokos}, 
are actually weakly differentiable. However, on the scale of oscillations of the gradient, a contribution of order $|h|^{\gamma -1}$ is lost. This part can be used to show that the gradient $\nabla w$ itself still has a certain fractional differentiability. The version presented here 
can be easily derived from 
\cite[Lemma 2.9]{Diening-Kim-Lee-Nowak}: see also \cite[Proposition 2.4]{Brasco-Lindgren}, \cite[Lemma 2.6]{Brasco-Lindgren-Schikorra}, and \cite{BDLMS}.

\begin{lemma}\label{lem:2nd-Ni-FS}
Let $q\in [1,\infty)$, $\gamma \in (0,1)$, $M>0$, $R>0$, and $d \in (0,R)$. Then, for any 
 $w\in W^{1,q}(B_{R+6d })$ that satisfies
\begin{equation}\label{ass:W^beta,q-second-diff}
    \int_{B_{R+4d }}\big|\btau_h(\btau_h w)\big|^q\,\dx
    \le
    M^q|h|^{q (1+\gamma )}\qquad \forall\, 0<|h|\in (0,d ],
\end{equation}
we have 
$$
\nabla u\in W^{\beta,q}(B_R)\quad \mbox{for any $\beta\in (0,\gamma )$.}
$$
Moreover, there  exists a constant $C$ depending only on $N$ and $q$,
such that 
\begin{align*}
     [\nabla w]_{W^{\beta,q}(B_R)}^q
    &\le
    \frac{Cd ^{q(\gamma -\beta)}}{(\gamma -\beta) \gamma ^q (1-\gamma )^q}
    \bigg[ M^q + \frac{(R+4d )^{q}}{\beta d ^{q(1+\gamma )}} 
    \int_{B_{R+4d }}|\nabla w|^q\,\dx
    \bigg].
\end{align*}
\end{lemma}
In the further course we will need an \emph{integration by parts formula for finite differences} at various points. This particular type of partial integration was implicitly used in \cite[Proof of Proposition 3.4]{Acerbi-Fusco} without providing the general formula.
\begin{lemma}\label{lem:int-parts}
Let  $E\subset \R^k$ be an open set, $F\in L^1(E)$, $\xi\in C^1_0(E)$, and $h\in \R^k$ with $0<|h|<\dist(\spt\xi ,\partial E)$. 
Then,  we have
\begin{equation*}
	\int_E \xi(x) \btau_h F(x) \,\dx
	=
	-\int_E\int_0^1 F(x+th)\,\dt\nabla\xi(x)\cdot h \dx.
\end{equation*}
\end{lemma}

\begin{remark}\label{rem:int-parts-0}\upshape
If $F\in L^q(E)$ for some  $q>1$, the claim of Lemma \ref{lem:int-parts} also holds  for any $\xi\in W^{1,q'}(E )$ with $\spt\xi\Subset E$, where $q'=\frac{q}{q-1}$denotes the H\"older conjugate of $q$. This can be shown by approximation.
\end{remark}

\begin{remark}\label{rem:int-parts}\upshape
We apply Lemma \ref{lem:int-parts} in the situation where $E=E_1\times E_2\subset \R^N\times \R^N$ and 
$F\in L^1(E_1\times E_2)$. Then, for any $\xi\in C^1_0(E_1\times E_2)$ and any  $h=(h_1,h_2)\in \R^N\times \R^N$ that satisfies
$0<|h|<\dist(\spt\xi ,\partial (E_1\times E_2))$ we have
\begin{align*}
    &\iint_{E_1\times E_2}\xi (x,y)\btau_{(h_1,h_2)}F (x,y)
    \dx\dy\\
    &\quad = -
    \iint_{E_1\times E_2}\int_0^1
    F(x+th_1,y+th_2)\dt 
    \big[ 
    \nabla_x\xi(x,y)\cdot h_1 + \nabla_y\xi(x,y)\cdot h_2
    \big]\dx\dy.
\end{align*}
\end{remark}

Finally, we recall the well known Morrey embedding for Sobolev functions \cite{Morrey}.
\begin{lemma}\label{Lem:morrey-classic}
Let $q\ge 1$ such that $q>N$. Then there exists a constant $C=C(N,q)$  such that for any $w\in W^{1, q}(B_R)$ we have
\begin{align*}
    [w]_{C^{0,1-\frac{N}{q}}(B_R)}\le C\|\nabla w\|_{L^{q}(B_R)}.
\end{align*}
\end{lemma}

\section{Energy inequality}

At various stages, it is appropriate to consider the corresponding localized variants $\eta\btau_hu$ respectively $\eta V_q(\btau_hu)$ with a cut-off function $\eta$ instead of $\btau_hu $ or $V_q(\btau_hu)$. To avoid having  constantly to explain the choice of
the cut-off function $\eta$, we fix it in advance. 

\begin{definition}\label{Def:Z}\upshape
Given $x_o\in \R^N$ and radii $0<r<R$, by $\mathfrak Z_{r,R}(x_o)$ we denote the class of functions $\eta\in C^1_0\big(B_{\frac12 (R+r)}(x_o),[0,1]\big)$ that satisfy $\eta=1$ in $B_r(x_o)$ and 
$$
\|\nabla\eta\|_{L^\infty(B_{\frac12 (R+r)}(x_o))} \le \frac{C}{R-r}.
$$
Similarly, $\widetilde{\mathfrak Z}_{r,R}(x_o)$ denotes the class of all cut-off functions 
$\eta\in C^2_0\big(B_{\frac12 (R+r)}(x_o),[0,1]\big)$ that satisfy $\eta=1$ in $B_r(x_o)$ and 
$$
    \|\nabla\eta\|_{L^\infty(B_{\frac12 (R+r)}(x_o))} \le \frac{C}{R-r}
    \quad\mbox{and}\quad
    \|\nabla^2\eta\|_{L^\infty(B_{\frac12 (R+r)}(x_o))} \le \frac{C}{(R-r)^2}.
$$
Here $C$ stands for a universal constant.
\end{definition}

Moreover, to simplify the notation, we will use for $u\colon\Omega\to\R$, $x,y\in\Omega$, and $h\in\R^N\setminus\{0\}$ such that $x+h,y+h\in\Omega$ the following abbreviations
\begin{equation}\label{def-U}
    U(x,y)=u(x)-u(y)\quad\mbox{and}\quad
    U_h(x,y)=u_h(x)-u_h(y)\equiv u(x+h)-u(y+h).
\end{equation}
Finally, we recall the definition of $\boldsymbol{\mathfrak T}(\cdot)$ from \eqref{def-T}.

\subsection{Auxiliary estimates}

As we will see, the tail estimates play an important role in the proof of the energy inequalities. In a sense, the exponents of $|h|$ and $|\btau_hu|$ that appear in the tail estimate determine the gain in fractional differentiability for $\btau_hu$.

\begin{lemma}[Tail estimate]\label{Lm:tail}
Let $p\in (1,2]$ and  $s\in(0,1)$. There exists a constant $C=C(N,p,s)$ such that whenever $u\in L^{p-1}_{sp}(\R^N)$, $x_o\in\R^N$, $R>0$, $r\in(0,R)$, and $d\in (0,\frac 14(R-r)]$, we have for any $x\in B_{\frac12 (R+r)}(x_o)$ and any $h\in\R^N\setminus\{0\}$ with $0<|h|\le d$ that
\begin{align*}
    \bigg|\int_{\R^N\setminus B_R(x_o)}&
    \frac{ V_{p-1}(U_h(x,y))- V_{p-1}(U(x,y)) }{|x-y|^{N+sp}}\,\dy \bigg|\\
    &\le 
    \frac{C}{R^{sp}} \Big(\frac{R}{R-r}\Big)^{N+sp+1}
    \bigg[|\btau_h u(x)|^{p-1} + 
    \frac{|h|}{R} 
    \boldsymbol{\mathfrak T}(u; x_o, R+d)^{p-1}\bigg].
\end{align*}
The constant $C$ has the form $C=\widetilde{C}(N,p)/s$.
\end{lemma}

\begin{proof}
Assume $x_o=0$ for simplicity. Like in \cite[Lemma~3.1]{BDLMS}, rewrite the left-hand side integral as
\begin{align*}
    &\int_{\R^N\setminus B_R(h)}
    \frac{V_{p-1}(u_h(x)-u(y))}{|x+h-y|^{N+sp}}\,\dy -
    \int_{\R^N\setminus B_R}
    \frac{V_{p-1}(u(x)-u(y))}{|x-y|^{N+sp}}\,\dy \\
    &\quad=
    \int_{\R^N\setminus (B_R(h)\cup B_R)}
    \bigg[\frac{V_{p-1}(u_h(x)-u(y))}{|x+h-y|^{N+sp}} -
    \frac{V_{p-1}(u(x)-u(y))}{|x-y|^{N+sp}} \bigg]\,\dy \\
    &\quad\phantom{=\,}
    +
    \int_{B_R\setminus B_R(h)}
    \frac{V_{p-1}(u_h(x)-u(y))}{|x+h-y|^{N+sp}}\,\dy -
    \int_{B_R(h)\setminus B_R}
    \frac{V_{p-1}(u(x)-u(y))}{|x-y|^{N+sp}}\,\dy.
\end{align*}
The proof hinges on estimating the absolute value of the three integral on the right, which we denote as $\mathbf{I}_j$, $j=1,2,3$. 

Estimating $\mathbf{I}_2$ and $\mathbf{I}_3$ is exactly the same as in \cite[Lemma~3.1]{BDLMS}. Namely,
\begin{align*}
    \mathbf{I}_2\le \frac{C |h| R^{N-1}}{(R-r)^{N+sp}}\|u\|_{L^{\infty}(B_R)}^{p-1}
\end{align*}
and 
\begin{align*}
    \mathbf{I}_3\le \frac{C |h| R^{N-1}}{(R-r)^{N+sp}}\|u\|_{L^{\infty}(B_{R+d})}^{p-1}
\end{align*}
for some $C=C(N,p)$.

To estimate $\mathbf{I}_1$, we proceed as follows. First, observe that since $|x|<\frac12(R+r)$ and $|y|>R$, for any $\xi\in B_{|h|}$ with $0<|h|\le d\le\frac14 (R-r)$, we have 
\begin{align}\label{est:xi}
    \frac{|x+\xi-y|}{|y|}\ge 1-\frac{|x|}{|y|}-\frac{|\xi|}{|y|}\ge1-\frac{R+r}{2R}-\frac{d}{R}\ge\frac{R-r}{4R}.
\end{align}
Then, the mean value theorem applied to the function $[0,1]\ni t\mapsto |x+t h-y|^{-N-sp}$ gives some $t\in [0,1]$, such that
\begin{align*}
    \bigg|\frac{1}{|x+h-y|^{N+sp}} -
    \frac{1}{|x-y|^{N+sp}} \bigg|
    &\le \frac{(N+sp)|h|}{|x+th-y|^{N+sp+1}}\\
    &\le
    |h| \Big(\frac{4R}{R-r}\Big)^{N+sp+1} \frac{N+p}{|y|^{N+sp+1}}.
\end{align*}
Here, we used \eqref{est:xi}, which is possible since $th\in B_{|h|}$. Now, we use the above observation together with Lemma~\ref{lem:Acerbi-Fusco} (choosing $b=u_h(x)-u(y)$, $a=u(x)-u(y)$, $\gamma =p-1\ge 1$, $C_2=p-1$) to estimate  the integrand of the first integral by
\begin{align}\label{est:diff-V_{p-1}}\nonumber
    \mathbf V_h&:=\bigg|\frac{V_{p-1}(u_h(x)-u(y))}{|x+h-y|^{N+sp}} -
    \frac{V_{p-1}(u(x)-u(y))}{|x-y|^{N+sp}} \bigg|\\\nonumber
    &
    \le
    \big|V_{p-1}(u_h(x)-u(y))\big| \bigg|\frac{1}{|x+h-y|^{N+sp}} -
    \frac{1}{|x-y|^{N+sp}} \bigg| \\\nonumber
    &\phantom{\le\,}
    +
    \frac{1}{|x-y|^{N+sp}}\big|V_{p-1}(u_h(x)-u(y)) -
    V_{p-1}(u(x)-u(y)) \big| \\\nonumber
    &
    \le 
   C|h| \Big(\frac{R}{R-r}\Big)^{N+sp+1} \frac{|u_h(x)-u(y)|^{p-1}}{|y|^{N+sp+1}}\\
    &\phantom{\le\,}
    +
    C\frac{|\btau_h u(x)|}{|x-y|^{N+sp}} \big(|u_h(x)-u(x)|+|u(x)-u(y)|\big)^{p-2}
\end{align}
for a constant $C=C(N,p)$. The second term on the right-hand side of \eqref{est:diff-V_{p-1}} is estimated with the aid of the elementary inequality 
\begin{align*}
    |\btau_h u(x)|\le |u_h(x)-u(y)|+ |u(x)-u(y)|;
\end{align*}
we obtain
\begin{align*}
    \mathbf V_h
    &
    \le C|h| \Big(\frac{R}{R-r}\Big)^{N+sp+1} \frac{|u_h(x)|^{p-1} + |u(y)|^{p-1}}{|y|^{N+sp+1}}
    + 
    C\Big(\frac{R}{R-r}\Big)^{N+sp}\frac{|\btau_h u(x)|^{p-1}}{|y|^{N+sp}}
\end{align*}
for some $C=C(N,p)$.
Plugging this into $\mathbf{I}_1$ we get
\begin{align*}
    \mathbf{I}_1 
    &\le  
    C|h| \Big(\frac{R}{R-r}\Big)^{N+sp+1} \int_{\R^N\setminus (B_R(h)\cup B_R)}\frac{|u_h(x)|^{p-1} + |u(y)|^{p-1}}{|y|^{N+sp+1}}\,\dy\\
    &\phantom{\le\,}
    + 
    C\Big(\frac{R}{R-r}\Big)^{N+sp}|\btau_h u(x)|^{p-1}\int_{\R^N\setminus (B_R(h)\cup B_R)}\frac{1}{|y|^{N+sp}}\,\dy\\
    &\le  
    C|h| \Big(\frac{R}{R-r}\Big)^{N+sp+1} \int_{\R^N\setminus B_R}\frac{\|u\|_{L^{\infty}(B_R)}^{p-1} + |u(y)|^{p-1}}{|y|^{N+sp+1}}\,\dy\\
    &\phantom{\le\,}
    + 
    C\Big(\frac{R}{R-r}\Big)^{N+sp}|\btau_h u(x)|^{p-1}\int_{\R^N\setminus B_R}\frac{1}{|y|^{N+sp}}\,\dy\\
    &\le 
    C\frac{|h|}{R^{sp+1}} \Big(\frac{R}{R-r}\Big)^{N+sp+1}\big[\|u\|_{L^{\infty}(B_R)}+\mathrm{Tail}(u; R)\big]^{p-1}\\
    &\phantom{\le\,}
    +
    \frac{C}{sp}\Big(\frac{R}{R-r}\Big)^{N+sp}
    \frac{|\btau_h u(x)|^{p-1}}{R^{sp}} \\
    &\le 
    C\frac{|h|}{R^{sp+1}} \Big(\frac{R}{R-r}\Big)^{N+sp+1}\boldsymbol{\mathfrak T}^{p-1} +
    \frac{C}{sp}\Big(\frac{R}{R-r}\Big)^{N+sp}
    \frac{|\btau_h u(x)|^{p-1}}{R^{sp}} .
\end{align*}
The last line followed from Remark~\ref{rem:t}.
Collecting all these estimates concludes the proof.
\end{proof}
When additional regularity is known, the tail estimate can be improved. 
The following lemma deals with the case $u\in W^{1+\theta,q}_{\rm loc}(\Omega)$. If $\theta$ is large enough, higher powers of the increment $h$ are obtained compared to Lemma~\ref{Lm:tail}. Due to the additional regularity requirement, a discrete  integration by parts can be carried out using Lemma \ref{lem:int-parts}. This is inspired by an argument from the local case of the singular $p$-Laplacian; see \cite{Acerbi-Fusco}.

\begin{lemma}[Higher order tail estimate]\label{Lm:tail-2}
Let $p\in (1,2]$,  $s\in(0,1)$, $q\in[p,\infty)$,   $\theta\in(0,1)$ and $\dl:=q-p+1$. There exists a constant $C=C(N,s,q)$ such that: whenever
$$
    u\in W^{1+\theta,q}_{\rm loc}(\Omega)\cap L^{p-1}_{sp}(\R^N)\cap L^\infty_{\rm loc}(\Omega),
$$
then for each $0<r<R$, each $d\in (0,\frac 18(R-r)]$,
each $B_{R+d}\equiv B_{R+d}(x_o)\Subset\Omega$,  each $\eta\in \mathfrak Z_{r,R}(x_o)$, and  each  $h\in\R^N\setminus\{0\}$ with $0<|h|\le d$, we have
\begin{align}\label{est:sec-tail-p<2}\nonumber
    \bigg|\iint_{B_{\frac12 (R+r)}\times (\R^N\setminus B_R)} & 
    \frac{\big[V_{p-1}(U_h(x,y))- V_{p-1}(U(x,y))\big]\big[V_{\dl} (\btau_hu)\eta^2\big](x)}{|x-y|^{N+sp}}\,\dx\dy \bigg| \\\nonumber
    &\!\!\!\!\!\!\!\!\!\!\!\!\!\!\!\!\!\!\!\!\le 
    \frac{C}{R^{sp}} \Big(\frac{|h|}{R}\Big)^{\dl+\theta}
    \Big(\frac{R}{R-r}\Big)^{N+sp+1}\\
     &\!\!\!\!\!\!\!\!\!\!\!\!\!\!\!\!\!\!\!\!\phantom{\le\,}\cdot
    \bigg[R^{(1+\theta) q}[\nabla u]^q_{W^{\theta,q}(B_{R})} +
    \frac{R^q}{\theta} \|\nabla u\|^q_{L^q(B_{R})} 
    +
    R^N \boldsymbol{\mathfrak T}(u; R+d)^q
    \bigg].
\end{align}
The constant $C$ has the form $C=\widetilde{C}(N,q)/s$. 
\end{lemma}

\begin{proof}
For $\epsilon\in(0,d)$ and $\rho>2R$ we choose a cut-off function $\psi_{\epsilon,\rho}
\in W_0^{1,\infty}([0,\infty),\R)$ such that $\psi_{\epsilon,\rho}\equiv 1$ in $[R+\epsilon,\rho]$, $\psi_{\epsilon,\rho}\equiv 0$ in $[0,R]\cup [2\rho ,\infty) $ and linearly interpolating otherwise. 
With the abbreviations 
$$
    F(x,y)
    :=
    V_{p-1}(U(x,y)) 
    \qquad\mbox{and}\qquad 
    G(x,y)
    :=
    \frac{[V_\delta (\btau_hu)\eta^2] (x)}{|x-y|^{N+sp}}
$$
the integral on the left-hand side of \eqref{est:sec-tail-p<2} can be re-written as
\begin{align*}
    \mathbf T
    &:=
    \iint_{B_{\frac12 (R+r)}\times (\R^N\setminus B_R)} 
    \btau_{(h,h)} F(x,y) G(x,y) \,\dx\dy  
    =
    \lim_{\rho\to\infty} \lim_{\epsilon\downarrow 0} \mathbf T_{\epsilon,\rho},
\end{align*}
where
\begin{align*}
    \mathbf T_{\epsilon,\rho}
    &:=
    \iint_{B_{\frac12 (R+r)}\times (\R^N\setminus B_R)} 
    \btau_{(h,h)} F(x,y) G(x,y) \psi_{\epsilon,\rho}(|y|)\,\dx\dy.
\end{align*}
An application of the integration by parts formula from Remark~\ref{rem:int-parts}  yields
\begin{align*}
    \mathbf T_{\epsilon,\rho}
    &=
    - \iint_{B_{\frac12 (R+r)}\times (\R^N\setminus B_R)} 
    \int_0^1 F(x+th,y+th) \,\dt\,
    \mathfrak H(x,y)  \,\dx\dy,
\end{align*}
where
\begin{align*}
       \mathfrak H(x,y)
       &=(h,h)\cdot \nabla_{(x,y)} \big(G(x,y)\psi_{\epsilon,\rho}(|y|)\big)\\
       &=
        h\cdot\nabla_xG(x,y)\psi_{\epsilon,\rho} (|y|)+
        h\cdot \nabla_y\big(G(x,y)\psi_{\epsilon,\rho}(|y|)\big)
        \\
        &=
         h\cdot \big[\nabla_xG(x,y)+  \nabla_yG(x,y)\big]\psi_{\epsilon,\rho}(|y|)
         +
         G(x,y)h\cdot \nabla_y\psi_{\epsilon,\rho}(|y|).
\end{align*}
For the term in $[\dots]$ of the last display, we compute 
\begin{align*}
    \nabla_x G(x,y) + \nabla_y G(x,y)
    &=
    \frac{\nabla_xV_\delta (\btau_hu(x))\eta^2 (x)}{|x-y|^{N+sp}} +
    \frac{V_\delta (\btau_hu(x)) \nabla_x\eta^2 (x)}{|x-y|^{N+sp}} \\
    &\quad +
    V_\delta (\btau_hu(x))\eta^2 (x)\underbrace{\Bigg[
    \nabla_x\frac{1}{|x-y|^{N+sp}}+ \nabla_y\frac{1}{|x-y|^{N+sp}}\bigg]}_{=0}
\\
    &=
    \frac{\nabla_xV_\delta (\btau_hu(x))\eta^2 (x)}{|x-y|^{N+sp}} +
    \frac{V_\delta (\btau_hu(x)) \nabla_x\eta^2 (x)}{|x-y|^{N+sp}} \\
    &
  \equiv
    \frac{\mathbf G(x)}{|x-y|^{N+sp}} ,
\end{align*}
with the obvious meaning of $\mathbf G(x)$. Inserting this into $\mathfrak H$ above, we get
\begin{align*}
      \mathfrak H(x,y)
      &=
      \frac{\big[\mathbf G(x)\psi_{\epsilon,\rho}(|y|) +
    V_\delta (\btau_hu(x))\eta^2 (x)\nabla\psi_{\epsilon,\rho}(|y|)\big]\cdot h}{|x-y|^{N+sp}}
    \equiv \frac{\mathfrak h(x,y)\cdot h}{|x-y|^{N+sp}},
\end{align*}
and 
\begin{align*}
    \mathbf T_{\epsilon,\rho}
    &=
    - \iint_{B_{\frac12 (R+r)}\times (\R^N\setminus B_R)} 
    \int_0^1 \frac{F(x+th,y+th)}{|x-y|^{N+sp}} \,\dt \,\mathfrak h(x,y)\cdot h\,\dx\dy .
\end{align*}
Next, in the integrand of $\mathbf T_{\epsilon,\rho}$ we replace $|x-y|$ with $|y+th|$ for any $t\in[0,1]$ and $0<|h|\le d\le\frac18 (R-r)$.  In fact, since $|x|<\frac12(R+r)$ and $|y|>R$, we have 
\begin{align*}
    \frac{|x-y|}{|y+th|}
    &\ge 
    1-\frac{|x+th|}{|y+th|}
    \ge
    1-\frac{|x|+|h|}{|y|-|h|}
    \ge
    1-\frac{\frac12(R+r)+d}{R-d}\\
   & =
    \frac{\frac12(R-r)-2d}{R-d} 
    \ge 
    \frac{\frac12(R-r)-2d}{R} 
    \ge 
    \frac{R-r}{4R}.
\end{align*}
Therefore, inserting this back to the previous identity of $\mathbf T_{\epsilon,\rho}$, taking absolute value on both sides and splitting the integral on the right-hand side into two terms according to the expression of $\mathfrak h(x,y)$, we arrive at
\begin{align*}
    |\mathbf T_{\epsilon,\rho}|
    &\le
    4^{N+sp}|h| \Big(\frac{R}{R-r}\Big)^{N+sp} 
    \big[\mathbf{I} + \mathbf{II}_{\epsilon,\rho}\big] ,
\end{align*}
where 
\begin{align*}
    \mathbf{I}
    &:=
    \int_{B_{\frac12 (R+r)}} \bigg[ 
    \int_{\R^N\setminus B_R} 
    \int_0^1 \frac{|F(x+th,y+th)|}{|y+th|^{N+sp}} \,\dt\dy \bigg] 
    \, |\mathbf G(x)| \,\dx, \\
    \mathbf{II}_{\epsilon,\rho}
    &:=
    \int_{B_{\frac12 (R+r)}}\bigg[  \int_{\R^N\setminus B_R} 
    \int_0^1 \frac{|F(x+th,y+th)|}{|y+th|^{N+sp}} |\nabla\psi_{\epsilon,\rho}(|y|)| \,\dt\dy \bigg]
    \,|\btau_hu(x)|^\delta \,\dx .
\end{align*}
To obtain the last line we used $|V_\delta (\btau_hu(x))|=|\btau_h u(x)|^\delta $ and the fact that $|\eta (x)|\le 1$ on $B_{\frac12 (R+r)}$.
Let us deal with these two terms separately. Fubini's theorem, the set inclusion $B_{R-|h|}\subset B_R(th)$, Lemma~\ref{lem:t}, and $R-|h|\ge \frac34 R$ allow us
to estimate the inner integral $[\dots ]$ of $\mathbf I$ by
\begin{align*}
    \int_{\R^N\setminus B_R} &
    \int_0^1 \frac{|F(x+th,y+th)|}{|y+th|^{N+sp}} \,\dt\dy  \\
    &\le
    \int_{\R^N\setminus B_R} 
    \int_0^1 \frac{|u(x+th)|^{p-1}+|u(y+th)|^{p-1}}{|y+th|^{N+sp}} \,\dt\dy\\
    &\le
     \int_0^1\int_{\R^N\setminus B_R} 
    \frac{\|u\|^{p-1}_{L^\infty(B_R)} + |u(y+th)|^{p-1}}{|y+th|^{N+sp}} \,
    \dy \dt\\
    &=
      \int_0^1\int_{\R^N\setminus B_R(th)} 
    \frac{\|u\|^{p-1}_{L^\infty(B_R)} + |u(z)|^{p-1}}{|z|^{N+sp}} \,\dz\dt\\
    &\le
    \int_{\R^N\setminus B_{R-|h|} }
    \frac{\|u\|^{p-1}_{L^\infty(B_R)} + |u(z)|^{p-1}}{|z|^{N+sp}} \,\dz\\
    &=
    \frac{N|B_1|}{sp}\frac{\|u\|^{p-1}_{L^\infty(B_R)}}{(R-|h|)^{sp}}+ \frac{\Tail(u; R-|h|)^{p-1}}{(R-|h|)^{sp}}\\
     &\le
     \frac{C(N)}{spR^{sp}}
     \boldsymbol{\mathfrak T}^{p-1}.
\end{align*}
with the abbreviation $\boldsymbol{\mathfrak T}:=\boldsymbol{\mathfrak T}(u; R+d)$. 
Hence, we have 
\begin{align*}
    |\mathbf I|
    &\le
    \frac{C(N)}{spR^{sp+1}}
    \boldsymbol{\mathfrak T}^{p-1}
    \int_{B_{\frac12 (R+r)}}   
    R|\mathbf G(x)| \,\dx .
\end{align*}
Arguing in a similar way as above and taking also into account the properties of $\psi_{\epsilon,\rho}$ we obtain for the inner integral $[\dots ]$ of $\mathbf{II}_{\epsilon,\rho}$ (note that $|y+th|\ge |y|-d\ge R-\frac18(R-r)\ge\frac12 R$ for $y\in\R^N\setminus B_R$ and $\epsilon<d$)
\begin{align*}
    &\int_{\R^N\setminus B_R} 
    \int_0^1 \frac{|F(x+th,y+th)|}{|y+th|^{N+sp}} |\nabla\psi_{\epsilon,\rho}(|y|)| \,\dt\dy  \\
    &\quad\le
    \int_0^1 \bigg[ 
    \frac{1}{\rho}\int_{B_{2\rho}\setminus B_\rho} 
    \frac{|F(x+th,y+th)|}{|y+th|^{N+sp}} \,\dy +
    \frac{1}{\epsilon} \int_{B_{R+\epsilon}\setminus B_R} 
    \frac{|F(x+th,y+th)|}{|y+th|^{N+sp}} \,\dy\bigg]\dt \\
    &\quad\le
     \frac{1}{\rho}
    \int_0^1
   \int_{\R^N\setminus B_R} 
    \frac{|u(x+th)|^{p-1} +|u(y+th)|^{p-1}}{|y+th|^{N+sp}} \,\dy\dt\\
    &\quad\phantom{\le\,}+
    \frac{1}{\epsilon}  \int_0^1\int_{B_{R+\epsilon}\setminus B_R} 
    \frac{|u(x+th)|^{p-1} +|u(y+th)|^{p-1}}{|y+th|^{N+sp}} \,\dy\dt\\
    &\quad\le
    \frac{1}{\rho} \int_{\R^N\setminus B_{\frac12R}} 
    \frac{\|u\|^{p-1}_{L^\infty(B_R)} + |u(y)|^{p-1}}{|y|^{N+sp}} \,\dy+C(N)\|u\|^{p-1}_{L^\infty(B_{R+d+\varep})} 
    \frac{|B_{R+\varep}\setminus B_R|}{\epsilon R^{N+sp}} \\
    &\quad\le
     \frac{C(N)}{\rho sp R^{sp}} \big( \|u\|_{L^\infty(B_R)} +\Tail(u; R)\big)^{p-1}
     +
     \frac{C(N)}{R^{sp+1}}
     \| u\|^{p-1}_{L^\infty (B_{R+d+\epsilon})}.
\end{align*}
The first term on the right-hand side  disappears as $\rho\to\infty$, while the second one converges to $C(N)R^{-sp-1} \| u\|^{p-1}_{L^\infty (B_{R+d})}$ as $\eps\downarrow 0$.  Therefore, we get
\begin{align*}
    \limsup_{\rho\to\infty} \limsup_{\epsilon\downarrow 0} 
    \mathbf{II}_{\epsilon,\rho}
    &\le 
    \frac{C(N)}{spR^{sp+1}}  \| u\|^{p-1}_{L^\infty (B_{R+d})}
    \int_{B_{\frac12 (R+r)}} |\btau_hu(x)|^\delta \,\dx .
\end{align*}
Combining these estimates, we end up with 
\begin{align*}
    |\mathbf T|
    &\le
    \frac{C(N)|h|}{spR^{sp+1}}
    \Big(\frac{R}{R-r}\Big)^{N+sp}
    \boldsymbol{\mathfrak T}^{p-1}
    \int_{B_{\frac12 (R+r)}}   
    \big[R|\mathbf G(x)| + |\btau_hu(x)|^\delta\big]\,\dx.
\end{align*}
It remains to compute $|\mathbf G(x)|$. In order for that, we use the product and chain rule for Sobolev functions to obtain
\begin{align*}
    |\mathbf G(x)|
    &\le
    |\nabla_xV_\delta (\btau_hu(x))|\eta^2 (x) +
    |\btau_hu(x)|^\delta |\nabla_x\eta^2 (x)| \\
    &\le  
    \delta|\btau_hu(x)|^{\delta-1} |\btau_h \nabla u(x)| + 
    \tfrac{C}{R-r}|\btau_hu(x)|^\delta.
\end{align*}
Inserting  this to the previous estimate gives
\begin{align*}
    |\mathbf T|
    &\le
    \frac{C(N)|h|}{spR^{sp+1}}
    \Big(\frac{R}{R-r}\Big)^{N+sp}
    \boldsymbol{\mathfrak T}^{p-1}
    \int_{B_{\frac12 (R+r)}}   
    \big[\delta R|\btau_hu|^{\delta-1} |\btau_h \nabla u| + 
    \tfrac{R}{R-r}|\btau_hu|^\delta\big]\,\dx.
\end{align*}
In order to estimate the above integral containing $|\btau_h \nabla u|$, we apply Lemma~\ref{lem:N-FS} with $(\gm,R,d)$ replaced by $(\theta,\frac12(R+r),  d=\frac12(R-r))$ and find that 
\begin{align*}
    \int_{B_{\frac12 (R+r)}} & 
    |\btau_h \nabla u|^q\,\dx\\
    &\le 
    C |h|^{\theta q} 
    \Bigg[(1-\theta)[\nabla u]^q_{W^{\theta,q}(B_{R})} +
    \bigg(\frac{R^{(1-\theta )q}}{ d^q} + \frac{1}{\theta  d^{\theta q}}\bigg)\|\nabla u\|^q_{L^q(B_{R})}
    \Bigg] \\
    &\le 
    C \Big(\frac{|h|}{R}\Big)^{\theta q} 
    \Big(\frac{R}{R-r}\Big)^{q}
    \bigg[R^{\theta q}[\nabla u]^q_{W^{\theta,q}(B_{R})} +
    \frac{1}{\theta} \|\nabla u\|^q_{L^q(B_{R})}
    \bigg],
\end{align*}
where $C=C(N,q)$. Therefore, we can estimate the integral appearing in the previous estimate of $|\mathbf T|$ by first applying Hölder's inequality, Lemma~\ref{lem:diff-quot-2} and then the last display. In this way, we obtain
\begin{align*}
    &\int_{B_{\frac12 (R+r)}} 
    \big[\delta R|\btau_hu|^{\delta-1} |\btau_h \nabla u| + |\btau_hu|^\delta\big]\,\dx \\
    &\quad\le 
    \delta R |B_R|^{1-\frac{\dl}{q}} 
    \bigg[\int_{B_{\frac12 (R+r)}}   
    |\btau_hu|^q \,\dx \bigg]^{\frac{\delta-1}{q}} 
    \bigg[ \int_{B_{\frac12 (R+r)}}|\btau_h \nabla u|^q \,\dx \bigg]^{\frac{1}{q}} \\
    &\quad\phantom{\le\,}+ 
    |B_R|^{1-\frac{\dl}{q}} 
    \bigg[\int_{B_{\frac12 (R+r)}} |\btau_hu|^q\,\dx \bigg]^{\frac{\delta}{q}} \\
    &\quad\le 
    \delta |B_R|^{1-\frac{\dl}{q}} R^{1-\theta} |h|^{\delta-1+\theta}
    \Big(\frac{R}{R-r}\Big)
    \bigg[\int_{B_{R}}   
    |\nabla u|^q \,\dx \bigg]^{\frac{\delta-1}{q}}
    \\
    &\quad\phantom{\le\,}\cdot
    \bigg[R^{\theta q}[\nabla u]^q_{W^{\theta,q}(B_{R})} +
    \frac{1}{\theta} \int_{B_{R}} |\nabla u|^q\,\dx \bigg]^{\frac{1}{q}}  + |B_R|^{1-\frac{\dl}{q}}
    |h|^\delta\bigg[\int_{B_{R}} |\nabla u|^q\,\dx \bigg]^{\frac{\delta}{q}} \\
    &\quad\le 
    q |B_R|^{1-\frac{\dl}{q}} \Big(\frac{|h|}{R}\Big)^{\delta-1+\theta}
    \Big(\frac{R}{R-r}\Big)
    \bigg[R^{(1+\theta) q}[\nabla u]^q_{W^{\theta,q}(B_{R})} +
    \frac{R^q}{\theta} \int_{B_{R}} |\nabla u|^q\,\dx
    \bigg]^{\frac{\delta}{q}} ,
\end{align*}
where from the second-to-last line we used $|h|^{\dl}\le R^{1-\theta}|h|^{\delta-1+\theta}$, since $\theta<1$. Substituting this back to the estimate of $|\mathbf T|$, recalling $\dl=q-p+1$, and further applying Young's inequality with exponents $\frac{q}{p-1}$ and $\frac{q}{q-p+1}$ we reach
\begin{align*}
    |\mathbf T|
    &\le
    \frac{C(N,q)}{spR^{sp}} \Big(\frac{|h|}{R}\Big)^{\delta+\theta}
    \Big(\frac{R}{R-r}\Big)^{N+sp+1}
    \big(|B_R|^{\frac{p-1}{q}} \boldsymbol{\mathfrak T}^{p-1}\big)\\
    &\phantom{\le\,}\cdot
    \bigg[R^{(1+\theta) q}[\nabla u]^q_{W^{\theta,q}(B_{R})} +
    \frac{R^q}{\theta} \int_{B_{R}} |\nabla u|^q\,\dx
    \bigg]^{\frac{\delta}{q}} \\
    &\le
    \frac{C(N,q)}{spR^{sp}} \Big(\frac{|h|}{R}\Big)^{\delta+\theta}
    \Big(\frac{R}{R-r}\Big)^{N+sp+1}\\
    &\phantom{\le\,}\cdot
    \bigg[R^{(1+\theta) q}[\nabla u]^q_{W^{\theta,q}(B_{R})} +
    \frac{R^q}{\theta} \int_{B_{R}} |\nabla u|^q\,\dx +
    R^N \boldsymbol{\mathfrak T}^q
    \bigg].
\end{align*}
This  proves the claimed inequality.
\end{proof}

The following lemma will be used to estimate the term containing the fractional derivative of the cut-off function in the improved energy estimate from Proposition~\ref{prop:energy-q-2}. 
The additional regularity assumption $u\in W^{1+\theta,q}_{\rm loc}(\Omega)$ allows us to get the same power of the increment $h$ as in Lemma~\ref{Lm:tail-2}.

\begin{lemma}\label{Lm:I-2}
Let $p\in (1,2]$,  $s\in(0,1)$, $q\in[p,\infty)$,  $\theta\in(0,1)$ and $\delta:=q-p+1$. 
There exists a constant $C=C(N,q)$ such that: whenever
$$
    u\in W^{1+\theta,q}_{\rm loc}(\Omega)\cap L^{p-1}_{sp}(\R^N)\cap L^\infty_{\rm loc}(\Omega),
$$
then for each $0<r<R$, each $d\in (0,\frac 18(R-r)]$,
each $B_{R+d}\equiv B_{R+d}(x_o)\Subset\Omega$,  each  $h\in\R^N$ with $0<|h|\le d$,  and each $C^2$-cut-off function $\eta\in\widetilde{\mathfrak Z}_{r,R}(x_o)$ we have 
\begin{align*}
    &\bigg|\iint_{K_R}
     \frac{\big[ V_{p-1}\big(U_h(x,y)\big)-V_{p-1}\big(U(x,y)\big)\big] 
     \big[V_{\dl}(\btau_hu)\eta\big](x) \big(\eta(x) - \eta(y)\big)}{|x-y|^{N+sp}} \,\dx\dy \bigg| \\
    &\quad\le 
    \frac{C}{R^{sp}} \Big(\frac{|h|}{R}\Big)^{\dl+\theta}
    \Big(\frac{R}{R-r}\Big)^{N+sp+1}\\
    &\quad\phantom{\le\,}\cdot
    \bigg[\frac{R^{(1+\theta)q}}{1-s} [\nabla u]^q_{W^{\theta,q}(B_{R+d})} +
    \frac{R^q}{\theta(1-s)} \|\nabla u\|^q_{L^q(B_{R+d})} +
    R^N \|u\|_{L^{\infty}(B_{R+d})}^q
    \bigg].
\end{align*}
\end{lemma}

\begin{proof} 
For $\epsilon\in(0,d)$ we choose a cut-off function $\psi_{\epsilon}
\in W^{1,\infty}([0,\infty),\R)$, such that $\psi_{\epsilon}\equiv 1$ in $[0,R-\epsilon]$, $\psi_{\epsilon}\equiv 0$ in $[R,\infty)$ and linearly interpolated otherwise. 
Moreover, we abbreviate 
$$
    F(x,y)
    :=
    V_{p-1}\big(U(x,y)\big) 
$$
and
$$
    G(x,y)
    :=
    [V_{\delta}(\btau_hu)\eta](x) \big(\eta(x) - \eta(y)\big).
$$
With this notation 
the integral on the left-hand side of the claimed inequality can be re-written in the form
\begin{align*}
    \mathbf L
    &:=
    \iint_{K_R} 
    \frac{\btau_{(h,h)} F(x,y) G(x,y)}{|x-y|^{N+sp}} \,\dx\dy  
    =
    \lim_{\epsilon\downarrow 0}  \mathbf L_{\epsilon},
\end{align*}
where
\begin{align*}
    \mathbf L_{\epsilon}
    &:=
    \iint_{K_{R}} 
    \frac{\btau_{(h,h)} F(x,y) G(x,y) \psi_{\epsilon}(|x|)\psi_{\epsilon}(|y|)}{|x-y|^{N+sp}}\,\dx\dy.
\end{align*}
An application of the integration by parts formula from Remark~\ref{rem:int-parts}  in connection with Remark~\ref{rem:int-parts-0} yields
\begin{align*}
    \mathbf L_{\epsilon}
    &=
    - \iint_{K_{R}} 
    \int_0^1 F(x+th,y+th) \,\dt\,
    \mathfrak H(x,y)  \,\dx\dy,
\end{align*}
where
\begin{align}\label{frak-H}
    \mathfrak H(x,y)
    &=(h,h)\cdot \nabla_{(x,y)}\bigg[ \frac{G(x,y)
    \psi_{\epsilon}(|x|)\psi_{\epsilon}(|y|)}{|x-y|^{N+sp}}\bigg].
\end{align}
 At this point, we  argue formally. In principle, we would have to include a second cut-off function in the argument, which truncates the singularity along the diagonal $K_R\cap\{x=y\}$; see Remark \ref{rigorous}. Since this only complicates the proof by a technical aspect, we outsource the argument 
and  proceed by computing the gradient 
on the right-hand side in \eqref{frak-H}. We have
\begin{align*}
    \mathfrak H (x,y)
    &=
    \frac{h\cdot \big[\nabla_xG(x,y) + \nabla_yG(x,y)\big]\psi_{\epsilon}(|x|)\psi_{\epsilon}(|y|)}{|x-y|^{N+sp}} \\
    &\quad +
    \frac{G(x,y) 
    h\cdot \big[\nabla_x\psi_{\epsilon}(|x|)\psi_{\epsilon}(|y|) +
    \psi_{\epsilon}(|x|)\nabla_y\psi_{\epsilon}(|y|) \big]}{|x-y|^{N+sp}} \\
    &\quad + 
    G(x,y)
    \psi_{\epsilon}(|x|)\psi_{\epsilon}(|y|) 
    \bigg[\underbrace{\nabla_x\frac{1}{|x-y|^{N+sp}} + \nabla_y\frac{1}{|x-y|^{N+sp}} }_{=0}\bigg].
\end{align*}
Among the three terms on the right, the third one vanishes due to symmetry in $x$ and $y$. Moreover, in the second term the product $G(x,y) h\cdot \nabla_x\psi_{\epsilon}(|x|)\psi_{\epsilon}(|y|)$ also vanishes due to the definition of $G(x,y)$, $\spt(\eta)\subset B_{\frac12(R+r)}$ and $\spt (\nabla_x\psi_\varep)\subset B_R\setminus B_{R-\varep}$. Applying a similar cancellation argument to the remaining part in the second term, we have that 
$$
    G(x,y) \psi_{\epsilon}(|x|)h\cdot \nabla_y\psi_{\epsilon}(|y|)
    =
    V_{\delta}(\btau_hu(x)) \eta^2(x) \psi_{\epsilon}(|x|)h\cdot \nabla_y\psi_{\epsilon}(|y|)
$$
Hence, using these simplifications and plugging the above expression of $\mathfrak H(x,y)$ in the definition of $\mathbf L_{\epsilon}$, we estimate
\begin{align}\label{est:L-eps}
    |\mathbf L_{\epsilon}|
    &\le |h|\big[\mathbf I + 
    \mathbf {II}_{\epsilon}\big] ,
\end{align}
where we introduced
\[
\mathbf I:=\iint_{K_{R}} 
    \int_0^1 \frac{|F(x+th,y+th)|}{|x-y|^{N+sp}} \,\dt \,|\mathbf G(x,y)|  \,\dx\dy
\]
with
\[
\mathbf G(x,y):=\nabla_x G(x,y) + \nabla_y G(x,y)
\]
and 
\[
\mathbf {II}_{\epsilon}:=\iint_{K_{R}} 
    \int_0^1 \frac{|F(x+th,y+th)|}{|x-y|^{N+sp}} \,\dt \,  |\btau_hu(x)|^\delta |\nabla_y\psi_{\epsilon}(|y|)| \,\dx\dy.
\]
We estimate $\mathbf{I}$ and $\mathbf{II}_{\varep}$ separately. 
First, by Hölder's inequality  with exponents $\frac{q}{p-1}$ and $\frac{q}{q-p+1}$ we have  
\begin{align}\label{est:I-FG}\nonumber
    |\mathbf I|
    &\le 
    \bigg[\iint_{K_{R}} 
    \int_0^1 \frac{|F(x+th,y+th)|^{\frac{q}{p-1}}}{|x-y|^{N+sq}} \,\dt \,\dx\dy \bigg]^{\frac{p-1}{q}} \\
    &\phantom{\le\,}\cdot
    \bigg[\iint_{K_{R}}
    \frac{|\mathbf G(x,y)|^{\frac{q}{q-p+1}}}{|x-y|^{N+s\frac{q}{q-p+1}}} \,\dx\dy \bigg]^{\frac{q-p+1}{q}}.
\end{align}
By a change of variable, the definition of $V_{p-1}$, and an application of Lemma~\ref{lem:FS-S}
the first integral of \eqref{est:I-FG} can be bounded by
\begin{align}\label{est:F}\nonumber
    \iint_{K_{R}} 
    \int_0^1 \frac{|F(x+th,y+th)|^{\frac{q}{p-1}}}{|x-y|^{N+sq}} \,\dt \,\dx\dy
    &=
    \int_0^1 \iint_{K_{R}(th)} 
    \frac{|u(x)-u(y)|^q}{|x-y|^{N+sq}} \,\dx\dy\,\dt \\\nonumber
    &\le 
    \iint_{K_{R+|h|}} 
    \frac{|u(x)-u(y)|^q}{|x-y|^{N+sq}} \,\dx\dy \\\nonumber
    &\le 
    C(N) \frac{(R+|h|)^{(1-s)q}}{(1-s)q} 
    \int_{B_{R+|h|}} |\nabla u|^q \,\dx \\
    &\le 
    \frac{C(N,q)}{R^{sq}}
    \frac{R^{q}}{1-s} 
    \int_{B_{R+d}} |\nabla u|^q \,\dx.
\end{align}
To further estimate the second integral in \eqref{est:I-FG}, we write $\mathbf G(x,y)$ more explicitly according to the definition of $G(x,y)$. Namely, we have
\begin{align*}
    \mathbf G(x,y)&=
    \nabla
    \big[V_\delta (\btau_hu)\eta\big]
    (x) \big(\eta (x) - \eta(y)\big) +
   \big[ V_\delta (\btau_hu)\eta\big](x) \big(\nabla\eta (x) - \nabla\eta (y)\big).
\end{align*}
Plug this in the second integral of \eqref{est:I-FG} and apply the standard trick of triangle's inequality to split the integral. This results in the first estimate in the following display.
To continue, recall that the cut-off function
$\eta\in \widetilde{\mathfrak Z}_{r,R}$ satisfies
$|\eta(x)-\eta(y)|\le \frac{C}{R-r}|x-y|$ and moreover $|\nabla_x\eta(x)-\nabla_y\eta(y)|\le \frac{C}{(R-r)^2}|x-y|$. This leads to the exponent  $N-(1-s)\frac{q}{q-p+1}$ of $|x-y|$ in both integrals of the second estimate in the following display. Finally, we apply the second part of Lemma~\ref{int-sing} to obtain the last estimate. That is,
\begin{align*}
    \iint_{K_{R}}&
    \frac{|\mathbf G(x,y)|^{\frac{q}{q-p+1}}}{|x-y|^{N+s\frac{q}{q-p+1}}} \,\dx\dy \\
    &\le 
    C\iint_{K_{R}}
    \frac{|\nabla [V_{q-p+1} (\btau_hu)\eta](x)|^{\frac{q}{q-p+1}}
    |\eta (x) - \eta(y)|^{\frac{q}{q-p+1}}}{|x-y|^{N+s\frac{q}{q-p+1 }}} \,\dx\dy \\
    &\phantom{\le\,} +
    C\iint_{K_{R}}
    \frac{|\btau_hu(x)|^{q} |\nabla\eta (x) - \nabla\eta (y)|^{\frac{q}{q-p+1}}}{|x-y|^{N+s\frac{q}{q-p+1}}} \,\dx\dy \\
    & \le 
    \frac{C}{(R-r)^{\frac{q}{q-p+1}}}
    \iint_{K_{R}}
    \frac{|\nabla[V_{q-p+1} (\btau_hu)\eta](x)|^{\frac{q}{q-p+1}}
    }{|x-y|^{N-(1-s)\frac{q}{q-p+1}}} \,\dx\dy \\
    &\phantom{\le\,}
    +
    \frac{C}{(R-r)^{\frac{2q}{q-p+1}}}
    \iint_{K_{R}}
    \frac{|\btau_hu(x)|^{q} }{|x-y|^{N-(1-s)\frac{q}{q-p+1}}} \,\dx\dy \\
    & \le 
    \frac{C }{(1-s)R^{s\frac{q}{q-p+1}}}
    \Big(\frac{R}{R-r}\Big)^{\frac{q}{q-p+1}}
    \int_{B_{R}}
    |\nabla[V_{q-p+1} (\btau_hu)\eta]|^{\frac{q}{q-p+1}}
    \,\dx \\
    &\phantom{\le\,} +
    \frac{C }{(1-s)R^{(1+s)\frac{q}{q-p+1}}}
    \Big(\frac{R}{R-r}\Big)^{\frac{2q}{q-p+1}}
    \int_{B_{R}}
    |\btau_hu|^{q} \,\dx,
\end{align*}
with $C=C(N,q)$. To further estimate the first integral on the right-hand side of the last display, we use the chain rule for Sololev functions (note that $u$ is locally bounded):
\begin{align*}
    |\nabla[V_{q-p+1} (\btau_hu)\eta]|
    &\le 
    |\nabla [V_{q-p+1} (\btau_hu)]|\eta +
    |\btau_hu|^{q-p+1} |\nabla\eta| \\
    &\le  
    (q-p+1)|\btau_hu|^{q-p} |\btau_h \nabla u| + 
    \tfrac{C}{R-r}|\btau_hu|^{q-p+1}.
\end{align*}
Substituting this back to the previous estimate, we
obtain
\begin{align}
\label{est:intermed}\nonumber
    \bigg[\iint_{K_{R}}&
    \frac{|\mathbf G(x,y)|^{\frac{q}{q-p+1}}}{|x-y|^{N+s\frac{q}{q-p+1}}} \,\dx\dy\bigg]^{\frac{q-p+1}{q}} \\\nonumber 
    & \le 
    \frac{C}{(1-s)^{\frac{q-p+1}{q}}R^{s}}
   \Big( \frac{R}{R-r} \Big)
    \bigg[
    \int_{B_{R}}
    |\btau_hu|^{\frac{(q-p)q}{q-p+1}} |\btau_h \nabla u|^{\frac{q}{q-p+1}}
    \,\dx \bigg]^{\frac{q-p+1}{q}} \\
    &\phantom{\le\,} +
    \frac{C }{(1-s)^{\frac{q-p+1}{q}}R^{1+s}}
    \Big(\frac{R}{R-r}\Big)^{2}
    \bigg[ \int_{B_{R}}
    |\btau_hu|^{q} \,\dx \bigg]^{\frac{q-p+1}{q}}.
\end{align}
The second integral in \eqref{est:intermed} is estimated by Lemma~\ref{lem:diff-quot-2} with $d=\frac12(R-r)$. Namely,
\begin{align*}
    \bigg[ \int_{B_{R}}
    |\btau_hu|^{q} \,\dx \bigg]^{\frac{q-p+1}{q}}&\le |h|^{q-p+1} \|\nabla u\|_{L^q(B_{R+d})}^{q-p+1}\\
    &=\Big(\frac{|h|}{R}\Big)^{q-p+1}
    \Big[R^q \|\nabla u\|^q_{L^q(B_{R+d})} \Big]^{\frac{q-p+1}{q}}.
\end{align*}
To deal with the first integral on the right-hand side of \eqref{est:intermed}, we use Hölder's inequality with exponents $\frac{q-p+1}{q-p}$ and $q-p+1$, followed by Lemma~\ref{lem:N-FS} and Lemma~\ref{lem:diff-quot-2} with $d=\frac12(R-r)$. In this way, we obtain 
\begin{align*}
    \bigg[\int_{B_{R}}&
    |\btau_hu|^{\frac{(q-p)q}{q-p+1}} |\btau_h \nabla u|^{\frac{q}{q-p+1}}
    \,\dx \bigg]^{\frac{q-p+1}{q}} \\
    &\le 
    \bigg[\int_{B_{R}}
    |\btau_hu|^{q} \,\dx \bigg]^{\frac{q-p}{q}}
    \bigg[\int_{B_{R}} 
    |\btau_h \nabla u|^q\,\dx \bigg]^{\frac{1}{q}} \\
    &\le 
    C |h|^{q-p} 
    \bigg[\int_{B_{R+d}}
    |\nabla u|^{q} \,\dx \bigg]^{\frac{q-p}{q}} \\
    &\phantom{\le\,}\cdot 
    |h|^{\theta}
    \Bigg[(1-\theta)[\nabla u]^q_{W^{\theta,q}(B_{R+d})} +
    \bigg(\frac{(R+d)^{(1-\theta )q}}{  d^q} + \frac{1}{\theta   d^{\theta q}}\bigg)\|\nabla u\|^q_{L^q(B_{R+d})}
    \Bigg]^{\frac{1}{q}} \\
    &\le 
    C |h|^{q-p} 
    \|\nabla u\|^{q-p}_{L^q(B_{R+d})}\\
    &\phantom{\le\,}\cdot 
    \Big(\frac{|h|}{R}\Big)^{\theta} 
    \Big(\frac{R}{R-r}\Big)
    \bigg[(1-\theta) R^{\theta q}[\nabla u]^q_{W^{\theta,q}(B_{R+d})} +
    \frac{1}{\theta} \|\nabla u\|^q_{L^q(B_{R+d})}
    \bigg]^{\frac{1}{q}} \\
    &\le 
    \frac{C}{R}
    \Big(\frac{|h|}{R}\Big)^{q-p+\theta}
    \!
    \frac{R}{R-r}
    \bigg[R^{q(1+\theta)}[\nabla u]^q_{W^{\theta,q}(B_{R+d})} +
    \frac{R^q}{\theta} \|\nabla u\|^q_{L^q(B_{R+d})}
    \bigg]^{\frac{q-p+1}{q}},
\end{align*}
where $C=C(N,q)$. Substituting these estimates to \eqref{est:intermed} yields
\begin{align}\label{est:G}\nonumber
    \bigg[\iint_{K_{R}}&
    \frac{|\mathbf G(x,y)|^{\frac{q}{q-p+1}}}{|x-y|^{N+s\frac{q}{q-p+1}}} \,\dx\dy\bigg]^{\frac{q-p+1}{q}} \\\nonumber
    &\le 
    \frac{C }{(1-s)^{\frac{q-p+1}{q}}R^{1+s}}  
    \Big(\frac{|h|}{R}\Big)^{q-p+\theta}
    \Big(\frac{R}{R-r}\Big)^2\\\nonumber
    &\phantom{\le\,}\cdot
    \bigg[R^{(1+\theta) q}[\nabla u]^q_{W^{\theta,q}(B_{R+d})} +
    \frac{R^q}{\theta} \|\nabla u\|^q_{L^q(B_{R+d})}
    \bigg]^{\frac{q-p+1}{q}} \\\nonumber
    &\phantom{\le\,}+
    \frac{C }{(1-s)^{\frac{q-p+1}{q}}R^{1+s}}
    \Big(\frac{|h|}{R}\Big)^{q-p+1}
    \Big(\frac{R}{R-r}\Big)^{2}
    \Big[R^q \|\nabla u\|^q_{L^q(B_{R+d})} \Big]^{\frac{q-p+1}{q}} \\\nonumber
    &\le 
    \frac{C }{(1-s)^{\frac{q-p+1}{q}}R^{1+s}}
    \Big(\frac{|h|}{R}\Big)^{q-p+\theta}
    \Big(\frac{R}{R-r}\Big)^{2}\\
    &\phantom{\le\,}\cdot
    \bigg[R^{(1+\theta) q}[\nabla u]^q_{W^{\theta,q}(B_{R+d})} +
    \frac{R^q}{\theta} \|\nabla u\|^q_{L^q(B_{R+d})}
    \bigg]^{\frac{q-p+1}{q}},
\end{align}
for a constant $C=C(N,q)$. 
Combining \eqref{est:F} and \eqref{est:G} in \eqref{est:I-FG} we obtain
\begin{align*}
    |\mathbf I|
    &\le 
    \frac{C}{(1-s)R^{sp+1}}
    \Big(\frac{|h|}{R}\Big)^{q-p+\theta}
    \Big(\frac{R}{R-r}\Big)^{2}\\
    &\phantom{\le\,}\cdot
    \bigg[R^{(1+\theta)q}[\nabla u]^q_{W^{\theta,q}(B_{R+d})} +
    \frac{R^{q}}{\theta}
    \|\nabla u\|^q_{L^q(B_{R+d})}
    \bigg] .
\end{align*}
This finishes the analysis of $\mathbf{ I}$.

Next, we turn our attention to the term $\mathbf {II}_{\epsilon}$. 
Using the definition of $F$, we have for the inner integral
\begin{align*}
    \int_0^1 |F(x+th,y+th)| \,\dt 
    &\le 
    \int_0^1 
    \big[|u(x+th)-u(y+th)|^{p-1}\big] \,\dt 
    \le 
    2 \|u\|^{p-1}_{L^\infty(B_{R+d})}
\end{align*}
for a.e.~$x,y\in B_R$, so that 
\begin{align*}
    |\mathbf {II}_{\epsilon}|
    &\le 
    2\|u\|^{p-1}_{L^\infty(B_{R+d})}
    \iint_{K_{R}}
    \frac{|\btau_hu(x)|^{q-p+1}|\nabla_y\psi_{\epsilon}(|y|)| }{|x-y|^{N+sp}}
    \,\dx\dy \\
    &\le 
    \frac{4}{\epsilon} \|u\|^{p-1}_{L^\infty(B_{R+d})}
    \iint_{B_{R}\times (B_R\setminus B_{R-\epsilon})} 
    \frac{|\btau_hu(x)|^{q-p+1}}{|x-y|^{N+sp}} 
    \,\dx\dy.
\end{align*}
To get  the  last line we used that $|\nabla_y\psi_{\epsilon}(|y|)|\le \frac1{\epsilon}$ on $B_R\setminus B_{R-\epsilon}$. 
In addition, for $x\in B_{\frac12(R+r)}$ and $y\in B_R\setminus B_{R-\epsilon}$ we have
$$
    |x-y|\ge R-\epsilon-\tfrac12(R+r)=\tfrac12(R-r)-\epsilon\ge \tfrac14(R-r).
$$ 
This allows us to estimate the kernel from above and obtain
\begin{align*}
    |\mathbf {II}_{\epsilon}|
    &\le 
    \frac{4^{N+sp+1}}{(R-r)^{N+sp}} \|u\|^{p-1}_{L^\infty(B_{R+d})} 
     \underbrace{\frac{|B_R\setminus B_{R-\epsilon}|}{\epsilon}}_{\le N|B_1|R^{N-1}}
    \int_{B_{R}} 
    |\btau_hu|^{q-p+1} 
    \,\dx \\
    &\le 
    \frac{C(N)}{R^{sp+1}} 
    \Big(\frac{R}{R-r}\Big)^{N+sp} 
    \|u\|^{p-1}_{L^\infty(B_{R+d})} 
    |h|^{q-p+1} \int_{B_{R+d}} 
    |\nabla u|^{q-p+1} 
    \,\dx \\
    &\le 
    \frac{C(N)}{R^{sp+1}} 
    \Big(\frac{|h|}{R}\Big)^{q-p+1}
    \Big(\frac{R}{R-r}\Big)^{N+sp} \\
    &\phantom{\le\,}\cdot
    \bigg[\frac{R^q}{1-s} \int_{B_{R+d}} |\nabla u|^q\,\dx +
    (1-s)R^N \|u\|_{L^{\infty}(B_{R+d})}^q
    \bigg].
\end{align*}
To obtain the second line we applied Lemma~\ref{lem:diff-quot-2} with $d=\frac12(R-r)$, whereas the last line follows from H\"older's inequality to raise the power of $|\nabla u|$ from $q-p+1$ to $q$ and then Young's inequality with exponents $\frac{q}{p-1}$ and $\frac{q}{q-p+1}$.
Inserting the estimates for $\mathbf I$ and $\mathbf{II}_\epsilon$ in \eqref{est:L-eps}, we end up with 
\begin{align*}
    |\mathbf L_{\epsilon}|
    &\le 
    \frac{C}{R^{sp}} \Big(\frac{|h|}{R}\Big)^{q-p+1+\theta}
    \Big(\frac{R}{R-r}\Big)^{N+sp+1}\\
    &\phantom{\le\,}\cdot
    \bigg[\frac{R^{(1+\theta )q}}{1-s} [\nabla u]^q_{W^{\theta,q}(B_{R+d})} +
    \frac{R^q}{\theta(1-s)} \|\nabla u\|^q_{L^q(B_{R+d})} +
    R^N \|u\|_{L^{\infty}(B_{R+d})}^q
    \bigg],
\end{align*}
where $C=C(N,q)$. This proves the claim.
\end{proof}

\begin{remark}\label{rigorous}\upshape
Instead of
$$
    \frac{G(x,y)\psi_{\epsilon}(|x|)\psi_{\epsilon}(|y|)}{|x-y|^{N+sp}},
$$
we  consider the function
$$
    \frac{G(x,y)\psi_{\epsilon}(|x|)\psi_{\epsilon}(|y|)}{|x-y|^{N+sp}}
\phi_\sigma (|x-y|)
$$ with a cut-off function $\phi_\sigma \in W^{1,\infty}([0,\infty),\R)$, which disappears on the interval  $ [0,\sigma]$, is identical 1 on $ [2\sigma,\infty)$ and is linearly interpolated in $[\sig,2\sig]$. Instead of $\mathbf L_\epsilon$ we would have $\mathbf L_{\epsilon,\sigma}$ with the corresponding integrand. For this integrand, the assumptions of Remark~\ref{rem:int-parts-0}  are fulfilled, as the singularity of the kernel has been cut off. After calculating the gradient $(h,h)\cdot \nabla_{(x,y)}$, this construction results in three terms in which the cut-off function $\phi_\sigma$ appears as a multiplicative factor and one further term in which $\phi_\sigma$ is differentiated. However, this term disappears as
$$
    (h,h)\cdot \nabla_{(x,y)}
    \phi_\sigma (|x-y|)
    =
    \tfrac1{\sigma}
    \boldsymbol{\chi}_{(\sigma,2\sigma)}(|x-y|)
    h\cdot \big[\nabla_x |x-y|+\nabla_y|x-y|\big]=0.
$$
In the following calculations, we  therefore only have to take into account a multiplicative factor $\phi_\sigma (|x-y|)\le 1$ in all integrals, which however does not create any problem, as all integrals exist.
\end{remark}

\subsection{Energy inequalities for $W^{1,q}$-solutions}

In this chapter we derive three versions of energy inequalities for $(s,p)$-harmonic functions. These differ from each other by regularity conditions assumed apriori.  For the first version, we assume that the solutions are already of class $W^{1,q}_{\rm loc}$ for some $q\ge p$. It provides a basic building block for the higher gradient regularity expanded on in Section~\ref{sec:W1q}. 
In the second energy inequality, we assume that $C^{0,\gamma}_{\rm loc}$-regularity in addition to $W^{1,q}_{\rm loc}$-regularity is available. This enables us to use large parts of the proof of the first energy inequality.  However, the assumption of H\"older continuity allows a more sophisticated coercivity estimate, increasing the integral exponent from $p$ to $2$. It will be used in Section~\ref{sec:fracdiff} to study the fractional differentiability of the gradient.
The third energy inequality refers to $(s,p)$-harmonic functions with $W^{1+\theta,q}_{\rm loc}$-regularity, where $\theta$ means the fractional differentiability of the gradient. 
In Section~\ref{sec:fracdiff} this will lead in the fractional context to the analogue of the $W^{2,2}_{\rm loc}$-regularity from the local case. 

\begin{proposition}[Energy inequality for $W^{1,q}$-solutions]\label{prop:energy-q-1}
Let $p\in (1,2]$, $s\in(0,1)$, $q\in[p,\infty)$.
There exists a constant $C=C(N,p,s, q)$
such that whenever
$u$ is a locally bounded $(s,p)$-harmonic function in the sense of Definition~\ref{def:loc-sol} that satisfies
$$
 u\in W^{1,q}_{\rm loc}(\Omega),
$$
then for every $0<r<R$, $d:=\frac14 (R-r)$,
$B_{R+d}\equiv B_{R+d}(x_o)\Subset\Omega$, every 
$\eps \in(0,1)$,
and  every  step size $0<|h|\le d$ we have 
\begin{align*}
    &\big[V_\frac{q}{p}(\btau_hu)
    \big]^p_{W^{(1-\eps)(1-\frac{p}{2})+s\frac{p}{2},p}(B_r)} \\
    &\qquad\le
    \frac{C}{R^{[(1-\eps)(1-\frac{p}{2})+s\frac{p}{2}]p}}
    \Big(\frac{|h|}{R}\Big)^{q-(1-\frac{p}{2})p}
    \Big(\frac{R}{R-r}\Big)^{N+sp+1} \\
    &\qquad\quad\cdot
    \bigg[ 
    \frac{R^q}{1-s} \int_{B_{R+d}} |\nabla u|^{q} \,\dx +
    R^{N} \boldsymbol{\mathfrak T}(u;R+d)^{q}\bigg]^\frac{p}{2} 
     \bigg[
     \frac{R^{q}}{\epsilon}
     \int_{B_{R}}|\nabla u|^{q}\,\dx
     \bigg]^{1-\frac{p}2}.
\end{align*}
The constant $C$ has the form 
$\widetilde C (N,p) q 4^{q}/s$.
\end{proposition}

\begin{proof}
Consider $x_o\in \Omega$, $0<r<R$ and $d=\tfrac14 (R-r)$ such that $B_{R+d}(x_o)\Subset\Omega$. 
Since $x_o$ is fixed we omit the reference to the center $x_o$ and write $B_\rho$ and $K_\rho$ instead of $B_\rho(x_o)$ and $K_\rho(x_o)$. 
Let $h\in \R^N\setminus\{ 0\}$ with $|h|\le d$. Testing \eqref{weak-sol} with $\varphi_{-h}(x):=\varphi(x-h)$ instead of $\varphi$, where $\varphi\in W^{s,p}(B_R)$ with $\spt\varphi\in B_{\frac12 (R+r)}$, we conclude by discrete integration by parts that also $u_h$ satisfies \eqref{weak-sol}. Subtracting \eqref{weak-sol} with $u$ from \eqref{weak-sol} with $u_h$, we obtain
\begin{align}\label{system-h}
    \iint_{\R^N\times\R^N}
    \frac{\big(V_{p-1}(u_h(x){-}u_h(y))- V_{p-1}(u(x){-}u(y))\big)(\varphi (x){-}\varphi(y)) }{|x-y|^{N+sp}}\,\dx\dy=0
\end{align}
for any $\varphi\in W^{s,p}(B_R)$ with $\spt\varphi\in B_{\frac12 (R+r)}$. In \eqref{system-h} we now choose
\begin{equation*}
    \varphi := V_\delta (\btau_hu)\eta^p,
\end{equation*}
with $\delta=q-p+1\ge 1$ and $\eta\in \mathfrak Z_{r,R}$.
Since $u$ is locally bounded, one can verify that $\varphi\in W^{s,p}(B_R)$.
Decomposing $\R^N$ into $B_R$ and its complement $\R^N\setminus B_R$ we obtain 
\begin{align*}
    0
    &=
     \iint_{K_R} \!\!
    \frac{\big(V_{p-1}(U_h(x,y))- V_{p-1}(U(x,y))\big)\big([V_{\delta} (\btau_hu)\eta^p] (x)-[V_{\delta} (\btau_hu)\eta^p](y)\big) }{|x-y|^{N+sp}}\,\dx\dy\\
    &\phantom{=\,}
    +
    \iint_{B_{\frac12 (R+r)}\times (\R^N\setminus B_R)}
    \frac{\big(V_{p-1}(U_h(x,y))- V_{p-1}(U(x,y))\big)[V_{\delta} (\btau_hu)\eta^p] (x) }{|x-y|^{N+sp}}\,\dx\dy\\
    &\phantom{=\,}
    -
    \iint_{(\R^N\setminus B_R)\times B_{\frac12 (R+r)}}
    \frac{\big(V_{p-1}(U_h(x,y))- V_{p-1}(U(x,y))\big)[V_{\delta} (\btau_hu)\eta^p] (y)}{|x-y|^{N+sp}}\,\dx\dy.
\end{align*}
Here, we used the abbreviations $U(x,y)$ and $U_h(x,y)$ from~\eqref{def-U}. By interchanging the roles of $x$ and $y$ in the second integral, it can be seen that it coincides with the last integral except for the sign. Therefore, we get 
\begin{equation}\label{I-2T}
      \mathbf I=-2\mathbf T,
\end{equation}
where
\begin{align*}
    \mathbf I
    :=
    \iint_{K_R} \!\!\!
    \frac{\big(V_{p-1}(U_h(x,y))- V_{p-1}(U(x,y))\big)\big([V_{\delta}(\btau_hu)\eta^p] (x)-[V_{\delta}(\btau_hu)\eta^p](y)\big) }{|x-y|^{N+sp}}\,\dx\dy,
\end{align*}
and
\begin{align*}
    \mathbf T
    :=
    \iint_{B_{\frac12 (R+r)}\times (\R^N\setminus B_R)}
    \frac{\big(V_{p-1}(U_h(x,y))- V_{p-1}(U(x,y))\big)[V_{\delta}(\btau_hu)\eta^p] (x)}{|x-y|^{N+sp}}\,\dx\dy.
\end{align*}
The {\bf local integral} $\mathbf I$, more precisely the integrand appearing in $\mathbf I$, is estimated from below using Lemma \ref{lem:algebraic-1} with $a=u_h(x)$, $b=u_h(y)$, $c=  u(x)$, $d=  u(y)$, $e=\eta(x)$, and  $f=\eta(y)$.  In fact, with a constant $C=\widetilde{C}(p)2^{\delta+1}\le\widetilde C(p)2^q$ we have
\begin{align*}
    \mathbf I
    &\ge \tfrac1C\mathbf I_1-C\mathbf I_2,
\end{align*}
where  $\mathbf I_1$ and $\mathbf I_2$ are defined by 
\begin{align*}
    \mathbf I_1
    &:=
    \iint_{K_R} \boldsymbol i_1(x,y)\,\dx\dy,
\end{align*}
and 
\begin{align*}
    \mathbf I_2
    &:=
    \iint_{K_R}
     \frac{ (|\btau_hu(x)|+|\btau_hu(y)|)^{q}| \eta(x)-\eta(y)|^p}{|x-y|^{N+sp}} \,\dx\dy.
\end{align*}
The  integrand $\boldsymbol i_1(x,y)$ is given by 
\begin{align}\label{def-i1}
    \boldsymbol i_1(x,y)
    &:=
     \frac{\mathcal U(x,y)^{p-2}(|\btau_hu(x)|+|\btau_hu(y)|)^{q-p} |\btau_hu(x)-\btau_hu(y)|^2\Theta (x,y)}{|x-y|^{N+sp}},
\end{align}
where
$$
    \mathcal U(x,y):= |U_h(x,y)|+|U(x,y)|\quad
    \mbox{and}\quad \Theta (x,y)=\eta^p(x)+\eta^p(y).
$$
Note that we may restrict the domain of integration in $\mathbf I_1$ to
$K_R^+:=K_R\cap\{ \mathcal U(x,y)>0\}$, since $ \mathcal U(x,y)=0$ implies that $U(x,y)=0=U_h(x,y)$. However, the latter means that $u(x)=u(y)$ and $u(x+h)=u(y+h)$, so that $|\btau_hu(x)- \btau_hu(y)|=0$.  
Joining the lower bound for $\mathbf I$ from above with the identity $\mathbf I=-2\mathbf T$, we find
$$
    \mathbf I_1
    \le 
    C\big[\mathbf I_2 + |\mathbf T|\big],
$$
where $C=\widetilde C(p)2^q$.
Next, we treat the integral $\mathbf I_2$. Using $|\eta (x)-\eta (y)|^p\le \frac{C(p)}{(R-r)^p}|x-y|^p$, and Lemma \ref{int-sing}, with $\beta =(1-s)p$, we have 
\begin{align*}
    \mathbf I_2
    &\le
    \frac{C(p)}{(R-r)^p}\iint_{K_R} \frac{(|\btau_hu(x)|+|\btau_hu(y)|)^{q}}{|x-y|^{N-(1-s)p}}\,\dx\dy\\
    &\le
     \frac{C}{R^{sp}}
    \Big(\frac{R}{R-r}\Big)^p
     \int_{B_R} \frac{|\btau_hu|^{q}}{1-s}\,\dx 
\end{align*}
for a constant $C=\widetilde C(N,p)2^q$.
Next, we deal with the non-local tail term $\mathbf{T}$ that we write in the form 
\begin{align*}
    \mathbf{T}
    =
    \int_{B_{\frac12 (R+r)}} \boldsymbol t(x) V_{q-p+1} (\btau_hu(x)) \eta^p(x) \,\dx,
\end{align*}
where 
\begin{align*}
    \boldsymbol t(x)
    &:=
    \int_{\R^N\setminus B_R}
    \frac{V_{p-1}\big(u_h(x)-u_h(y)\big)- V_{p-1}\big(u(x)-u(y)\big)}{|x-y|^{N+sp}}\,\dy ,
\end{align*}
for any $x\in B_{\frac12 (R+r)}$.
From Lemma \ref{Lm:tail} we have 
\begin{align*}
    |\boldsymbol t(x)|
    \le 
    \frac{C}{sR^{sp}} \Big(\frac{R}{R-r}\Big)^{N+sp+1}
    \bigg[ \frac{|h|}{R} 
    \boldsymbol{\mathfrak T}^{p-1}
    +
    |\btau_h u(x)|^{p-1} \bigg],
\end{align*}
for any $x\in B_{\frac12 (R+r)}$ and with $\boldsymbol{\mathfrak T}:=\boldsymbol{\mathfrak T}(u;x_o,R+d)$ and $C=C(N,p)$.
As a result of this estimate we obtain
\begin{align*}
    |\mathbf{T}|
    &\le
    \int_{B_{\frac12 (R+r)}} \boldsymbol |t(x)||\btau_hu(x)|^{q-p+1}\,\dx
    \\
    &\le
    \frac{C}{sR^{sp}} \Big(\frac{R}{R-r}\Big)^{N+sp+1}
    \Bigg[ \frac{|h|}{R} 
    \boldsymbol{\mathfrak T}^{p-1}\int_{B_{R}} |\btau_hu|^{q-p+1} \,\dx +
    \int_{B_{R}} |\btau_hu|^{q} \,\dx\Bigg] ,
\end{align*}
where $C=C(N,p)$. 
Now we estimate the integral of $|\btau_hu|^{q-p+1}$ such that it merges into the integral of $|\btau_hu|^{q}$ except for a term containing $\boldsymbol{\mathfrak T}$. This can be done by first applying H\"older's inequality and subsequently Young's inequality with the exponents $\frac{q}{p-1}$ and $\frac{q}{q-p+1}$. This way we obtain
\begin{align*}
    \frac{|h|}{R} \boldsymbol{\mathfrak T}^{p-1} \int_{B_{R}} |\btau_hu|^{q-p+1} \,\dx 
    &\le
    C\frac{|h|}{R} \boldsymbol{\mathfrak T}^{p-1} R^{N\frac{p-1}{q}}
    \bigg[\int_{B_{R}} |\btau_hu|^{q} \,\dx\bigg]^\frac{q-p+1}{q}\\
    &=
     C\frac{|h|}{R}    \big[ R^N  \boldsymbol{\mathfrak T}^{q}
     \big]^{\frac{p-1}{q}} 
    \bigg[\int_{B_{R}} |\btau_hu|^{q} \,\dx\bigg]^\frac{q-p+1}{q}\\
    &\le
    C
    \bigg[
    \int_{B_{R}} |\btau_hu|^{q} \,\dx
    +
    \Big(\frac{|h|}{R}\Big)^\frac{q}{p-1}
    R^N  \boldsymbol{\mathfrak T}^{q}
    \bigg] \\
    &\le 
    C
    \bigg[
    \int_{B_{R}} |\btau_hu|^{q} \,\dx
    +
    \Big(\frac{|h|}{R}\Big)^q R^N  \boldsymbol{\mathfrak T}^{q}
    \bigg],
\end{align*}
where $C=C(N)$. To obtain the last line we used 
$\frac{q}{p-1}\ge q$.
Therefore, we get
\begin{align*}
    |\mathbf T|
    &\le
     \frac{C}{sR^{sp}} \Big(\frac{R}{R-r}\Big)^{N+sp+1}
    \bigg[ 
    \int_{B_{R}} |\btau_hu|^{q} \,\dx
    +
    \Big(\frac{|h|}{R}\Big)^q R^N  \boldsymbol{\mathfrak T}^{q}
    \bigg].
\end{align*}
Now we substitute the inequalities for $\mathbf I_2$ and $|\mathbf T|$ into those for $\mathbf I_1$. We also take into account the special structure of the constants in the previous inequalities, as well as $p\le N+sp+1$ to homogenize the exponents of $\frac{R}{R-r}$. 
We obtain
\begin{align}\label{est-I1}
    \mathbf I_1
    &\le
    \frac{C}{sR^{sp}}
    \Big(\frac{R}{R-r}\Big)^{N+sp+1}
    \bigg[ 
    \int_{B_{R}} \frac{|\btau_hu|^{q}}{1-s} \,\dx +
    \Big(\frac{|h|}{R}\Big)^q R^{N} \boldsymbol{\mathfrak T}^{q}\bigg],
\end{align}
where $C=\widetilde C(N,p)4^q$.  

Our next aim is to deduce a lower bound for $\mathbf I_1$. 
Using Lemma \ref{lem:Acerbi-Fusco} with $\gamma=\frac{q}{p}\ge 1$, $a=\btau_hu(x)$, and $b=\btau_hu(y)$, we obtain (note that $C_2=\gamma= \frac{q}{p}$)
\begin{align}\label{est:J-q}\nonumber
      \mathbf J
      &:=
      \iint_{K_r}
     \frac{\big| V_\frac{q}{p}(\btau_hu(x))-V_\frac{q}{p}(\btau_hu(y))\big|^p}{|x-y|^{N+(1-\eps)(1-\frac{p}{2})p + sp\frac{p}{2}}}\,\dx\dy\\\nonumber
     &\ \le
     \tfrac{q}{p}
     \iint_{K_r} 
     \frac{(|\btau_hu(x)|+|\btau_hu(y)|)^{q-p} |\btau_hu(x)-\btau_hu(y)|^p}{|x-y|^{N+(1-\eps)(1-\frac{p}{2})p + sp\frac{p}{2}}} \,\dx\dy\\
     &\ =
     \tfrac{q}{p} 2^{-\frac{p}{2}}
     \iint_{K_r} \boldsymbol i_1^\frac{p}{2}(x,y)\cdot \boldsymbol j^{1-\frac{p}{2}}(x,y)\,\dx\dy.
\end{align}
To obtain the last line we   rewrote  the integrand as  the product of $\boldsymbol i_1^\frac{p}{2}\cdot \boldsymbol j^{1-\frac{p}{2}}$ with 
\begin{align*}
    \boldsymbol j(x,y)
    &:=
     \frac{\mathcal U(x,y)^p
     (|\btau_hu(x)|+|\btau_hu(y)|)^{q-p}}{|x-y|^{N+(1-\eps) p}}
\end{align*}
and $\boldsymbol{i}_1$ defined in~\eqref{def-i1}.
Note that $\Theta(x,y) =2$ for $(x,y)\in K_r$, since $\eta\equiv 1$ in $B_r$. 
We now distinguish between different cases. We start with the {\bf case $p<2$}.  We apply to the right-hand side of \eqref{est:J-q}  H\"older's inequality with the exponents $\frac2p$ and $\frac{2}{2-p}$ to obtain
\begin{align}\label{est:J-p<2-q}
     \mathbf J
    &\le 
    \tfrac{q}{p} \bigg[\iint_{K_r} \boldsymbol i_1(x,y) \,\dx\dy \bigg]^\frac{p}{2}
    \bigg[\iint_{K_r} \boldsymbol j(x,y) \,\dx\dy\bigg]^{1-\frac{p}2} \nonumber\\
    &\le 
    \tfrac{q}{p} \mathbf I_1^\frac{p}{2}
    \bigg[\iint_{K_r} \frac{\mathcal U(x,y)^p
     (|\btau_hu(x)|+|\btau_hu(y)|)^{q-p} }{|x-y|^{N+(1-\eps) p}}\,\dx\dy\bigg]^{1-\frac{p}2}.
\end{align}
If $q=p$, the second integral on the right-hand side of \eqref{est:J-p<2-q} is considerably more simple, because the second factor in the nominator, i.e.~the one with power $q-p$  reduces to 1.
Recalling the definitions of $\mathcal U$, $U$, and $U_h$ we have
\begin{align}\label{est:J-p<2-q=p}\nonumber
    \mathbf J
    &\le
      \mathbf I_1^\frac{p}{2}
    \bigg[\iint_{K_r} \frac{(|u_h(x)-u_h(y)|+|u(x)-u(y)|)^{p}
      }{|x-y|^{N+(1-\eps) p}}\,\dx\dy\bigg]^{1-\frac{p}2}\\
     &\le
     2^\frac{3}2 \mathbf I_1^\frac{p}{2}
     \bigg[
     \iint_{K_{r+d}}\frac{|u(x)-u(y)|^{p}}{|x-y|^{N+(1-\eps) p}}\,\dx\dy
     \bigg]^{1-\frac{p}2}.
\end{align}
If $q>p$ we apply to the right-hand side in \eqref{est:J-p<2-q} H\"older's inequality  with exponents $\frac{q}{p}$ and $\frac{q}{q-p}$. Taking also into account the definitions of $U_h$ and $U$, this yields 
\begin{align}\label{est:J-final-q}\nonumber
    \mathbf J
     &\le
     \tfrac{q}{p} \mathbf I_1^\frac{p}{2}
     \bigg[
     \iint_{K_r}\frac{(|u_h(x)-u_h(y)|+|u(x)-u(y)|)^{q}}{|x-y|^{N+q - \eps p}}\,\dx\dy
     \bigg]^{\frac{p}{q}(1-\frac{p}2)}\\\nonumber
     &\qquad\qquad\quad\cdot
     \bigg[ \iint_{K_r}\frac{(|\btau_hu(x)|+|\btau_hu(y)|)^{q}}{|x-y|^{N-\eps p}}\,\dx\dy
     \bigg]^{\frac{q-p}{q}(1-\frac{p}2)}\\\nonumber
     &\le
     \tfrac{q}{p} \mathbf I_1^\frac{p}{2}
     \bigg[ 2^{q}
     \iint_{K_{r+d}}\frac{|u(x)-u(y)|^{q}}{|x-y|^{N+q-\eps p}}\,\dx\dy
     \bigg]^{\frac{p}{q}(1-\frac{p}2)}\\\nonumber
     &\qquad\qquad\quad\cdot
     \bigg[ 2^{q}
     \iint_{K_r}\frac{|\btau_hu(x)|^{q}}{|x-y|^{N-\eps p}}\,\dx\dy
     \bigg]^{\frac{q-p}{q}(1-\frac{p}2)}\\
     &\le 
     C \mathbf I_1^\frac{p}{2}
     \bigg[
     \iint_{K_{r+d}}\frac{|u(x)-u(y)|^{q}}{|x-y|^{N+q + \eps p}}\,\dx\dy
     \bigg]^{\frac{p}{q}(1-\frac{p}2)}
     \bigg[\frac{R^{\eps p}}{\epsilon p}
     \int_{B_r}|\btau_hu|^{q}\,\dx
     \bigg]^{\frac{q-p}{q}(1-\frac{p}2)} ,
\end{align}
where $C\le 2q$. 
To obtain the last line we applied Lemma \ref{int-sing} with $\beta =\eps p$ to bound the second integral.
At this point, it must be noted that \eqref{est:J-p<2-q=p} can be  obtained from \eqref{est:J-final-q} by letting  $q \downarrow p$. This is possible, since  the constant in 
\eqref{est:J-final-q} is stable as $q\downarrow p$. 
In the {\bf case} $p=2$,  \eqref{est:J-q} reduces to 
$\mathbf J\le \tfrac12 q\mathbf I_1$
and no further estimation of $\mathbf I_1$ is necessary. We can therefore continue with the proof following \eqref{est:J-final-q}. Note also, \eqref{est:J-final-q}$_{p<2}$ reduces  to the $p=2$ estimate as  $p\uparrow  2$, regardless of whether $q=p$. Therefore, in any case we can continue with \eqref{est:J-final-q}.

Once  reached this point, we have all the estimates at hand to complete the proof of the energy estimate. Plugging~\eqref{est-I1} into the right-hand side of \eqref{est:J-final-q} we obtain
\begin{align}\label{est:J-final-q-}\nonumber
    \mathbf J
     &\le 
     \frac{C}{s R^{sp\frac{p}{2}}}
    \Big(\frac{R}{R-r}\Big)^{N+sp+1}
    \bigg[ 
    \int_{B_{R}} \frac{|\btau_hu|^{q}}{1-s} \,\dx +
    \Big(\frac{|h|}{R}\Big)^q R^{N} \boldsymbol{\mathfrak T}^{q}\bigg]^\frac{p}{2} \nonumber\\
    &\qquad\cdot 
     \bigg[
     \iint_{K_{r+d}}\frac{|u(x)-u(y)|^{q}}{|x-y|^{N+q + \eps p}}\,\dx\dy
     \bigg]^{\frac{p}{q}(1-\frac{p}2)}
     \bigg[\frac{R^{\eps p}}{\epsilon p}
     \int_{B_r}|\btau_hu|^{q}\,\dx
     \bigg]^{\frac{q-p}{q}(1-\frac{p}2)}  .
\end{align}
Finally, we apply Lemma~\ref{lem:FS-S} with $\gamma =1-\eps\frac{p}{q}$ to estimate the fractional norm of $u$ in terms of $\nabla u$ and the standard estimate for difference quotients from Lemma~\ref{lem:diff-quot-2}. In this way, we obtain
\begin{align*}
    \mathbf J
     &\le 
     \frac{C}{s R^{sp\frac{p}{2}}}
    \Big(\frac{R}{R-r}\Big)^{N+sp+1}
    \bigg[ 
    |h|^q\int_{B_{R+d}} \frac{|\nabla u|^{q}}{1-s} \,\dx +
    \Big(\frac{|h|}{R}\Big)^q R^{N} \boldsymbol{\mathfrak T}^{q}\bigg]^\frac{p}{2} \\
    &\qquad\cdot 
     \bigg[
     \frac{R^{\eps p}}{\epsilon p}
     \int_{B_{r+d}}|\nabla u|^{q}\,\dx
     \bigg]^{\frac{p}{q}(1-\frac{p}2)}
     \bigg[\frac{R^{\eps p}}{\epsilon p}
     |h|^q\int_{B_{r+d}}|\nabla u|^{q}\,\dx
     \bigg]^{\frac{q-p}{q}(1-\frac{p}2)} \\
     &=
     \frac{C}{sR^{[(1-\eps)(1-\frac{p}{2})+s\frac{p}{2}]p}}
    \Big(\frac{|h|}{R}\Big)^{q-(1-\frac{p}{2})p}
    \Big(\frac{R}{R-r}\Big)^{N+sp+1} \\
    &\qquad\cdot
    \bigg[ 
    \frac{R^q}{1-s} \int_{B_{R+d}} |\nabla u|^{q} \,\dx +
    R^{N} \boldsymbol{\mathfrak T}^{q}\bigg]^\frac{p}{2} 
     \bigg[
     \frac{R^{q}}{\epsilon}
     \int_{B_{R}}|\nabla u|^{q}\,\dx
     \bigg]^{1-\frac{p}2}.
\end{align*}
The constant $C$ has the structure
$
\widetilde C (N,p) q 4^{q}
$.
The integrals on the right-hand side are finite due to the assumption $u\in W^{1,q}_{\rm loc}(\Omega)$. 
\end{proof}

Under the additional assumption $u\in C_{\loc}^{0,\gamma}(\Omega)$, the gain in fractional differentiability in Proposition~\ref{prop:energy-q-1} can be improved. Since $p\in (1,2]$, the resulting fractional differentiability $\gamma -\frac12 (\gamma -s)p=\frac12 p s + (1-\frac12 p)\gamma$ is obviously a convex combination of $s$ and $\gamma$, and thus between $0$ and $1$.

\begin{corollary}\label{cor:energy-q-1}
 Let $p\in (1,2]$, $s\in(0,1)$, $q\in[p,\infty)$, and $\gamma\in(0,1)$.
There exists a constant $C= C(N,p,s,q)$
such that whenever
$u$ is a locally bounded $(s,p)$-harmonic function in the sense of Definition~\ref{def:loc-sol} that satisfies
$$
    u\in W^{1,q}_{\rm loc}(\Omega) \cap 
    C_{\loc}^{0,\gamma}(\Omega),
$$
then for every $0<r<R$, $d:=\frac14 (R-r)$,
$B_{R+d}\equiv B_{R+d}(x_o)\Subset\Omega$, and  every  step size $0<|h|\le d$ we have 
\begin{align*}
    &\big[
    V_\frac{q-p+2}{2}(\btau_hu)
    \big]^2_{W^{\gamma-\frac{1}{2}(\gamma-s)p,2}(B_r)} \\
    &\quad\le 
    \frac{C}{R^{sp}} \Big(\frac{|h|}{R}\Big)^{q}
    \Big(\frac{R}{R-r}\Big)^{N+sp+1}
    [u]_{C^{0,\gamma}(B_{R})}^{2-p} 
    \bigg[ 
    \frac{R^q}{1-s} \int_{B_{R+d}} |\nabla u|^{q} \,\dx +
     R^{N} \boldsymbol{\mathfrak T}^{q}\bigg],
\end{align*}
where $\boldsymbol{\mathfrak T}:=\boldsymbol{\mathfrak T}(u;x_o,R+d)$. 
Moreover, the constant $C$ has the form 
$\widetilde C (N,p) q^2 4^{q}/s$.
\end{corollary}

\begin{proof}
The starting point is inequality~\eqref{est-I1} from the proof of Proposition~\ref{prop:energy-q-1}. Taking into account that $\Theta(x,y)=2$ for $(x,y)\in K_r$ and that the integrand of $\mathbf I_1$ in~\eqref{est-I1} is non-negative, we have
\begin{align}\label{def:tilde-I-1}\nonumber
    \widetilde{\mathbf I}_1
    &:=
    \iint_{K_r} 
    \frac{\mathcal U(x,y)^{p-2}(|\btau_hu(x)|+|\btau_hu(y)|)^{q-p} |\btau_hu(x)-\btau_hu(y)|^2}{|x-y|^{N+sp}}\,\dx\dy \\\nonumber
    &\, \le
    \mathbf I_1
    \le 
    \frac{C}{sR^{sp}}
    \Big(\frac{R}{R-r}\Big)^{N+sp+1}
    \bigg[ 
    \int_{B_{R}} \frac{|\btau_hu|^{q}}{1-s} \,\dx +
    \Big(\frac{|h|}{R}\Big)^q R^{N} \boldsymbol{\mathfrak T}^{q}\bigg] \\
    &\, \le
    \frac{C}{sR^{sp}} \Big(\frac{|h|}{R}\Big)^q
    \Big(\frac{R}{R-r}\Big)^{N+sp+1}
    \bigg[ 
    \frac{R^q}{1-s} \int_{B_{R+d}} |\nabla u|^{q} \,\dx +
    R^{N} \boldsymbol{\mathfrak T}^{q}\bigg],
\end{align}
where $C=\widetilde C(N,p)4^q$ and 
$$
    \mathcal U(x,y)
    = 
    |U_h(x,y)|+|U(x,y)|
    =
    |u_h(x)-u_h(y)|+|u(x)-u(y)|.
$$
From the second to last line we used the standard estimate for difference quotients from Lemma~\ref{lem:diff-quot-2}.
As observed in the proof of Proposition~\ref{prop:energy-q-1}, we may restrict the domain of integration in $\widetilde{\mathbf I}_1$ to $K_r^+:=K_r\cap\{ \mathcal U(x,y)>0\}$.

To complete the proof it remains to bound $\widetilde{\mathbf I}_1$ from below. In view of the assumption $u\in C_{\loc}^{0,\gamma}(\Omega)$ we have 
\begin{align*}
    \mathcal U(x,y)^{p-2}
    &=
    \big(|u_h(x)-u_h(y)|+|u(x)-u(y)|\big)^{p-2} \nonumber\\
    &\ge 
    2^{p-2} [u]_{C^{0,\gamma}(B_{r+d})}^{p-2} |x-y|^{\gamma(p-2)}
\end{align*}
for any $(x,y)\in K_r^+$. Using this inequality and Lemma~\ref{lem:Acerbi-Fusco} applied with $\gamma=\frac{q-p+2}{p}\ge 1$, $a=\btau_hu(x)$, and $b=\btau_hu(y)$ (note that $C_2=\gamma= \frac{q-p+2}{p}$)  in the preceding estimate, we find
\begin{align}\label{lower-H}
    \widetilde{\mathbf I}_1
    &\ge 
    2^{p-2} [u]_{C^{0,\gamma}(B_{r+d})}^{p-2}
    \iint_{K_r} \frac{(|\btau_hu(x)|+|\btau_hu(y)|)^{q-p} |\btau_hu(x)-\btau_hu(y)|^2}{|x-y|^{N+sp+\gamma(2-p)}}\,\dx\dy \nonumber\\
    &\ge 
    \frac{2^p}{(q-p+2)^2} [u]_{C^{0,\gamma}(B_{R})}^{p-2}
    \iint_{K_r} \frac{\big| V_\frac{q-p+2}{2}(\btau_hu(x))-V_\frac{q-p+2}{2}(\btau_hu(y))\big|^2}{|x-y|^{N+sp+\gamma(2-p)}}\,\dx\dy \nonumber\\
    &\ge
    \frac{1}{2^{2-p}q^2} [u]_{C^{0,\gamma}(B_{R})}^{p-2}
    \big[
    V_\frac{q-p+2}{2}(\btau_hu)
    \big]^2_{W^{\gamma-\frac{1}{2}(\gamma-s)p,2}(B_r)}.
\end{align}
In combination with the upper bound for $\widetilde{\mathbf I}_1$, we conclude the proof.
\end{proof}

In contrast to Corollary \ref{cor:energy-q-1}, we now additionally assume $u\in W^{1+\theta,q}_{\rm loc}(\Omega)$. This allows us to apply the higher order tail estimate from Lemma \ref{Lm:tail-2} and Lemma \ref{Lm:I-2}. If $\theta$ is large enough, we end up with a larger exponent of the increment $h$ than in Corollary \ref{cor:energy-q-1}.

\begin{proposition}[Energy  inequality for $W^{1+\theta,q}$-solutions]\label{prop:energy-q-2}
 Let $p\in (1,2]$, $s\in(0,1)$, $q\in[p,\infty)$, and $\theta,\gamma\in(0,1)$.
Then there exists a constant $C=\widetilde C(N,p,q)/s$
such that whenever
$u$ is a locally bounded $(s,p)$-harmonic function in the sense of Definition~\ref{def:loc-sol} that satisfies
$$
    u\in W^{1+\theta,q}_{\rm loc}(\Omega) \cap 
    C_{\loc}^{0,\gamma}(\Omega),
$$
then for every $0<r<R$, $d:=\frac14 (R-r)$,
$B_{R+d}\equiv B_{R+d}(x_o)\Subset\Omega$, and  every  step size $0<|h|\le d$ we have 
\begin{align*}
    &\big[
    V_\frac{q-p+2}{2}(\btau_hu)
    \big]^2_{W^{\gamma-\frac{1}{2}(\gamma-s)p,2}(B_r)} \\
    &\qquad\le 
    \frac{C}{R^{sp}} \Big(\frac{|h|}{R}\Big)^{q-p+1+\theta}
    \Big(\frac{R}{R-r}\Big)^{N+sp+1}
    [u]_{C^{0,\gamma}(B_{R})}^{2-p} \\
    &\qquad\phantom{\le\,}\cdot
    \bigg[\frac{R^{(1+\theta)q}}{1-s} [\nabla u]^q_{W^{\theta,q}(B_{R+d})} +
    \frac{R^q}{\theta(1-s)} \int_{B_{R+d}} |\nabla u|^q\dx +
    R^N \boldsymbol{\mathfrak T}(u;R+d)^q
    \bigg].
\end{align*}
Note that the constant $C$ remains stable as $p\uparrow 2$ and blows up as $p\downarrow 1$. 
\end{proposition}

\begin{proof}
As in the proof of Proposition~\ref{prop:energy-q-1} we omit the reference to the center $x_o$. We consider $\eta\in \mathfrak Z_{r,R}$ and proceed exactly as in the proof of Proposition~\ref{prop:energy-q-1} up to identity~\eqref{I-2T}. The only difference consists in the fact that we replace $\eta^p$ by $\eta^2$ in the choice of the test function.
We also use the abbreviations
$U_h(x,y):= u_h(x)-u_h(y)$ and  $U(x,y):= U_0(x,y)$ and 
$\mathcal U(x,y):= |U_h(x,y)|+|U(x,y)|$. In this way, we obtain
\begin{equation*}
      \mathbf I=-2\mathbf T,
\end{equation*}
where $\mathbf I$ and $\mathbf T$ are defiend by
\begin{align*}
    \mathbf I
    &:=
    \iint_{K_R}\!\!
    \frac{\big(V_{p-1}(U_h(x,y)) {-} V_{p-1}(U(x,y))\big)\big([V_{\delta}(\btau_hu)\eta^2] (x) {-} [V_{\delta}(\btau_hu)\eta^2](y)\big) }{|x-y|^{N+sp}}\,\dx\dy,\\
    \mathbf T
    &:=
    \iint_{B_{\frac12 (R+r)}\times (\R^N\setminus B_R)}
    \frac{\big(V_{p-1}(U_h(x,y)) - V_{p-1}(U(x,y))\big)[V_{\delta}(\btau_hu)\eta^2](x)}{|x-y|^{N+sp}}\,\dx\dy,
\end{align*}
with $\delta=q-p+1\ge 1$. 
The {\bf local integral} $\mathbf I$ is estimated from below using Lemma~\ref{lem:algebraic-1-2} with $a=u_h(x)$, $b=u_h(y)$, $c=  u(x)$, $d=  u(y)$, $e=\eta(x)$, and  $f=\eta(y)$.  In fact, we have
\begin{align*}
    \mathbf I
    &\ge
    \tfrac{p-1}{2^{q-p}} \mathbf I_1 + \mathbf I_2,
\end{align*}
where  $\mathbf I_1$ and $\mathbf I_2$ are defined by 
\begin{align*}
    \mathbf I_1
    &:=
    \iint_{K_R} \frac{\mathcal U(x,y)^{p-2}(|\btau_hu(x)|+|\btau_hu(y)|)^{q-p} |\btau_hu(x)-\btau_hu(y)|^2\eta (x)\eta(y)}{|x-y|^{N+sp}}\,\dx\dy,
\end{align*}
and 
\begin{align*}
    \mathbf I_2
    &:=
    \iint_{K_R}
     \frac{ \big(V_{p-1}(U_h(x,y)) - V_{p-1}(U(x,y))\big) \big([V_{\delta}(\btau_hu)\eta](x) + [V_{\delta}(\btau_hu)\eta] (y)\big) }{|x-y|^{N+sp}}\\
    &\qquad\qquad\qquad\qquad
    \qquad\qquad\qquad
    \qquad
    \cdot  \big(\eta(x)-\eta(y)\big)\,\dx\dy.
\end{align*}
Joining the estimate from below for $\mathbf I$ with the identity $\mathbf I=-2\mathbf T$, we obtain
$$
     \mathbf I_1
     \le 
      \tfrac{2^{q-p}}{p-1} \big[|\mathbf I_2| +
      |\mathbf T| \big].
$$
To estimate $|\mathbf T|$ we use Lemma~\ref{Lm:tail-2}, while to estimate  $|\mathbf I_2|$ we use Lemma~\ref{Lm:I-2}. This implies
\begin{align*}
      \mathbf I_1
    &\le 
    \frac{C}{sR^{sp}} \Big(\frac{|h|}{R}\Big)^{q-p+1+\theta}
    \Big(\frac{R}{R-r}\Big)^{N+sp+1}\\
    &\phantom{\le\,}\cdot
    \bigg[\frac{R^{(1+\theta)q}}{1-s} [\nabla u]^q_{W^{\theta,q}(B_{R+d})} +
    \frac{R^q}{\theta(1-s)} \int_{B_{R+d}} |\nabla u|^q\,\dx +
    R^N \boldsymbol{\mathfrak T}^q
    \bigg],
\end{align*}
for a constant $C=C(N,p,q)$. 
At this point it remains to estimate $\mathbf I_1$ from below. Since $\eta=1$ in $B_r$, we may decrease the domain of integration to $K_r$ and omit $\eta$ in the integrand, so that $\mathbf I_1\ge \widetilde{\mathbf I}_1$, where $\widetilde{\mathbf I}_1$ is defined in \eqref{def:tilde-I-1}  from the proof of Corollary~\ref{cor:energy-q-1}. In view of the regularity assumption $u\in C_{\loc}^{0,\gamma}(\Omega)$, we subsequently may use inequality~\eqref{lower-H} from the proof of Corollary~\ref{cor:energy-q-1} to obtain 
\begin{align*}
    \mathbf I_1
    &\ge 
    \widetilde{\mathbf I}_1 
    \ge 
    \frac{1}{2^{2-p}q^2} [u]_{C^{0,\gamma}(B_{r+h})}^{p-2}
    \iint_{K_r} \frac{\big| V_\frac{q-p+2}{2}(\btau_hu(x))-V_\frac{q-p+2}{2}(\btau_hu(y))\big|^2}{|x-y|^{N+sp+\gamma(2-p)}}\,\dx\dy \\
    &=
    \frac{1}{2^{2-p}q^2} [u]_{C^{0,\gamma}(B_{R})}^{p-2}
    \big[
    V_\frac{q-p+2}{2}(\btau_hu)
    \big]^2_{W^{\gamma-\frac{1}{2}(\gamma-s)p,2}(B_r)}.
\end{align*}
Together with the upper bound for $\mathbf I_1$ this proves the energy inequality.
\end{proof}

\subsection{Energy inequality for $W^{s+\theta,p}$-solutions}

In this subsection we prove the final energy inequality. It can be seen as a variant of Proposition~\ref{prop:energy-q-1} for the case $q=p$ and in a situation where the existence of a weak derivative is not apriori assumed. In fact, this energy inequality will be used in Section~\ref{sec:W^1p} to show that $(s,p)$-harmonic functions have a weak gradient in 
$L^{p}_{\rm loc}$.
This is achieved by improving the fractional differentiability step by step in an iteration process.

\begin{proposition}[Energy inequality for $W^{s+\theta,p}$-solutions]\label{prop:energy-p}
Let $p\in (1,2]$, $s\in(0,1)$, and $\theta\in [0,1-s)$.
There exist a constant $C=\widetilde C(N,p)/s$
such that whenever
$u$ is a locally bounded $(s,p)$-harmonic function in the sense of Definition~\ref{def:loc-sol} that satisfies
$$
 u\in W^{s+\theta,p }_{\rm loc}(\Omega),
$$
then for every $0<r<R$, $d:=\frac14 (R-r)$,
$B_{R+d}\equiv B_{R+d}(x_o)\Subset\Omega$,  and  every  step size $0<|h|\le d$ we have 
\begin{align*}
    [\btau_hu]^p_{W^{s+\theta (1-\frac{p}{2}),p}(B_r)}
    &\le
    \frac{C}{R^{sp\frac{p}{2}}} 
     \Big(\frac{R}{R-r}\Big)^{N+sp+1} 
    [u]_{W^{s+\theta, p}(B_{R})}^{p(1-\frac{p}2)}\\
     &\qquad\qquad \cdot
     \Bigg[
    \int_{B_{R}} \frac{| \btau_hu|^{p}}{1-s} \dx
    +
    \Big(\frac{|h|}{R}\Big)^pR^N \boldsymbol{\mathfrak T}(u;R+d)^{p}
    \Bigg]^{\frac{p}2} .
\end{align*}
\end{proposition}

\begin{proof}
The starting point is the inequality~\eqref{est:J-final-q-} from the proof of Proposition~\ref{prop:energy-q-1}, applied with $q=p$ and $\epsilon=\theta-(1-s)\in(0,1)$. Note that the assumption $u\in W^{1,q}_{\rm loc}(\Omega)$ in Proposition~\ref{prop:energy-q-1} has not been used up to this point. The inequality takes the form
\begin{align*}
     \mathbf J
     &:=
    \iint_{K_r}
     \frac{| \btau_hu(x)-\btau_hu(y)|^p}{|x-y|^{N+[s+\theta(1-\frac{p}{2})]p}}\,\dx\dy 
     \\
     &\, \le 
     \frac{C }{s R^{sp\frac{p}{2}}}
    \Big(\frac{R}{R-r}\Big)^{N+sp+1}
    \bigg[
     \iint_{K_{r+d}}\frac{|u(x)-u(y)|^{p}}{|x-y|^{N+(s+\theta)p}}\,\dx\dy
     \bigg]^{1-\frac{p}2} \\
    &\qquad\cdot 
     \bigg[ 
    \int_{B_{R}} \frac{|\btau_hu|^{p}}{1-s} \,\dx +
    \Big(\frac{|h|}{R}\Big)^p R^{N} \boldsymbol{\mathfrak T}^{p}\bigg]^{\frac{p}{2}} \\
    &\, \le 
    \frac{C }{s R^{sp\frac{p}{2}}}
    \Big(\frac{R}{R-r}\Big)^{N+sp+1}
    [u]_{W^{s+\theta,p}(B_{R})}^{p(1-\frac{p}2)} 
     \bigg[ 
    \int_{B_{R}} \frac{|\btau_hu|^{p}}{1-s} \,\dx +
    \Big(\frac{|h|}{R}\Big)^p R^{N} \boldsymbol{\mathfrak T}^{p}\bigg]^{\frac{p}{2}} ,
\end{align*}
where  $C=(N,p)$. This proves the claimed inequality.
\end{proof}

\section{$W^{1,p}$-estimates}\label{sec:W^1p}
In this section we prove that $(s,p)$-harmonic functions possess a gradient $\nabla u$ in $L^p_{\rm loc}(\Omega ,\R^N)$. The main ingredient is the energy inequality from Proposition~\ref{prop:energy-p}, which is applied in the next lemma in order to obtain an estimate of the second finite difference of $u$ in $L^p$.

\begin{lemma}\label{lem:Nikol-est-p<2}
Let $p\in (1,2]$, $s\in(0,1)$, and $\theta\in [0,1-s)$. Then, there exists a constant $C=\widetilde C(N,p)/s$ such that whenever $u$ is a locally bounded $(s,p)$-harmonic function in the sense of Definition~\ref{def:loc-sol}, which satisfies
$$
    u\in W^{s+\theta, p}_{\rm loc}(\Omega),
$$
 we have
\begin{align*}
      \int_{B_r}|\btau_h(\btau_h u)|^p\,\dx 
      &\le
      C\Big(\frac{|h|}{R}\Big)^{(s+\theta (1-\frac{p}2))p}
       \Big(\frac{R}{R-r}\Big)^{N+sp+1} \\
       &\qquad\cdot  
       \Big[R^{(s+\theta)p}(1-s)[u]_{W^{s+\theta,p}(B_R)}^p +
       R^{N} \|u\|_{L^\infty(B_R)}^p\Big]^{1-\frac{p}2} \\
       &\qquad \cdot
    \bigg[
    \int_{B_{R}} | \btau_hu|^{p} \dx
    +
    \Big(\frac{|h|}{R}\Big)^p R^{N}
    \boldsymbol{\mathfrak T}(u;R)^p
    \bigg]^{\frac{p}2} 
\end{align*}
for any two concentric balls $B_r\subset B_R\Subset \Omega$, $0<r<R$,  and any  $0<|h|\le d:=\frac17 (R-r)$.
\end{lemma}

\begin{proof}
 We apply the energy estimate from Proposition \ref{prop:energy-p} with  $\theta
\in [0,1-s)$, 
$\tilde r=\frac17(5r+2R)$, and 
$\widetilde R=\frac17(r+6R)$ instead of  $r$ and $R$. 
Then, $d=\tfrac14(\widetilde R-\tilde r)=\tfrac17(R-r)$. Moreover, $\frac{\widetilde R}{\widetilde R-\widetilde r}$
and $\frac1{\widetilde R}$ can be bounded in terms of
$\frac{R}{R-r}$ and $\frac1R$ apart from a multiplicative factor. With these choices, the energy inequality becomes (note that $\widetilde R+d=R$)
\begin{align}\label{est:energy-delta=1-p<2}\nonumber
     [\btau_hu]_{W^{s+\theta (1-\frac{p}2),p}(B_{\widetilde r})}^p &\le
    \frac{C}{R^{sp\frac{p}{2}}} 
     \Big(\frac{R}{R-r}\Big)^{N+sp+1} 
    [u]_{W^{s+\theta,p}(B_{\widetilde R})}^{p(1-\frac{p}2)}\\
     &\quad\cdot
     \bigg[
    \int_{B_{\widetilde R}} \frac{| \btau_hu|^{p}}{1-s} \dx
    +
    \Big(\frac{|h|}{R}\Big)^p R^N
    \boldsymbol{\mathfrak T}(u;R)^p
    \bigg]^{\frac{p}2},
\end{align}
where $C= \widetilde C(N,p)/s$. 
In order to estimate the left-hand side of the above inequality from below we apply  Lemma~\ref{lem:N-FS} with $w$, $q$, $\gamma$, $R$, $d$ replaced by $\btau_hu$, $p$, $s+\theta (1-\frac{p}{2})$,
$\widetilde r$, $d=\frac1{7}(R-r)$.  With a constant $C=C(N,p)$ we obtain 
\begin{align*}
    \int_{B_{\widetilde r-d}}|\btau_\lambda(\btau_h u)|^p\,\dx
    &\le
    C |\lambda|^{(s+\theta (1-\frac{p}2))p}
    \bigg[ 
    \big(\underbrace{1-(s+\theta(1-\tfrac{p}{2}))}_{\le\, 1-s}\big)[\btau_hu]_{W^{s+\theta (1-\frac{p}2),p}(B_{\widetilde r})}^p\\
    &\qquad\qquad\qquad\qquad\;+
    \frac{1}{sR^{(s+\theta (1-\frac{p}2))p}}
    \Big(\frac{R}{R-r}\Big)^p
    \| \btau_hu\|^p_{L^p(B_{\widetilde r})}
    \bigg],
\end{align*}
for any $0<|\lambda|\le d$ and any $0<|h|\le d$. To obtain the last line from the  application of Lemma \ref{lem:N-FS}, we used
$s+\theta (1-\tfrac{p}2)<1$ and
\begin{align*}
    \frac{\widetilde r^{(1-s-\theta (1-\frac{p}{2}))p}}{d^p}
    +
    \frac{1}{(s+\theta (1-\frac{p}2))d^{(s+\theta(1-\frac{p}2))p}}
    &\le 
    \frac{C(p)}{s R^{(s+\theta (1-\frac{p}2))p}}\Big(\frac{R}{R-r}\Big)^p.
\end{align*}
In the above inequality  for $|\btau_\lambda(\btau_h u)|^p$, we use \eqref{est:energy-delta=1-p<2} to bound the fractional semi-norm of $\btau_h u$ on the right-hand side and obtain 
\begin{align*}
      & \int_{B_{\widetilde r-d}} |\btau_\lambda(\btau_h u)|^p\,\dx\\
      &\quad\le
      C\Big(\frac{|\lambda|}{R}\Big)^{(s+\theta (1-\frac{p}2))p}
       \Big(\frac{R}{R-r}\Big)^{N+sp+1} \\
       &\qquad
       \cdot
       \Bigg[
       (1-s)
       \Big[R^{(s+\theta)p}[u]_{W^{s+\theta,p}(B_R)}^p\Big]^{1-\frac{p}2}
       \bigg[
    \int_{B_{R}} \frac{| \btau_hu|^{p}}{1-s} \,\dx
    +
    \Big(\frac{|h|}{R}\Big)^p R^{N}
    \boldsymbol{\mathfrak T}(u;R)^p
    \bigg]^{\frac{p}2} \\
    &\quad\le
      C\Big(\frac{|\lambda|}{R}\Big)^{(s+\theta (1-\frac{p}2))p}
       \Big(\frac{R}{R-r}\Big)^{N+sp+1} \\
       &\qquad
       \cdot
       \bigg[R^{(s+\theta)p}(1-s)[u]_{W^{s+\theta,p}(B_R)}^p +
       \int_{B_{\widetilde R}} | \btau_hu|^{p} \,\dx\bigg]^{1-\frac{p}2} \\
       &\qquad \cdot
       \bigg[
    \int_{B_{R}} | \btau_hu|^{p} \,\dx
    +
    \Big(\frac{|h|}{R}\Big)^p R^{N}
    \boldsymbol{\mathfrak T}(u;R)^p
    \bigg]^{\frac{p}2} ,
\end{align*}
where $C= \widetilde C(N,p)/s$. Letting $\lm=h$ and observing that $\widetilde r-d\ge r$ as well as 
$$
    \int_{B_{\widetilde R}} | \btau_hu|^{p} \,\dx
    \le 
    2^p R^N \|u\|_{L^\infty(B_R)}^p,
$$
we conclude the proof.
\end{proof}

The next lemma is the core element for the proof of the $W^{1,p}$-regularity. The underlying idea is to apply Lemma~\ref{lem:Nikol-est-p<2}  in an iteration argument in order to increase the order of fractional differentiability. 

\begin{lemma}\label{lem:diff-quot-p<2}
Let $p\in ( 1,2]$ and  $s\in(0,1)$. Then, there exists a  constant $C\ge 1$ depending on $N$, $p$ and $s$, such that whenever $u$ is a locally bounded $(s,p)$-harmonic function in the sense of Definition~\ref{def:loc-sol}, $B_{R}\equiv B_{R}(x_o)\Subset \Omega$, and $r\in(0,R)$ we have
\begin{align}\label{est:tau_hu^p-p<2}
    \int_{B_r} |\btau_h u|^p \,\dx 
    \le 
    C\Big(\frac{|h|}{R}\Big)^p \Big(\frac{R}{R-r}\Big)^{(1+\frac{4}{sp})(N+sp+1
    )} 
    \boldsymbol{\mathfrak K}^{p}
\end{align}
for any $h\in\R^N\setminus\{0\}$ with $|h|\le R-r$, where 
\begin{align}\label{def:K<2}
    \boldsymbol{\mathfrak K}^p
    :=
    R^{sp}(1-s)[u]^p_{W^{s,p}(B_{R})} +
    R^{N} \boldsymbol{\mathfrak T}(u;R)^p .
\end{align}
Moreover, the constant $C$ is stable as $s\uparrow1$.
\end{lemma}
\begin{proof}
For $i\in \N_0$ we define sequences
\begin{equation*}
    \rho_i:= r+\frac{R-r}{2^{i+1}}
    \quad\mbox{and}\quad
    \sigma_i:=\tfrac14 spi.
\end{equation*}
First, we apply 
Lemma \ref{lem:N-FS} to $\big(u, p, s, d=\frac12 (R-r), \rho_o\big)$ instead of $\big( w,q,\gamma , d,R\big)$. Note that
$\rho_o+\frac12 (R-r)=R$. With a constant $C=C(N,p)$, Lemma \ref{lem:N-FS} shows that
\begin{align}\label{est:tau-hu-rho-o}\nonumber
     \int_{B_{\rho_o}} |\btau_h u|^p \,\dx 
     &\le
     C |h|^{sp} \bigg[
     (1-s)[u]_{W^{s,p}(B_R)}^p+\Big( \frac{R^{(1-s)p}}{d^{p}}
     +\frac1{sd^{sp}}\Big) \|u\|^p_{L^p(B_R)}
     \bigg]\\\nonumber
     &\le
     C\Big( \frac{|h|}{R} \Big)^{sp}\bigg[ R^{sp}
     (1-s)[u]_{W^{s,p}(B_R)}^p+\Big( \frac{R^{p}}{d^{p}}
     +\frac{R^{sp}}{sd^{sp}}\Big) R^N\|u\|_{L^\infty(B_R)}^p
     \bigg] \\
     &\le 
     \widetilde C_o \Big( \frac{|h|}{R} \Big)^{sp}\Big(\frac{R}{R-r}\Big)^p \boldsymbol{\mathfrak K}^p,
\end{align}
for any $0<|h|\le\tfrac12 (R-r)$. To obtain the last line we used the definitions of $d$ and $\boldsymbol{\mathfrak K}$. The constant $\widetilde C_o$ has the form $\widetilde C(N,p)/s$. 

Next, we apply Lemma \ref{lem:Nikol-est-p<2} for given $i\in\N_0$
with $R=\rho_i$,  $r=\rho_{i+1}$, $d_i=\frac1{7}(\rho_i-\rho_{i+1})=\frac{1}{7\cdot 2^{i+2}}(R-r)$, and $\theta=\sigma\in [0,1-s)$ to be chosen later.  This is only possible if the condition $u\in W^{s+\sigma,p}(B_{\rho_i})$ is fulfilled, which we will assume for the moment.
With a constant $C=\widetilde C(N,p)/s$, the application of Lemma \ref{lem:Nikol-est-p<2}
yields
\begin{align}\label{est:it-i-1-p<2} 
    \int_{B_{\rho_{i+1}}} |\btau_h (\btau_h u)|^p \,\dx
    &\le 
    C \Big(\frac{|h|}{\rho_i}\Big)^{(s+\sigma (1-\frac{p}2))p}
    \Big(\frac{\rho_i}{\rho_i-\rho_{i+1}}\Big)^{N+sp+1} \nonumber\\
    &\qquad\cdot
    \Big[\rho_i^{(s+\sigma)p}
    (1-s)[u]^p_{W^{s+\sigma,p}(B_{\rho_i})} +
    \rho_i^{N} \|u\|_{L^\infty(B_{\rho_i})}\Big]^{1-\frac{p}2} \nonumber\\
     &\qquad\cdot
     \bigg[\int_{B_{\rho_i}}|\btau_h u|^p\,\dx +
     \Big(\frac{|h|}{\rho_i}\Big)^p \rho_i^{N} 
     \boldsymbol{\mathfrak T}(u;\rho_i)^p 
     \bigg]^\frac{p}2 
\end{align}
for any $h\in\R^N\setminus\{0\}$ with $|h|\le d_i$. 
Observing that
$$
    R
    \ge
    \rho_i
    =
    \frac{R}{2^{i+1}}+ r\Big( 1-\frac{1}{2^{i+1}}\Big)
    >
    \frac{R}{2^{i+1}},
$$
and 
$$
    \rho_i-\rho_{i+1}
    =
    \frac{R-r}{2^{i+2}},
$$
we have
$$
     \frac{1}{\rho_i^{(s+\sigma (1-\frac{p}2))p}} \Big(\frac{\rho_i}{\rho_i-\rho_{i+1}}\Big)^{N+sp+1}
     \le
     \frac{C(N,p)2^{(N+2p+1)i}}{R^{(s+\sigma (1-\frac{p}2))p}} \Big(\frac{R}{R-r}\Big)^{N+sp+1}
$$
and by Remark~\ref{rem:t} we have 
\begin{align*}
    \rho_i^{N}\boldsymbol{\mathfrak T}(u;\rho_i)^p
    \le
    \rho_i^{N}
    \bigg[ 
    C(N)
    \Big(\frac{R}{\rho_i}\Big)^{N}\bigg]^\frac{p}{p-1}\boldsymbol{\mathfrak T}(u;R)^p
    \le
     C(N)^\frac{p}{p-1}2^{\frac{N}{p-1}(i+1)}R^{N}\boldsymbol{\mathfrak T}(u;R)^p.
\end{align*}
With these remarks, \eqref{est:it-i-1-p<2} immediately becomes
\begin{align}\label{est:it-i-2-p<2}\nonumber
    \int_{B_{\rho_{i+1}}} |\btau_h (\btau_h u)|^p \,\dx
     &\le 
    C_i \Big(\frac{|h|}{R}\Big)^{(s+\sigma (1-\frac{p}2))p}
    \Big(\frac{R}{R-r}\Big)^{N+sp+1}\\\nonumber
    &\phantom{\le\,}\cdot
    \Big[R^{(s+\sigma)p}(1-s) [u]^p_{W^{s+\sigma,p}(B_{\rho_i})} +
    \boldsymbol{\mathfrak K}^p \Big]^{1-\frac{p}2}\\
     &\phantom{\le\,}\cdot\bigg[
     \int_{B_{\rho_i}}|\btau_h u|^p\,\dx +
     \Big(\frac{|h|}{R}\Big)^p \boldsymbol{\mathfrak K}^p 
     \bigg]^\frac{p}2 
\end{align}
for any $h\in\R^N\setminus\{0\}$ with $|h|\le d_i$, provided that $u\in W^{s+\sigma,p}(B_{\rho_i})$, where $\sigma\in[0,1-s)$, and with a constant $C_i$ of the form 
\begin{equation*}
    C_i=
    \boldsymbol c\,\boldsymbol b^{i+1},
    \quad \mbox{where $\boldsymbol c=\boldsymbol c(N,p,s)=\frac{\widetilde C(N,p)}{s}$\,,\, $\boldsymbol b=\boldsymbol b(N,p) = 2^{\frac{Np}{p-1}+3p+1}$.}
\end{equation*}
Later in the proof, we will iterate inequality~\eqref{est:it-i-2-p<2}. However, we start by considering the inequality for $i=0$. In fact, we apply \eqref{est:it-i-2-p<2}$_{i=0}$ with $\sigma =0$. This is possible, since the regularity assumption $u\in W^{s,p}(B_{\rho_o})$ is granted.
Subsequently, we use inequality \eqref{est:tau-hu-rho-o}, the definition of $\boldsymbol{\mathfrak K}$ in \eqref{def:K<2} and the fact that  $\frac{|h|}{R}\le 1$, which allows us to reduce the exponent from $p$ to $sp$. In this way, we obtain 
\begin{align}\label{it-0}
    \int_{B_{\rho_{1}}}& |\btau_h (\btau_h u)|^p \,\dx \nonumber\\
     &\le 
    C_0 \Big(\frac{|h|}{R}\Big)^{sp}
    \Big(\frac{R}{R-r}\Big)^{N+sp+1} \nonumber\\
    &\phantom{\le\,}
    \cdot\Big[
    \underbrace{
    R^{sp}(1-s)[u]^p_{W^{s,p}(B_{\rho_o})}}_{\le \boldsymbol{\mathfrak K}^p} +
    \boldsymbol{\mathfrak K}^p\Big]^{1-\frac{p}2}
     \bigg[
    \int_{B_{\rho_o}}|\btau_h u|^p\,\dx +
    \Big(\frac{|h|}{R}\Big)^p\boldsymbol{\mathfrak K}^p
     \bigg]^\frac{p}2 \nonumber\\
    &\le 
    C_02^{1-\frac{p}2} \Big(\frac{|h|}{R}\Big)^{sp}
    \Big(\frac{R}{R-r}\Big)^{N+sp+1}
     \bigg[
    \widetilde C_o\Big(\frac{|h|}{R}\Big)^{sp}\Big(\frac{R}{R-r}\Big)^{p} +
    \Big(\frac{|h|}{R}\Big)^p
     \bigg]^\frac{p}2 \boldsymbol{\mathfrak K}^p \nonumber\\
     &\le
     \widetilde B_o \Big(\frac{|h|}{R}\Big)^{sp(1+\frac{p}{2})}
    \Big(\frac{R}{R-r}\Big)^{2(N+sp+1)}\boldsymbol{\mathfrak K}^p,
\end{align}
for every $0<|h|\le d_o$ and with a constant $2C_0 \widetilde C_o^\frac{p}2\le 2\boldsymbol c\boldsymbol b \widetilde C_o=:\widetilde B_o$.

We now choose $i_o\in\N$, such that $s+\sigma_{i_o}\ge 1$ and $s+\sigma_{i_o-1}< 1$. The precise value of $i_o$ is
\begin{equation}\label{def-io}
    i_o=\Bigl\lceil\frac{4(1-s)}{sp}\Bigr\rceil.
\end{equation}
Note that $\sigma_i\in[0,1-s)$ for any $i\in\{0,\dots,i_o-1\}$.
Moreover, $i_o$ only depends on $p$ and $s$.
In order to examine the stability of our estimates as $s\uparrow1$ we consider {\bf two cases}, namely $i_o=1$, which corresponds to $s\in[s_o,1)$, and $i_o\ge 2$, which corresponds to $s\in(0,s_o)$, where $s_o:=\frac{4}{4+p}$. 

First, we deal with \textbf{the case} $i_o=1$. In this case we have $s\in[s_o,1)$. From~\eqref{it-0} we have 
\begin{align*}
    \int_{B_{\rho_{1}}} |\btau_h (\btau_h u)|^p \,\dx 
    \le
     \widetilde B_o \Big(\frac{|h|}{R}\Big)^{sp(1+\frac{p}{2})}
    \Big(\frac{R}{R-r}\Big)^{2(N+sp+1)}\boldsymbol{\mathfrak K}^p,
\end{align*}
for every $0<|h|\le d_o$, where $\widetilde B_o=2\boldsymbol c\boldsymbol b \widetilde C_o$.
For the exponent of $\frac{|h|}{R}$ we compute
\begin{align*}
    s(1+\tfrac{p}{2})
    \ge 
    s_o(1+\tfrac{p}{2})
    =
    \frac{4+2p}{4+p}
    =
    1+\frac{p}{4+p}\ge 1+\tfrac15.
\end{align*}
Therefore, we can apply inequality~\eqref{est-1st-diffquot>1} from Lemma \ref{lem:Domokos} with
$(q,\gamma, r,R,d)$ replaced by $\big( p,s(1+\frac{p}{2}), \rho_{1}, \rho_o,d_{1}\big)$ and with
$$
    M^p=\frac{\widetilde B_o}{R^{sp(1+\frac{p}{2})}}
    \Big(\frac{R}{R-r}\Big)^{2(N+sp+1)}\boldsymbol{\mathfrak K}^p
$$
and obtain
\begin{align*}
     \int_{B_{\rho_{1}}}& |\btau_h u|^p \,\dx\\
     &\le
     C(p)|h|^p
     \bigg[
    \Big(\frac{M}{s(1+\tfrac{p}{2})-1}\Big)^p d_{1}^{(s(1+\frac{p}{2})-1)p}
    +
    \frac{1}{d_{1}^p}\int_{B_{\rho_o}}|u|^p\,\dx
     \bigg]\\
     &\le
     C(p)|h|^p
     \bigg[\widetilde B_o
    \underbrace{\frac{d_{1}^{(s(1+\frac{p}{2})-1)p}}{\frac15 R^{sp(1+\frac{p}{2})}}}_{\le \frac{5}{R^{p}}} \Big(\frac{R}{R-r}\Big)^{2(N+sp+1)}
    +
    \frac{2^N}{R^p}\Big(\frac{R}{\frac1{28}(R-r)} \Big)^p
     \bigg]\boldsymbol{\mathfrak K}^p\\
    &\le
    C(N,p)\widetilde B_o\Big(\frac{|h|}{R}\Big)^p \Big(\frac{R}{R-r}\Big)^{2(N+sp+1)}\boldsymbol{\mathfrak K}^p
\end{align*}
for every $0<|h|\le d_{1}$. Observe that the constant remains stable as $s\downarrow s_o$ and $s\uparrow 1$. Moreover, it blows up as $p\downarrow 1$ and remains stable as $p\uparrow 2$. Since $\rho_1\ge r$, this proves inequality~\eqref{est:tau_hu^p-p<2} in the case $i_o=1$ for any $0<|h|\le d_{1}$. 

Next, we consider \textbf{the case} $i_o\ge2$, where $s\in(0,s_o)$. 
In this range, no stability problem arises, since $s$ is bounded away from 1. 
Now, let
\begin{equation*}
    \alpha:=\frac{1-s}{\sigma_{i_o}+\frac18sp}.
\end{equation*}
Since $\sigma_{i_o}\ge 1-s$, we have 
\begin{equation}\label{est:alpha}
    \alpha\le \frac{1-s}{1-s+\frac18sp}<1.
\end{equation}
On the other hand, since $\sigma_{i_o}=\sigma_{i_o-1}+\frac14 sp<1-s +\frac14 sp$ and and the fact that $[0,1]\ni s\mapsto \frac{1-s}{1-s+\frac38 sp}$ is decreasing, we have
\begin{align}\label{est:alpha>1/5}
    \alpha
    >
    \frac{1-s}{1-s +\frac14 sp +\frac18 sp}
    =
    \frac{1-s}{1-s+ \frac38 sp}
    \ge
    \frac{1-s_o}{1-s_o+ \frac38 s_o p}
    =
    \tfrac25.
\end{align}
Next, we estimate the distance from $s+\alpha \sigma_{i_o}$ to $1$. Using in turn the definition of $\alpha$, i.e. $1-s=\alpha(\sigma_{i_o}+\tfrac18sp)$ and~\eqref{est:alpha>1/5}, we find
\begin{align}\label{est:sigma-io}
    s+\alpha \sigma_{i_o}
    =
    1-(1-s) +\alpha \sigma_{i_o}
    =
    1- \tfrac18sp\alpha 
    <
    1-\tfrac{1}{20}sp.
\end{align}
Similarly, we can estimate the distance from  $s+\alpha \sigma_{i_o+1}$ to $1$. In fact, from~\eqref{est:sigma-io} and~\eqref{est:alpha>1/5} we have
\begin{align}\label{est:sigma-io+1}
    s+\alpha \sigma_{i_o+1}
    =
    s + \alpha \sigma_{i_o} + \tfrac14sp\alpha 
    =
    1 - \tfrac18sp\alpha + \tfrac14sp\alpha
    =
    1 + \tfrac18sp\alpha
    >
    1+ \tfrac{1}{20}sp.
\end{align}
Recall that we are dealing with the case $i_o\ge 2$.
By {\bf induction} we show for any $i\in\{ 1,\dots ,i_o-1\}$ that
\begin{align}\label{est:ind-i}\nonumber
    \int_{B_{\rho_i}}&|\btau_hu|^p\,\dx\\
    &\le
    \underbrace{
    2\boldsymbol c^i\boldsymbol b^{1+2+\dots +i}C_\ast^{2i-1}\widetilde C_o}_{:=
    \mathfrak C_i} \Big(\frac{|h|}{R}\Big)^{(s+\alpha\sigma_{i+1})p}
    \Big(\frac{R}{R-r}\Big)^{(i+1)(N+sp+1)}\boldsymbol{\mathfrak K}^p
\end{align}
for any $0<|h|\le d_i$,
and
\begin{align}\label{ind-requirement}
    [u]_{W^{s+\alpha\sigma_i,p}(B_{\rho_i})}^p
    &\le
    \underbrace{
    C_\ast \mathfrak C_i}_{=:\mathfrak D_i}
    \frac{1}{R^{(s+\alpha\sigma_i)p}}\Big(\frac{R}{R-r}\Big)^{(i+1)(N+sp+1)}\boldsymbol{\mathfrak K}^p.
\end{align}
The constant  $C_\ast$  is defined by
\begin{align*}
    C_\ast :=28^p2^{N+1} \max\bigg\{ \frac{C(p)}{s^p}, \frac{C(N,p)}{s}\bigg\},
\end{align*}
where $C(p)$ stands for the constant in \eqref{est-1st-diffquot<1} from Lemma \ref{lem:Domokos}, and $C(N,p)$ for the one from Lemma \ref{lem:FS-N}.

Let us start  by considering {\bf the case} $i=1$. From~\eqref{it-0} we have
\begin{align}\label{est:tauh-tauh-p<2}\nonumber
    \int_{B_{\rho_{1}}}|\btau_h (\btau_h u)|^p \,\dx
      &\le 
    \widetilde B_o \Big(\frac{|h|}{R}\Big)^{sp(1+\frac{p}{2})}
    \Big(\frac{R}{R-r}\Big)^{2(N+sp+1)}\boldsymbol{\mathfrak K}^p\\
    &\le
     \widetilde B_o \Big(\frac{|h|}{R}\Big)^{(s+\alpha \sigma_2)p}
    \Big(\frac{R}{R-r}\Big)^{2(N+sp+1)} \boldsymbol{\mathfrak K}^p,
\end{align}
for every $0<|h|\le d_o$, where $\widetilde B_o=2\boldsymbol c\boldsymbol b \widetilde C_o$.
In the transition to the last line we used  $\sigma _2=s\frac{p}{2}$ and $\alpha <1$; cf.~\eqref{est:alpha}. Note that the exponent $s+\alpha\sigma_2$
of $|h|$ is strictly less than 1. In fact, we have $s+\alpha\sigma_2\le s+\alpha\sigma_{i_o}<1-\tfrac1{20}sp$; cf.~\eqref{est:sigma-io}. 
Inequality \eqref{est:tauh-tauh-p<2} plays the role of assumption \eqref{ass:Domokos}$_{\sigma<1}$ in Lemma \ref{lem:Domokos}
and permits us to apply \eqref{est-1st-diffquot<1}  with $\big(p, s+\alpha\sigma_2, \rho_1, \rho_{o}, d_o\big)  $ in the place of $\big(q, \gamma, r,R, d\big)$ and with
$$
    M^p
    := 
    \frac{\widetilde B_o}{R^{(s+\alpha\sigma_2)p}}
    \Big(\frac{R}{R-r}\Big)^{2(N+sp+1)}\boldsymbol{\mathfrak K}^p.
$$
Using again the $L^\infty$-bound for $u$, the lower bound $1-(s+\alpha\sigma_2)\ge \frac{1}{20}s$, the fact that $s+\alpha\sigma_2<1<N+sp+1$, the definition of $\boldsymbol{\mathfrak K}$, and $d_o=\frac{1}{7 \cdot 4}(R-r)$, we obtain for every $0<|h|\le d_o$ that
\begin{align*}
     \int_{B_{\rho_{1}}} |\btau_h u|^p \,\dx
     &\le
     C(p) |h|^{(s+\alpha\sigma_2)p}
     \bigg[
        \bigg(\frac{M}{\frac1{20}s}\bigg)^p
        +
        \frac{1}{d_o^{(s+\alpha\sigma_2)p}}
        \int_{B_{\rho_o}}|u|^p\,\dx
     \bigg]\\
     &\le
      \frac{28^pC(p) \widetilde B_o}{s^p}\Big(\frac{|h|}{R}\Big)^{(s+\alpha\sigma_2)p}\\
      &\phantom{\le \,}\cdot
     \bigg[
       \Big(\frac{R}{R-r}\Big)^{2(N+sp+1)}\boldsymbol{\mathfrak K}^p
        +
        \Big(\frac{R}{R-r}\Big)^{(s+\alpha\sigma_2)p}
        2^{N}R^N\| u\|_{L^\infty (B_R)}^p
     \bigg]\\
     &\le
     \underbrace{
    \frac{28^p2^{N+1}C(p)}{s^p}}_{\le C_\ast}\widetilde B_o
     \Big(\frac{|h|}{R}\Big)^{(s+\alpha\sigma_2)p}\Big(\frac{R}{R-r}\Big)^{2(N+sp+1)}\boldsymbol{\mathfrak K}^p. 
\end{align*}
This proves \eqref{est:ind-i}$_{i=1}$ with the constant $C_\ast \widetilde B_o=2\mathbf c\mathbf b C_\ast\widetilde C_o= \mathfrak C_1$.
To obtain the assertion~\eqref{ind-requirement}$_{i=1}$, we apply Lemma \ref{lem:FS-N} with $(q,\gamma,\beta,R,d)$ replaced by $\big( p,s+\alpha\sigma_2, s+\alpha\sigma_1, \rho_1, d_o\big)$ and
\begin{equation*}
    M^p
    :=
     \frac{\mathfrak C_1}{R^{(s+\alpha\sigma_2)p}}\Big(\frac{R}{R-r}\Big)^{2(N+sp+1)}\boldsymbol{\mathfrak K}^p. 
\end{equation*}
Then, $\beta-\gamma =\alpha (\sigma_2-\sigma_1)=\alpha s\frac{p}{4}\ge\frac1{10}sp>\frac1{10}s $; cf.~\eqref{est:alpha>1/5}. In this context, Lemma \ref{lem:FS-N} and $s+\alpha\sigma_1>s$ imply  that 
\begin{align*}
   & [u]^p_{W^{s+\alpha\sigma_1,p}(B_{\rho_1})}\\
    &\quad\le
   C(N,p)
   \Bigg[
   \frac{d_o^{\alpha s\frac{p}{4}p}}{\frac{1}{10}s}
   M^p +
   \frac{1}{(s+\alpha\sigma_1)d_o^{(s+\alpha\sigma_1)p}}\int_{B_{\rho_1}}|u|^p\,\dx
   \Bigg]\\
   &\quad\le
   \frac{ 28^p\cdot 2^{N}C(N,p)}{s}
   \frac{1}{R^{(s+\alpha\sigma_1)p}}\\
   &\qquad\qquad\cdot
   \bigg[\mathfrak C_1
   \underbrace{
   \frac{R^{(s+\alpha\sigma_1)p}(R-r)^{\alpha s\frac{p}{4}p}}{R^{(s+\alpha\sigma_2)p}}}_{\le 1}\Big(\frac{R}{R-r}\Big)^{2(N+sp+1)}
   +\Big( \frac{R}{R-r}\Big)^{(s+\alpha\sigma_1)p}
   \bigg]\boldsymbol{\mathfrak K}^p\\
   &\quad\le 
   \underbrace{\frac{28^p2^{N+1}C(N,p)}{s}}_{\le C_\ast}
   \frac{\mathfrak C_1}{R^{(s+\alpha\sigma_1)p}}\Big(\frac{R}{R-r}\Big)^{2(N+sp+1)} \boldsymbol{\mathfrak K}^p\\
    &\quad \le 
   \frac{C_\ast \mathfrak C_1}{R^{(s+\alpha\sigma_1)p}}\Big(\frac{R}{R-r}\Big)^{2(N+sp+1)} \boldsymbol{\mathfrak K}^p.
\end{align*}
In turn we used $(s+\alpha\sigma_1)p<2(N+sp+1)$ to homogenize the exponents of $\frac{R}{R-r}$ and the definitions of $\boldsymbol{\mathfrak K}$ and $C_\ast$.
This proves \eqref{ind-requirement}$_{i=1}$ with the constant $C_\ast \mathfrak C_1=\mathfrak D_1$.

For the {\bf induction step}, we assume that \eqref{est:ind-i}  and \eqref{ind-requirement}  hold true for some $i\in\{1,\dots ,i_o-2\}$, and show they continue to hold for $i+1$. To this end, applying \eqref{est:it-i-2-p<2}$_i$ with $\sigma=
\alpha\sigma_i$, which is possible since \eqref{ind-requirement}$_i$ ensures that
$u\in W^{s+\alpha \sigma_i,p}(B_{\rho_i})$, we obtain
\begin{align*}
    \int_{B_{\rho_{i+1}}} |\btau_h (\btau_h u)|^p \,\dx
     &\le 
    C_i\Big(\frac{|h|}{R}\Big)^{(s+\alpha \sigma_i (1-\frac{p}2))p}
    \Big(\frac{R}{R-r}\Big)^{N+sp+1}\\
    &\qquad\cdot\Big[
    R^{(s+\alpha \sigma_i)p}(1-s)[u]^p_{W^{s+\alpha \sigma_i,p}(B_{\rho_i})} +
    \boldsymbol{\mathfrak K}^p\Big]^{1-\frac{p}2}\\
     &\qquad
    \cdot
     \bigg[
    \int_{B_{\rho_i}}|\btau_h u|^p + 
    \Big(\frac{|h|}{R}\Big)^p \boldsymbol{\mathfrak K}^p\,\dx
     \bigg]^\frac{p}2\\
     &\le 
    C_i \Big(\frac{|h|}{R}\Big)^{(s+\alpha \sigma_i (1-\frac{p}2))p}
    \Big(\frac{R}{R-r}\Big)^{N+sp+1}\boldsymbol{\mathfrak K}^p\\
    &\qquad
    \cdot\bigg[\mathfrak D_i\Big(\frac{R}{R-r}\Big)^{(i+1)(N+sp+1)}
    + 1\bigg]^{1-\frac{p}2}\\
    &\qquad
    \cdot
     \bigg[
    \mathfrak C_i\Big(\frac{|h|}{R}\Big)^{(s+\alpha\sigma_{i+1})p}
    \Big(\frac{R}{R-r}\Big)^{(i+1)(N+sp+1)} + 
    \Big(\frac{|h|}{R}\Big)^p
     \bigg]^\frac{p}2
\end{align*}
for every $0<|h|\le d_i$.
Here, to obtain the last line we  used \eqref{est:ind-i}$_i$ and \eqref{ind-requirement}$_i$. Since $s+\alpha\sigma_{i+1}\le s+\alpha\sigma_{i_o-1}<1$, we may reduce the exponent of $\frac{|h|}{R}$ form $p$ to $(s+\alpha\sigma_{i+1})p$. This leads to
\begin{align*}
    &\int_{B_{\rho_{i+1}}} |\btau_h (\btau_h u)|^p \,\dx\\
     &\qquad\le 
    \underbrace{
    2C_i\mathfrak  D_i^{1-\frac{p}{2}} 
    \mathfrak  C_i^\frac{p}{2}}_{=:\,\mathfrak B_i}
    \Big(\frac{|h|}{R}\Big)^{[s+\alpha \sigma_i (1-\frac{p}2)+(s+\alpha\sigma_{i+1})\frac{p}2]p}
    \Big(\frac{R}{R-r}\Big)^{(i+2)(N+sp+1)}\boldsymbol{\mathfrak K}^p.
\end{align*}
Note that
\begin{align}\label{est:mathfrak-B_i}
    \mathfrak B_i
    = 
    2C_i[C_\ast \mathfrak C_i]^{1-\frac{p}2}\mathfrak C_i^\frac{p}2
    \le 
    C_\ast C_i \mathfrak C_i
    =
    2\boldsymbol c^{i+1}\boldsymbol b^{1+2+\dots +i+(i+1)}C_\ast^{2i}\widetilde C_o.
\end{align}
For the exponent of $\frac{|h|}{R}$ we have
\begin{align*}
    s+\alpha \sigma_i (1-\tfrac{p}2)
    +
     (s+\alpha \sigma_{i+1})\tfrac{p}{2}
     &>
     s+\alpha \sigma_i (1-\tfrac{p}2)
    +
     (s+\alpha \sigma_{i})\tfrac{p}{2}\\
     &=
     s+\alpha\sigma_i+s\tfrac{p}2\\
     &=
     s+\alpha\sigma_{i+2}+s\tfrac{p}{2}(1-\alpha)\\
     &>
      s+\alpha\sigma_{i+2}.
\end{align*}
Therefore, we find for every $0<|h|\le d_i$ that
\begin{equation*}
    \int_{B_{\rho_{i+1}}} |\btau_h (\btau_h u)|^p \,\dx
     \le 
     \mathfrak B_i \Big(\frac{|h|}{R}\Big)^{(s+\alpha \sigma_{i+2} )p}
    \Big(\frac{R}{R-r}\Big)^{(i+2)(N+sp+1)}\boldsymbol{\mathfrak K}^p.
\end{equation*}
Note that the exponent of $|h|$  is still strictly less than $p$. Indeed, since $i+2\le i_o$, we have
\begin{equation*}
    s+\alpha\sigma_{i+2}\le s+\alpha\sigma_{i_o}
    < 1-\tfrac{1}{20}sp,
\end{equation*}
where we  used \eqref{est:sigma-io}. For this reason we are able to apply~\eqref{est-1st-diffquot<1} from Lemma \ref{lem:Domokos} with $(\gamma , q,r,R,d)$ replaced by
$(s+\alpha\sigma_{i+2}, p,\rho_{i+1}, \rho_i, d_i)$ and
\begin{equation*}
    M^p
    :=
    \frac{\mathfrak B_i}{R^{(s+\alpha \sigma_{i+2} )p}}
    \Big(\frac{R}{R-r}\Big)^{(i+2)(N+sp+1)}\boldsymbol{\mathfrak K}^p,
\end{equation*}
to obtain
\begin{align*}
    &\int_{B_{\rho_{i+1}}}|\btau_hu|^p\,\dx\\
    &\quad\le
    C(p)|h|^{(s+\alpha \sigma_{i+2} )p}\bigg[
    \bigg(\frac{M}{\frac{1}{20} s}\bigg)^p
    +
    \frac{1}{d_i^{(s+\alpha \sigma_{i+2} )p}}\int_{B_{\rho_i}}|u|^p\,\dx
    \bigg]\\
    &\quad\le
    C(p)\Big(\frac{|h|}{R}\Big)^{(s+\alpha \sigma_{i+2} )p}
    \\
    &\qquad\quad\cdot\Bigg[ 
    \frac{\mathfrak B_i}{(\frac1{20}s)^p}\Big(\frac{R}{R-r}\Big)^{(i+2)(N+sp+1)} + \Big( \frac{R}{\frac{1}{7\cdot 2^{i+2}}(R-r)}
    \Big)^{(s+\alpha \sigma_{i+2} )p}2^NR^N\Bigg]\boldsymbol{\mathfrak K}^p\\
    &\le
    \frac{28^pC(p)2^N}{s^p}
    \Big(\frac{|h|}{R}\Big)^{(s+\alpha \sigma_{i+2} )p}
    \Big(\frac{R}{R-r}\Big)^{(i+2)(N+sp+1)}\big[ \underbrace{\mathfrak B_i+2^{ip}}_{\le 2\mathfrak B_i}\big]\boldsymbol{\mathfrak K}^p\\
    &\le
    \underbrace{\frac{28^p2^{N+1}C(p)\mathfrak B_i}{s^p}}_{\le\, C_\ast\mathfrak B_i\le\, \mathfrak C_{i+1}}
    \Big(\frac{|h|}{R}\Big)^{(s+\alpha \sigma_{i+2} )p}
    \Big(\frac{R}{R-r}\Big)^{(i+2)(N+sp+1)}\boldsymbol{\mathfrak K}^p
\end{align*}
 for every $0<|h|\le d_i$. Here, we also used \eqref{est:sigma-io} to bound $1-s-\alpha\sigma_{i+2}$ from below by $\frac1{20}s$, the definition of $d_i$ to replace $\frac1{d_i}$ by $\frac1{R-r}$, the $L^\infty$-bound for $u$, the definition of $\boldsymbol{\mathfrak K}$, and the fact that $(s+\alpha \sigma_{i+2} )p<p\le 2\le 2N$. Moreover, we used  that $2^{ip}\le \mathfrak B_i$ by \eqref{est:mathfrak-B_i}.
This proves \eqref{est:ind-i}$_{i+1}$. 

The completion of the induction step now consists in establishing \eqref{ind-requirement}$_{i+1}$. 
To do this, we apply Lemma \ref{lem:FS-N} with $(q,\gamma ,\beta, R,d)$ replaced by $\big( p,s+\alpha\sigma_{i+2}, s+\alpha\sigma_{i+1}, \rho_{i+1}, d_{i+1}\big)$ and with 
\begin{equation*}
    M^p
    :=
     \frac{\mathfrak C_{i+1}}{R^{(s+\alpha\sigma_{i+2})p}}\Big(\frac{R}{R-r}\Big)^{(i+2)(N+sp+1)}\boldsymbol{\mathfrak K}^p. 
\end{equation*}
Then, $\gamma -\beta  =\alpha (\sigma_{i+2}-\sigma_{i+1})=\alpha s\frac{p}{4}\ge\frac1{10}sp$ by \eqref{est:alpha>1/5}. Using also the $L^\infty$-bound for $u$ to estimate the integral of $|u|^p$ on $B_{\rho_{i+1}}$ in terms of $\boldsymbol{\mathfrak K}^p$, the definition of $M^p$, the relation between $\sigma_{i+2}$ and $\sigma_{i+1}$, the definition of $d_{i+1}$, and $(s+\alpha \sigma_{i+1} )p<p\le 2\le 2N$ to increase the exponent of $\frac{R}{R-r}$ from $(s+\alpha \sigma_{i+1} )p$ to $(i+2)(N+sp+1)$ we get
\begin{align*}
    [u]_{W^{s+\alpha\sigma_{i+1},p}(B_{\rho_{i+1}})}^p
    &\le
    C(N,p)\Bigg[
    \frac{d_{i+1}^{\alpha s\frac{p}{4}p}}{\frac1{10}s}M^p +
    \frac{1}{ s d_{i+1}^{(s+\alpha\sigma_{i+1})p}}
    \int_{B_{\rho_{i+1}}}|u|^p\,\dx
    \Bigg]\\
    &\le
    \frac{C(N,p)}{s}\bigg[
    10 R^{\alpha s\frac{p}{4}p} M^p +
    \Big( \frac{7\cdot 2^{i+3}}{R-r}
    \Big)^{(s+\alpha \sigma_{i+1} )p}
    2^N\boldsymbol{\mathfrak K}^p
    \bigg]\\
    &\le
    \frac{28^p2^{N}C(N,p)}{s}\bigg[
    R^{\alpha s\frac{p}{4}p} M^p +
    \frac{2^{(i+1)p}}{(R-r)^{(s+\alpha \sigma_{i+1} )p}}
    \boldsymbol{\mathfrak K}^p
    \bigg]\\
    &\le
     \frac{28^p2^{N}C(N,p)}{s R^{(s+\alpha\sigma_{i+1})p}}
     \Big(\frac{R}{R-r}\Big)^{(i+2)(N+sp+1)}
     \big[\mathfrak C_{i+1} +
    2^{(i+1)p} \big]\boldsymbol{\mathfrak K}^p\\
     &\le
      \frac{28^p2^{N+1}C(N,p)}{s}\frac{\mathfrak C_{i+1}}{R^{(s+\alpha\sigma_{i+1})p}}
     \Big(\frac{R}{R-r}\Big)^{(i+2)(N+sp+1)} \boldsymbol{\mathfrak K}^p\\
      &\le
      \frac{C_\ast \mathfrak C_{i+1}}{R^{(s+\alpha\sigma_{i+1})p}}
     \Big(\frac{R}{R-r}\Big)^{(i+2)(N+sp+1)} \boldsymbol{\mathfrak K}^p.
\end{align*}
In turn we used $2^{(i+1)p}\le \mathfrak C_{i+1}$.
Since $C_\ast \mathfrak C_{i+1}=\mathfrak D_{i+1}$ this proves the claim \eqref{ind-requirement}$_{i+1}$.

At this point, it remains to {\bf establish inequality~\eqref{est:tau_hu^p-p<2}} in the case $i_o\ge 2$. By \eqref{ind-requirement}$_{i_o-1}$ we have
$[u]_{W^{s+\alpha\sigma_{i_o-1},p}(B_{\rho_{i_o-1}})}<\infty$, which is the requirement for the application of \eqref{est:it-i-2-p<2}$_{i_o-1}$
on $B_{\rho_{i_o-1}}$ with $\alpha\sigma_{i_o-1}$ instead of $\sigma$. 
According to \eqref{est:it-i-2-p<2}$_{i_o-1}$ we have
\begin{align*}
    \int_{B_{\rho_{i_o}}}& |\btau_h (\btau_h u)|^p \,\dx
     \le 
    C_{i_o-1} \Big(\frac{|h|}{R}\Big)^{(s+\alpha \sigma_{i_o-1} (1-\frac{p}2))p}
    \Big(\frac{R}{R-r}\Big)^{N+sp+1}\mathbf I_1^{1-\frac{p}2}\mathbf{I}_2^\frac{p}2
 \end{align*}
 for any $h\in\R^N\setminus\{0\}$ with $|h|\le d_{i_o-1}$. Here we abbreviated
 \begin{align*}
     \mathbf I_1
     &:=
    R^{(s+\alpha \sigma_{i_o-1})p}(1-s)
    [u]^p_{W^{s+\alpha \sigma_{i_o-1},p}(B_{\rho_{i_o-1}})} +
    \boldsymbol{\mathfrak K}^p,\\
    \mathbf{I}_2
    &:=
    \int_{B_{\rho_{i_o-1}}}|\btau_h u|^p\,\dx +
    \Big(\frac{|h|}{R}\Big)\boldsymbol{\mathfrak K}^p.
 \end{align*}
To estimate $\mathbf I_1$ we use \eqref{ind-requirement}$_{i_o-1}$, which implies
\begin{align*}
    \mathbf I_1
    &\le
    \mathfrak D_{i_o-1} 
    \Big(\frac{R}{R-r}\Big)^{i_o(N+sp+1)}\boldsymbol{\mathfrak K}^p +
    \boldsymbol{\mathfrak K}^p\\
     &\le
     2\mathfrak D_{i_o-1} 
     \Big(\frac{R}{R-r}\Big)^{i_o(N+sp+1)}\boldsymbol{\mathfrak K}^p.
\end{align*}
With \eqref{est:ind-i}$_{i_o-1}$ we estimate $\mathbf I_2$ by
\begin{align*}
    \mathbf I_2
    &\le
    \bigg[
    \mathfrak C_{i_o-1}
     \Big(\frac{|h|}{R}\Big)^{(s+\alpha\sigma_{i_o})p}
    \Big(\frac{R}{R-r}\Big)^{i_o(N+sp+1)} +
    \Big(\frac{|h|}{R}\Big)^p
    \bigg]\boldsymbol{\mathfrak K}^p\\
    &\le 
    2 \mathfrak C_{i_o-1}
    \Big(\frac{|h|}{R}\Big)^{(s+\alpha\sigma_{i_o})p}
    \Big(\frac{R}{R-r}\Big)^{i_o(N+sp+1)}
    \boldsymbol{\mathfrak K}^p.
\end{align*}
Here, to obtain the last line we used \eqref{est:sigma-io}, 
i.e.~$s+\alpha\sigma_{i_o}<1$, in order to reduce the exponent of $\frac{|h|}{R}$. Plugging these estimates back and recalling that $\mathfrak D_{i_o-1} =C_\ast \mathfrak C_{i_o-1}$, we get
\begin{align*}
    &\int_{B_{\rho_{i_o}}} |\btau_h (\btau_h u)|^p \,\dx\\
     &\quad
     \le 
    \underbrace{2C_\ast^{-\frac{p}2}}_{\le\, 1}
    \underbrace{C_{i_o-1} \mathfrak D_{i_o-1}}_{\le\, \mathfrak C_{i_o}}
    \Big(\frac{|h|}{R}\Big)^{(s+\alpha \sigma_{i_o-1} (1-\frac{p}2)
    + (s+\alpha\sigma_{i_o})\frac{p}{2})p }
    \Big(\frac{R}{R-r}\Big)^{(i_o+1)(N+sp+1)}
    \boldsymbol{\mathfrak K}^p.
 \end{align*}
 The exponent of $\frac{|h|}{R}$ can be simplified. In fact, we have
 \begin{align*}
     s+\alpha \sigma_{i_o-1} (1-\tfrac{p}2)
     +
     (s+\alpha\sigma_{i_o})\tfrac{p}{2}
     &=
     s+\alpha \sigma_{i_o-1} (1-\tfrac{p}2)
     +
     (s+\alpha\sigma_{i_o-1}+\alpha\tfrac14 sp)\tfrac{p}{2}\\
     &>
     s+\alpha \sigma_{i_o-1}+ \tfrac{1}2 sp\\
     &>
    s+\alpha \sigma_{i_o+1}. 
\end{align*}
Therefore, for any $h\in\R^N\setminus\{0\}$ with $|h|\le d_{i_o-1}$ we get
\begin{align}\label{est:tauh-tauh-io}
    \int_{B_{\rho_{i_o}}} |\btau_h (\btau_h u)|^p \,\dx
     &
     \le 
    \mathfrak C_{i_o} \Big(\frac{|h|}{R}\Big)^{(s+\alpha \sigma_{i_o+1} )p}
    \Big(\frac{R}{R-r}\Big)^{(i_o+1)(N+sp+1)}
    \boldsymbol{\mathfrak K}^p.
 \end{align}
Recall that the exponent $s+\alpha \sigma_{i_o+1}$ satisfies \eqref{est:sigma-io+1}, i.e.
$$
    s+\alpha \sigma_{i_o+1}>1+\tfrac1{20}sp.
$$
Therefore, we can apply \eqref{est-1st-diffquot>1} from Lemma \ref{lem:Domokos}  on $B_{\rho_i}$, i.e.~the case in which the exponent $\sigma$
is greater than 1,  with $(p,s+\alpha \sigma_{i_o+1}, \rho_{i_o-1}, \rho_{i_o}, d_{i_o})$ instead of $(q,\gamma, R,r, d)$ and with
$$
    M^p
    :=
    \frac{\mathfrak C_{i_o}}{R^{(s+\alpha \sigma_{i_o+1} )p}}
    \Big(\frac{R}{R-r}\Big)^{(i_o+1)(N+sp+1)}
    \boldsymbol{\mathfrak K}^p.
$$
Note that \eqref{est:tauh-tauh-io} plays the role of the hypothesis \eqref{ass:Domokos}$_{\sig>1}$ in Lemma \ref{lem:Domokos}.  Using also the $L^\infty$-bound for $u$ to estimate the integral of $|u|^p$ on $B_{\rho_{i_o}}$ in terms of $\boldsymbol{\mathfrak K}^p$, the definition of $d_{i_o}$, and $p\le (i_o+2)(N+sp+1)$ we get
\begin{align*}
    \int_{B_{\rho_{i_o}}} |\btau_h u|^p \,\dx
    &\le
    C(p)|h|^p
    \bigg[
    \Big(\frac{M}{s+\alpha \sigma_{i_o+1}-1}\Big)^p
    d_{i_o}^{(s+\alpha \sigma_{i_o+1}-1)p}
    +
    \frac{1}{d_{i_{o}}^p}\int_{B_{\rho_{i_o}}}|u|^p\,\dx
    \bigg]\\
    &\le
    C(p)|h|^p
    \bigg[
    \frac{M^p}{(\tfrac1{20}sp)^p}
    R^{(s+\alpha \sigma_{i_o+1}-1)p}
    +
    \Big(\frac{7\cdot 2^{i_o+2}}{R-r}\Big)^{p}2^N \boldsymbol{\mathfrak K}^p
    \bigg]\\
    &\le
    \frac{28^p2^{N}C(p)}{s^p}\Big(\frac{|h|}{R}\Big)^p\Big(\frac{R}{R-r}\Big)^{(i_o+1)(N+sp+1)}\big[\mathfrak C_{i_o}+  2^{i_op}\big]\boldsymbol{\mathfrak K}^p\\
    &\le
    \frac{28^p2^{N+1}C(p)\mathfrak C_{i_o}}{s^p}
    \Big(\frac{|h|}{R}\Big)^p\Big(\frac{R}{R-r}\Big)^{\frac{4}{sp}(N+sp+1)}\boldsymbol{\mathfrak K}^p\\
     &\le
    C_\ast \mathfrak C_{i_o}
    \Big(\frac{|h|}{R}\Big)^p\Big(\frac{R}{R-r}\Big)^{\frac{4}{sp}(N+sp+1)}\boldsymbol{\mathfrak K}^p
\end{align*}
for every $0<|h|\le d_{i_o}$.
Here, to obtain the second last line we used $i_o+1<\frac{4}{sp}$, which follows from~\eqref{def-io}, since
$$
    i_o+1
    < 
    \frac{4(1-s)}{sp}+2
    =
    \frac{4}{sp}-\frac{4-2p}{p}
    \le 
    \frac{4}{sp}.
$$
The constant in the final estimate above has the form
$$
    C_\ast \mathfrak C_{i_o}
    =
    \mathfrak D_{i_o}
    =
    2\boldsymbol c^{i_o}\boldsymbol b^{1+2+\dots +i_o}C_\ast^{2i_o}\widetilde C_o
    \le 
    2\boldsymbol c^{\frac{4}{sp}-1} 
    \boldsymbol b^{\frac2{sp}(\frac4{sp}-1)}
    C_\ast^{2(\frac4{sp}-1)}\widetilde C_o
    =:
    C(N,p,s).
$$
We denote this constant by $C$. Note that $C$ blows up as $s\downarrow 0$. Moreover, $C$ blows up as $p\downarrow 1$ and remains stable as $p\uparrow 2$. It is also stable as $s\uparrow s_o$. Taking into account that $\rho_{i_o}\ge r$, this proves inequality~\eqref{est:tau_hu^p-p<2} in the case $i_o\ge 2$ for any $0<|h|\le d_{i_o}$.

Finally, it remains to establish~\eqref{est:tau_hu^p-p<2} when $d_{i_o}<|h|\le R-r $. However, in this case we have $\frac{R}{|h|}<\frac{R}{d_{i_o}}=7\cdot 2^{i_o+2} \frac{R}{R-r}$ and hence
\begin{align*}
    \int_{B_r} |\btau_h u|^p \,\dx 
    \le 
    C r^N \|u\|_{L^\infty(B_r)}
    \le 
    2^{i_o}C\Big(\frac{|h|}{R}\Big)^p \Big(\frac{R}{R-r}\Big)^{p} 
    \boldsymbol{\mathfrak K}^{p},
\end{align*}
where $C=C(N)$. Recalling that $p\le \frac{4}{sp}(N+sp+1)$ and $i_o<\frac{4}{sp}-1$ this proves~\eqref{est:tau_hu^p-p<2} and  finishes the proof of Lemma~\ref{lem:diff-quot-p<2}.
\end{proof}

Combination of Lemma \ref{lem:diff-quot-p<2} with the standard estimate for difference quotients from Lemma \ref{lem:diff-quot-1} immediately yields that $(s,p)$-harmonic functions are of class $W^{1,p}_{\rm loc}(\Omega)$. Theorem~\ref{thm:W1p-p<2} is an immediate consequence of Corollary~\ref{cor:W1p-p<2} applied with the choice $r=\frac12 R$.

\begin{corollary}[$W^{1,p}$-regularity]\label{cor:W1p-p<2}
Let $p\in (1,2]$ and $s\in(0,1)$. Then,  whenever $u$ is a locally bounded $(s,p)$-harmonic function in the sense of Definition~\ref{def:loc-sol}, we have
$$
    u\in W^{1,p}_{\rm loc}(\Omega).
$$
Moreover, there exists a constant $C=C(N,p,s)$ such that for any $B_{R}\equiv B_{R}(x_o)\Subset \Omega$ and $r\in(0,R)$ the quantitative $W^{1,p}$-estimate
\begin{align*}
    \int_{B_r} |\nabla u|^p \,\dx 
    &\le 
    \frac{C}{R^{p}} \Big(\frac{R}{R-r}\Big)^{\frac4{sp}(N+sp+1)} 
    \boldsymbol{\mathfrak K}^{p}
\end{align*}
holds true, where $\boldsymbol{\mathfrak K}$ is defined in~\eqref{def:K<2}. 
Moreover, the constant $C$ is stable as $s\uparrow1$. It blows up as $s\downarrow 0$ or $p\downarrow 1$.
\end{corollary}

\section{Higher gradient regularity}\label{sec:W1q}

In the previous section we established that the weak gradient of $(s,p)$-harmonic functions exists in $L^p_{\rm loc}$. 
The aim in this section is to improve the integrability of the gradient $\nabla u$ to any exponent $q\ge p$. This will be achieved by a Moser-type iteration argument. The following lemma is the first step in this direction.

\begin{lemma}\label{lem:second-diff-theta-p<2-new}
Let $p\in(1,2]$, $s\in(0,1)$, and $q\in [p,\infty)$. For any 
\begin{equation}\label{def:eps-q}
     \eps \in(0, 1-s)
\end{equation}
there exists a constant $C=C(N,p,s,q,\eps )$ such that whenever $u$ is a locally bounded $(s,p)$-harmonic function in the sense of Definition~\ref{def:loc-sol} that satisfies
$$
u\in W^{1,q}_{\rm loc}(\Omega),
$$
then for any ball $B_{R}\equiv B_{R}(x_o)\Subset \Omega$ and any $r\in(0,R)$ we have
\begin{align*}
    \int_{B_{r}} 
    |\btau_h\btau_hu|^q \,\dx
    &\le 
    C\Big(\frac{|h|}{R}\Big)^{q+sp\frac{p}{2}-\eps p(1-\frac{p}{2})} 
    \Big(\frac{R}{R-r}\Big)^{N+sp +1}
   \boldsymbol{\mathfrak K}^q,
\end{align*}
where
\begin{align}\label{def:K-2diff}
    \boldsymbol{\mathfrak K}^{q}
    :=
    R^{q}\|\nabla u\|^{q}_{L^q(B_R)} +
    R^{N}\boldsymbol{\mathfrak T}(u;R)^{q}.
\end{align}
The constant $C$ has the form
\begin{equation}\label{structure-C}
    C
    = 
    \frac{\widetilde C(N,p,q)}{s}
    \Big(\frac{1-s}{\eps}\Big)^{1-\frac{p}{2}}.
\end{equation}
\end{lemma}

\begin{proof}
We apply the energy inequality from Proposition \ref{prop:energy-q-1} with $r,R,d$ replaced by $\tilde r=\frac17(5r+2R)$,
$\widetilde R=\frac17(r+6R)$, $d=\frac14(\widetilde R-\tilde r)=\frac17(R-r)$. Note that $\widetilde R+d=R$ and 
$\frac{\widetilde R}{\widetilde R-\widetilde r}$
and $\frac1{\widetilde R}$ can be bounded in terms of
$\frac{R}{R-r}$ and $\frac1R$ apart from a multiplicative factor. With these choices Proposition \ref{prop:energy-q-1} yields
\begin{align*}
    \mathbf I
    &:=
    \big[V_\frac{q}{p}(\btau_hu)
    \big]^p_{W^{(1-\eps)(1-\frac{p}{2})+s\frac{p}{2},p}(B_{\widetilde r})} \\
    &\ \le
    \frac{C}{R^{[(1-\eps)(1-\frac{p}{2})+s\frac{p}{2}]p}}
    \Big(\frac{|h|}{R}\Big)^{q-(1-\frac{p}{2})p}
    \Big(\frac{R}{R-r}\Big)^{N+sp+1} 
    \boldsymbol{\mathfrak K}^{q},
\end{align*}
where the constant $C$ has the form
\begin{align*}
        C
        &=
        \frac{\widetilde C (N,p,q)}{s(1-s)^{\frac{p}{2}}\eps^{(1-\frac{p}2)} }.
\end{align*}
Note that $\widetilde C$ blows up as $p\downarrow 1$. 
Now, we apply Lemma~\ref{lem:N-FS} to $w=V_\frac{q}{p}(\btau_hu)$ on $B_{\widetilde r}$ with $q=p$, $d=\frac17 (R-r)$, and 
$$
    \gamma 
    = 
    (1-\eps)(1-\tfrac{p}{2})+s\tfrac{p}{2}
    =
    1-(1-s)\tfrac{p}{2}-\eps (1-\tfrac{p}{2}),
$$
to deduce that
\begin{align*}
    \int_{B_{\tilde r-d}} &
    \big|\btau_\lambda\big(V_{\frac{q}{p}}(\btau_hu) \big)\big|^p \,\dx \\
    &\le 
    C|\lambda|^{\gamma p} 
    \bigg[(1-\gamma)\mathbf I +
    \bigg(\frac{\widetilde R^{(1-\gamma)p}}{d^p}
    +\frac{1}{\gamma d^{\gamma p}}
    \bigg)  
    \int_{B_{\widetilde r}} \big|V_{\frac{q}{p}}(\btau_hu)\big|^p \,\dx\bigg] \\
    &\le 
    C\Big(\frac{|\lambda|}{R}\Big)^{\gamma p} 
    \bigg[(1-\gamma) R^{\gamma p}\mathbf I +
    \frac{1}{\gamma}\Big(\frac{R}{R-r}\Big)^p
    \int_{B_{\widetilde r}} |\btau_hu|^q \,\dx\bigg],
\end{align*}
for any $0<|\lambda|\le d$, where $C=C(N,p)$ stands for the constant from  Lemma~\ref{lem:N-FS}. We reduce the domain of integration on the left-hand side from $B_{\tilde r-d}$ to $B_r$. Using \eqref{def:eps-q} we have $\gamma>s$ and hence $1-\gamma < 1-s$. Moreover, we use the estimate for $\mathbf I$ from above and the standard estimate for difference quotients from Lemma~\ref{lem:diff-quot-2} and the definition of $\boldsymbol{\mathfrak K}$. In this way, we obtain for any $0<|\lambda|\le d$ that
\begin{align*}
    \int_{B_{r}} &
    \big|\btau_\lambda\big(V_{\frac{q}{p}}(\btau_hu) \big)\big|^p \,\dx
    \\
    &\le 
    C\Big(\frac{|\lambda|}{R}\Big)^{\gamma p} 
    \bigg[\frac{1-s}{s(1-s)^{\frac{p}{2}}\eps^{(1-\frac{p}2)}} \Big(\frac{|h|}{R}\Big)^{q-(1-\frac{p}{2})p}
    \Big(\frac{R}{R-r}\Big)^{N+sp+1} 
    \boldsymbol{\mathfrak K}^{q} \\
    &\qquad\qquad\qquad\quad +
    \frac{1}{s}\Big(\frac{|h|}{R}\Big)^q \Big(\frac{R}{R-r}\Big)^p
    R^q\int_{B_{\widetilde r+d}} |\nabla u|^q \,\dx\bigg] \\
    &\le 
    \frac{C}{s}\Big(\frac{|\lambda|}{R}\Big)^{\gamma p} \Big(\frac{R}{R-r}\Big)^{N+sp+1}
    \bigg[\frac{(1-s)^{(1-\frac{p}2)}}{\eps^{(1-\frac{p}2)}} \Big(\frac{|h|}{R}\Big)^{q-(1-\frac{p}{2})p}
    \boldsymbol{\mathfrak K}^{q} +
    \Big(\frac{|h|}{R}\Big)^q 
    \boldsymbol{\mathfrak K}^{q} \bigg] \\
    &\le 
    \frac{C}{s} \Big(\frac{1-s}{\epsilon}\Big)^{1-\frac{p}2}
    \Big(\frac{|\lambda|}{R}\Big)^{\gamma p}
    \Big(\frac{|h|}{R}\Big)^{q-(1-\frac{p}{2})p}
    \Big(\frac{R}{R-r}\Big)^{N+sp+1}
    \boldsymbol{\mathfrak K}^{q} ,
\end{align*}
where $C=C(N,p,q)$. In turn, we also used $p\le N+sp+1$. 
Here we choose $\lambda=h$ and compute the exponent 
$$
    \gamma p + q-\big(1-\tfrac{p}{2}\big)p
    =
    \Big[1-(1-s)\tfrac{p}{2}-\eps \big(1-\tfrac{p}{2}\big)\Big]p
    + q-\big(1-\tfrac{p}{2}\big)p
    =
    q+sp\tfrac{p}{2}-\eps p\big(1-\tfrac{p}{2}\big).
$$
Since $\frac{q}{p}\ge 1$, we have
\begin{align*}
    \big|\btau_h\big(V_{\frac{q}{p}}(\btau_hu) \big)\big| 
    \ge 
    |\btau_h(\btau_h u)|^{\frac{q}{p}}
    \qquad \mbox{in $B_{r}$,}
\end{align*}
so that
\begin{align*}
    \int_{B_{r}} 
    |\btau_h\btau_hu|^q \,\dx
    &\le 
    \frac{C}{s} \Big(\frac{1-s}{\epsilon}\Big)^{1-\frac{p}2}
    \Big(\frac{|h|}{R}\Big)^{q+sp\frac{p}{2}-\eps p(1-\frac{p}{2})} 
    \Big(\frac{R}{R-r}\Big)^{N+sp +1}
   \boldsymbol{\mathfrak K}^q,
\end{align*}
which proves the claim.
\end{proof}

The next proposition is an application of Lemma~\ref{lem:second-diff-theta-p<2-new}. The estimate for second finite differences is translated into fractional differentiability of the gradient. 

\begin{proposition}\label{prop:grad-frac-1}
Let $p\in (1,2]$ and  $s\in(0,1)$ and $q\in[p,\infty)$. Then, whenever $u$ is a locally bounded $(s,p)$-harmonic function in the sense of Definition~\ref{def:loc-sol} that satisfies
\[
    u\in W^{1,q}_{\rm loc}(\Omega),
\]
we have
$$
    \nabla u\in W^{\alpha,q}_{\rm loc}(\Omega)\qquad 
    \mbox{for any $\alpha\in(0,\beta)$, where $\beta:=\tfrac{sp}{2}\cdot\frac{p}{q}$.}
$$
Moreover, there exists a constant $C=C(N,p,s,q,\alpha)$ such that for any ball $B_{R}\equiv B_{R}(x_o)\Subset \Omega$ and for any $r\in(0,R)$, we have
\begin{align*}
    [\nabla u]_{W^{\alpha, q}( B_{r})}^{q}
    \le
    \frac{C}{R^{(1+\alpha) q}} 
    \Big(\frac{R}{R-r}\Big)^{N+q+sp}
    \boldsymbol{\mathfrak K}^{q},
\end{align*}
where $\boldsymbol{\mathfrak K}$ is defined in~\eqref{def:K-2diff} and the constant is of the form $C=
\frac{\widetilde C(N,p,q)}{s\alpha\beta^{q}(\beta-\alpha)^2 (1-\alpha)^q}$.
\end{proposition}

\begin{proof}
We apply Lemma~\ref{lem:second-diff-theta-p<2-new} with $r$ replaced by $\tilde r=\frac{1}{11}(7r+4R)$. We leave the larger radius $R$ unchanged in the application. 
Taking into account $R-\tilde r=\frac{7}{11}(R-r)$, we obtain for any $\epsilon\in (0,1-s)$ that
\begin{align}\label{est:tau_htau_hu-p<2}
    \int_{B_{\tilde r}}\big| \btau_h(\btau_hu)\big|^{q}\,\dx
    \le 
    C\Big(\frac{|h|}{R}\Big)^{q+sp\frac{p}{2}-\eps p(1-\frac{p}{2})}
    \Big(\frac{R}{R-r}\Big)^{N+sp+1}
    \boldsymbol{\mathfrak K}^{q},
\end{align}
for any $0<|h|\le d=\frac17(R-\tilde r)=\frac{1}{11}(R-r)$, and where $C$ has the same structure as the constant in \eqref{structure-C} from Lemma \ref{lem:second-diff-theta-p<2-new}. We choose some $\alpha\in(0,\beta)$ and let 
\begin{equation*}
    \eps 
    :=
    \min\big\{ \tfrac12 (1-s), \tfrac{1}{2-p}\tfrac{q}{p}(\beta-\alpha)\big\},
\end{equation*}
so that 
\begin{equation*}
    sp\tfrac{p}{2}-\eps p(1-\tfrac{p}{2})
    \ge 
    \tfrac12(\alpha+\beta) q
    =:
    \tilde\beta q.
\end{equation*}
Reducing the exponent of $|h|$ in  \eqref{est:tau_htau_hu-p<2} if necessary, which is allowed due to the inequality $|h|\le R$, we find
\begin{align}\label{est:tau_htau_hu-p<2-beta}
    \int_{B_{\tilde r}}\big| \btau_h(\btau_hu)\big|^{q}\,\dx
    &\le 
     C_1
    \Big(\frac{|h|}{R}\Big)^{(1+\tilde\beta)q}
    \Big(\frac{R}{R-r}\Big)^{N+sp+1}\boldsymbol{\mathfrak K}^{q}.
\end{align}
The constant $C_1$ takes the form
$$
    C_1
    =
    \frac{\widetilde C(N,p,q)}{s}
    \bigg(\frac{1-s}{\min\big\{\frac12(1-s), \tfrac{1}{2-p}\tfrac{q}{p}(\beta-\alpha)\big\}}\bigg)^{1-\frac{p}{2}}
    \le 
    \frac{\widetilde C(N,p,q)}{s(\beta-\alpha)}.
$$
Estimate \eqref{est:tau_htau_hu-p<2-beta} plays the role of assumption \eqref{ass:W^beta,q-second-diff} in Lemma \ref{lem:2nd-Ni-FS} with 
$$
    M^q
    :=
    \frac{C_1}{R^{(1+\tilde\beta)q}}
    \Big(\frac{R}{R-r}\Big)^{N+sp+1}
    \boldsymbol{\mathfrak K}^{q} ,
$$
which we apply on $B_{\tilde r}$; note that $\tilde r=r+4d$. The other parameters of Lemma \ref{lem:2nd-Ni-FS} are fixed by $\gamma= \tilde\beta$, $q=q$, $d=\frac1{11}(R-r)$. Finally, $\alpha$ takes the role of $\beta$; we have $\tilde\beta -\alpha=\frac12 (\beta-\alpha)$. The application is allowed due to the assumption $W^{\alpha,q}_{\rm loc}(\Omega)$. Lemma~\ref{lem:2nd-Ni-FS} ensures that $\nabla u\in W^{\alpha, q}( B_{r})$ with the quantitative estimate 
\begin{align*}
    &[\nabla u]_{W^{\alpha, q}( B_{r})}^{q} \\
    &\qquad\le
    \frac{C(N,q)\,d^{(\tilde\beta-\alpha) q}}{\frac12  (\beta-\alpha) \tilde\beta^{q}(1-\tilde\beta)^{q}}
    \bigg[
    M^{q} +
    \frac{\tilde r^{q}}{\alpha d^{(1+\tilde \beta)q}}\int_{B_{\tilde r}}|\nabla u|^{q}\,\dx\bigg]\\
    &\qquad\le
    \frac{C(N,q)}{\alpha(\beta-\alpha) \beta^{q}(1-\alpha)^q}
    \Bigg[
    \frac{C_1 d^{(\tilde\beta-\alpha) q}}{R^{(1+\tilde\beta) q}} \Big(\frac{R}{R-r}\Big)^{N+sp+1}
    \boldsymbol{\mathfrak K}^{q}
    +
    \frac{1}{d^{(1+ \alpha)q}}
    \boldsymbol{\mathfrak K}^{q}
    \Bigg]\\
    &\qquad\le
    \frac{C(N,q)}{\alpha(\beta-\alpha) \beta^{q}(1-\alpha)^q}
    \bigg[
     \frac{C_1}{R^{(1+\alpha)q}}\Big(\frac{R}{R-r}\Big)^{N+sp+1}
    +
    \frac{1}{d^{(1+ \alpha)q}}
     \bigg] \boldsymbol{\mathfrak K}^{q}\\
    &\qquad\le
    \frac{C(N,q)C_1}{\alpha(\beta-\alpha) \beta^{q}(1-\alpha)^q R^{(1+\alpha)q}}
    \bigg[
    \Big(\frac{R}{R-r}\Big)^{N+sp+1}
    +
    \Big(\frac{R}{R-r}\Big)^{(1+\alpha)q}
    \bigg]\boldsymbol{\mathfrak K}^{q}\\
    &\qquad\le
    \frac{C(N,q)C_1 }{\alpha(\beta-\alpha) \beta^{q}(1-\alpha)^q R^{(1+\alpha)q}}
    \Big(\frac{R}{R-r}\Big)^{N+q+sp}
    \boldsymbol{\mathfrak K}^{q}\\
    &\qquad\equiv
    \frac{C_2}{R^{(1+\alpha)q}}
    \Big(\frac{R}{R-r}\Big)^{N+q+sp}
    \boldsymbol{\mathfrak K}^{q}.
\end{align*}
Here to obtain the second-to-last line we used $(1+\alpha)q<N+sp+q$.  Note that the constant $C_2$ takes the form
$$
     C_2
     =
    \frac{\widetilde C(N,p,q)}{s\alpha(\beta-\alpha)^2 \beta^{q}(1-\alpha)^q}.
$$
This finishes the proof of the proposition.
\end{proof}

By the Sobolev embedding for fractional Sobolev spaces the fractional differentiability of the gradient obtained in Proposition~\ref{prop:grad-frac-1} leads to higher integrability of the gradient.

\begin{lemma}\label{lem:increase-exp-p<2}
Let $p\in (1,2]$ and  $s\in(0,1)$ and $q\in [p,\infty)$. There exists a constant $C$ depending on $N$, $p$, $s$ and $q$, such that whenever $u$ is a locally bounded $(s,p)$-harmonic function in the sense of Definition~\ref{def:loc-sol}, that satisfies
\[
    u\in W^{1,q}(B_R)
\]
for some $B_{R}\equiv B_{R}(x_o)\Subset \Omega$, then for any $r\in(0,R)$ we have
\begin{align*}
    \bigg[\mint_{B_{r}}| \nabla u|^{\frac{Nq}{N-\alpha q}}\,\dx\bigg]^{\frac{N-\alpha q}{N}}
    &\le
    C\Big(\frac{R}{r}\Big)^N \Big(\frac{R}{R-r}\Big)^{N+q+sp}
    \boldsymbol{\mathfrak M}^q,
\end{align*}
where $\alpha=\frac{sp^2}{4 q}$ and
\begin{align*}
    \boldsymbol{\mathfrak M}
    :=
    \bigg[\mint_{B_{R}}|\nabla u|^{q}\,\dx\bigg]^\frac{1}{q} +
    \frac{1}{R}\boldsymbol{\mathfrak T}(u;R) .
\end{align*}
The constant $C$ has the form
$C=\frac{\widetilde C(N,p,q)}{s^{q+4}}$.
\end{lemma}

\begin{proof}
We apply Proposition~\ref{prop:grad-frac-1} with the choice $\alpha=\tfrac{sp^2}{4q}\in (0,\beta)$, where $\beta=\tfrac{sp^2}{2q}$ and infer that $\nabla u\in W^{\alpha, q}_{\rm loc}(\Omega)$. Moreover, the quantitative estimate 
\begin{align*}
    [\nabla u]_{W^{\alpha, q}( B_{r})}^{q}
    \le
    C_2 R^{N-\alpha q} 
    \Big(\frac{R}{R-r}\Big)^{N+q+sp}
    \boldsymbol{\mathfrak M}^{q}
\end{align*}
holds. Note that $\boldsymbol{\mathfrak K}^{q}\le C(N)R^{N+q}\boldsymbol{\mathfrak M}^{q}$, where $\boldsymbol{\mathfrak K}$ is defined in~\eqref{def:K-2diff} and the constant is of the form 
$$
    C_2
    =
    \frac{\widetilde C(N,p,q)}{s\alpha\beta^{q}(\beta-\alpha)^2 (1-\alpha)^q}
    \le 
    \frac{\widetilde C(N,p,q)}{s^{q+4}}.
$$
By the Sobolev embedding for fractional Sobolev spaces from Lemma~\ref{lem:frac-Sob-2} -- note that the application is permitted since $N-\alpha q>0$ -- we conclude that $\nabla u\in L^{\frac{Nq}{N-\alpha q}}(B_{r})$ together with the quantitative estimate
\begin{align*}
    \bigg[\mint_{B_{r}}| \nabla u|^{\frac{Nq}{N-\alpha q}}\,\dx\bigg]^{\frac{N-\alpha q}{N}}
    &\le 
    C_{\rm Sob}\bigg[r^{\alpha q-N} [\nabla u]_{W^{\alpha, q}( B_{r})}^{q} +
    \mint_{B_r} |\nabla u|^q \,\dx \bigg]\\
    &\le
    C_{\rm Sob}\bigg[
    C_2 \Big(\frac{R}{r}\Big)^{N-\alpha q}
    \Big(\frac{R}{R-r}\Big)^{N+q+sp}
     +\Big(\frac{R}{r}\Big)^{N}\bigg]
      \boldsymbol{\mathfrak M}^{q}
    \\
    &\le
     2C_{\rm Sob}C_2 \Big(\frac{R}{r}\Big)^{N} \Big(\frac{R}{R-r}\Big)^{N+q+sp}
    \boldsymbol{\mathfrak M}^{q}.
\end{align*}
The  constant $2C_{\rm Sob}C_2$ has the structure
\begin{align*}
    2C_{\rm Sob}C_2
    =
    \frac{C(N,q)(1-\alpha)}{(N-\alpha q)^{q-1}}\cdot
    \frac{\widetilde C(N,p,q)}{s^{q+4}}
    \le
    \frac{\widetilde C(N,p,q)}{s^{q+4}}.
\end{align*}
To obtain the last line we used $N-\alpha q= N-\frac14 sp^2\ge N-1\ge 1$.
This proves the claimed inequality.
\end{proof}

Lemma~\ref{lem:increase-exp-p<2} allows to set up a Moser-type iteration scheme that improves for any given $q\in[p,\infty)$ the regularity of an $(s,p)$-harmonic function from $W^{1,p}_{\rm loc}(\Omega)$ to $W^{1,q}_{\rm loc}(\Omega)$.

\begin{theorem}[$W^{1,q}$-gradient regularity]\label{lem:W1q-p<2}
Let $p\in (1,2]$ and  $s\in(0,1)$. Then, whenever $u$ is a locally bounded $(s,p)$-harmonic function in the sense of Definition~\ref{def:loc-sol},  we have
$$
    u\in W^{1,q}_{\rm loc}(\Omega)\quad \mbox{for any $q\in [p,\infty)$.}
$$
Moreover, for any $q\ge p$ there exists a constant  $C=C(N,p,s,q)$, such that on any ball $B_{R}\equiv B_{R}(x_o)\Subset \Omega$ the quantitative $L^q$-gradient estimate 
\begin{align*}
    \bigg[\mint_{B_{R/2}}| \nabla u|^{q}\,\dx\bigg]^{\frac{1}{q}}
    &\le
    C\Bigg[\bigg[\mint_{B_{R}}|\nabla u|^{p}\,\dx\bigg]^\frac{1}{p} +
    \frac{1}{R}\boldsymbol{\mathfrak T}(u;R)\Bigg]
\end{align*}
holds true. The constant $C$ is stable as $s\uparrow 1$. It blows up as $p\downarrow 1$.
\end{theorem}

\begin{proof}
In view of Theorem~\ref{thm:W1p-p<2} we know that $u\in W^{1,p}_{\rm loc}(\Omega)$. 
Based on the quantitative higher integrability from Lemma \ref{lem:increase-exp-p<2} we set up an iteration argument in order to iteratively increase the integrability exponent. To this end, we define a sequence $(q_i)_{i\in\N_0}$ of exponents, a sequence $(\rho_i)_{i\in\N_0}$ of radii, and the associated sequence of shrinking concentric balls $(B_i)_{i\in\N_0}$ by
\begin{equation*}
    \left\{
    \begin{array}{c}
    \displaystyle
    q_o:=p,
    \qquad
    q_{i}=\frac{Nq_{i-1}}{N-\frac14sp^2}
    =\bigg(\frac{N}{N-\frac14sp^2}\bigg)^{i}p,\\[12pt]
    \displaystyle
    \rho_i:=\frac{R}2 +\frac{R}{2^{i+1}},\qquad B_i=B_{\rho_i}.
     \end{array}
    \right.
\end{equation*}
Clearly $q_i\ge p$ is an increasing sequence and $q_i\to \infty$ as $i\to\infty$. For $i\in \N$ we apply Lemma \ref{lem:increase-exp-p<2} with $r=\rho_i$, $R=\rho_{i-1}$, $q=q_{i-1}$, and $\alpha=\frac14\frac{sp^2}{q_{i-1}}$.
This amounts to
\begin{align}\label{est:int-q_i}
    \bigg[
    \mint_{B_i}|\nabla u|^{q_i}\,\dx
    \bigg]^\frac1{q_i}
    &\le
    C^\frac1{q_{i-1}}
    \Big(\frac{\rho_{i-1}}{\rho_i}\Big)^{\frac{N}{q_{i-1}}}
    \Big(\frac{\rho_{i-1}}{\rho_{i-1}-\rho_{i}}\Big)^{\frac{N+sp}{q_{i-1}}
    +1}\boldsymbol{\mathfrak M}_{i-1},
 \end{align} 
where we abbreviated 
\begin{align*}
    \boldsymbol{\mathfrak M}_{i-1}
    &:=
    \bigg[\mint_{B_{i-1}}|\nabla u|^{q_{i-1}}\,\dx   \bigg]^\frac1{q_{i-1}}
    +
    \frac1{\rho_{i-1}}\boldsymbol{\mathfrak T}(u; \rho_{i-1}).
\end{align*}
To proceed further, we estimate the numerical factors and the tail term in the above inequality. Indeed, we have
$$
    \frac{\rho_{i-1}}{\rho_i}=\frac{\frac12 R+\frac1{2^i}R}{\frac12 R+\frac1{2^{i+1}}R}<2,
    \quad\quad
    \frac{R}{\rho_{i-1}}
    =
    \frac{R}{\frac12 R+\frac1{2^{i}}R}
    \le 2,
$$
and
$$
    \frac{\rho_{i-1}}{\rho_{i-1}-\rho_{i}}
    =
    \frac{\frac12 R+\frac1{2^{i}}R}{\frac1{2^{i}}R-\frac1{2^{i+1}}R}
    =
    2^{i+1}(\tfrac12+\tfrac1{2^{i}})
    \le
    2^{i+2}.
$$
Moreover, by Remark~\ref{rem:t} we have
\begin{align*}
    \boldsymbol{\mathfrak T}(u; \rho_{i-1})
    \le
    C(N)\Big(\frac{R}{\rho_{i-1}}\Big)^{\frac{N}{p-1}} 
    \boldsymbol{\mathfrak T}(u; R)
    \le
    C(N,p) \boldsymbol{\mathfrak T}(u;R).
\end{align*}
Using these estimates in \eqref{est:int-q_i}$_i$, we obtain for any $i\in\N$ that
\begin{align}\label{before-iteration}\nonumber
    \bigg[
    \mint_{B_i}&|\nabla u|^{q_i}\,\dx
    \bigg]^\frac1{q_i}\\\nonumber
    &\le 
    C^\frac1{q_{i-1}}2^\frac{N}{q_{i-1}} \big(2^{i+2}\big)^{\frac{N+sp}{q_{i-1}}
    +1}
    \Bigg[
    \bigg[\mint_{B_{i-1}}|\nabla u|^{q_{i-1}}\,\dx  \bigg]^\frac1{q_{i-1}}
    +\frac{C(N,p)}{R}\boldsymbol{\mathfrak T}(u;R)
    \Bigg]\\\nonumber
    &\le
   C(N,p)\underbrace{2^{\big[\frac{N+sp}{q_{i-1}}
    +1\big](i-1)}}_{\le 2^{(N+2)(i-1)}} C^\frac1{q_{i-1}}\Bigg[
    \bigg[\mint_{B_{i-1}}|\nabla u|^{q_{i-1}}\,\dx \bigg]^\frac1{q_{i-1}}
    +
    \frac1{R}\boldsymbol{\mathfrak T}(u;R)\Bigg]\\
    &=:C_{i-1} \Bigg[\bigg[\mint_{B_{i-1}}|\nabla u|^{q_{i-1}}\,\dx   \bigg]^\frac1{q_{i-1}}
    +
    \frac1{R}\boldsymbol{\mathfrak T}(u;R)\Bigg].
\end{align}
Recalling that
\begin{align*}
    C^\frac1{q_{i-1}}
    &=
    \bigg[
    \frac{\widetilde C(N,p,q_{i-1})}{s^{q_{i-1}+4}}\bigg]^\frac1{q_{i-1}}\le 
    \frac{\widetilde C(N,p,q_{i-1})}{s^5},
\end{align*}
we obtain that $C_{i-1}$ has the structure
$$
    C_{i-1}=
     \frac{2^{(N+2)(i-1)}\widetilde C(N,p,q_{i-1})}{s^5}.
$$
Iterating \eqref{before-iteration} results in
 \begin{align*}   
    \bigg[
    \mint_{B_i}|\nabla u|^{q_i}\,\dx
    \bigg]^\frac1{q_i} 
    &\le 
    C\Bigg[
    \bigg[\mint_{B_{R}}|\nabla u|^{p}\,\dx   \bigg]^\frac1{p}
    +
    \frac{1}{R}\boldsymbol{\mathfrak T}(u;R)\Bigg],
\end{align*} 
where $C=i\prod_{j=0}^{i-1} C_j$.
Since $q_i\to\infty$ as $i\to\infty$ there exists $i_o\in\N$, such that $q_{i_o-1}<q\le q_{i_o}$. This fixes $i_o\in\N$ in dependence on $N$, $p$, $s$, and $q$. Enlarging the domain of integration from $B_{\frac12 R}$ to $B_{i_o}$ and using H\"older's inequality, we finally get
\begin{align*}
    \bigg[
    \mint_{B_{\frac12 R}}|\nabla u|^{q}\,\dx
    \bigg]^\frac1{q}
    &\le
    2^\frac{N}{q}\bigg[
    \mint_{B_{i_o}}|\nabla u|^{q_{i_o}}\,\dx
    \bigg]^\frac1{q_{i_o}} \\
    &\le 
    C\Bigg[
    \bigg[\mint_{B_{R}}|\nabla u|^{p}\,\dx   \bigg]^\frac1{p}
    +
    \frac1{R}\boldsymbol{\mathfrak T}(u;R)\Bigg],
\end{align*} 
where $C=C(N,p,s,q)$. 
\end{proof}

At this point Theorem~\ref{thm:W1q-p<2} can be achieved by combining Theorem~\ref{cor:W1p-p<2} and Proposition~\ref{lem:W1q-p<2}.

\begin{proof}[\textbf{\upshape Proof of Theorem~\ref{thm:W1q-p<2}}]
First we apply Theorem~\ref{lem:W1q-p<2} on the balls $B_{\frac12 R}$ and $B_{\frac34 R}$, which is possible after slightly changing the radii. Subsequently we use Corollary~\ref{cor:W1p-p<2} to estimate the $L^p$ norm of $\nabla u$ and Lemma~\ref{lem:t} to increase  in the tail term the radius $\tfrac34 R$ to $R$. In this way, we get
\begin{align*}
    \|\nabla u\|_{L^q(B_{R/2})}
    &\le
    C R^{\frac{N}{q}} 
    \Big[R^{-\frac{N}{p}} \|\nabla u\|_{L^p(B_{\frac34 R})} +
    \tfrac{1}{R}\boldsymbol{\mathfrak T}\big(u;\tfrac34 R\big)\Big] \\
    &\le
    C R^{\frac{N}{q}-1} 
    \Big[R^{s-\frac{N}{p}} (1-s)^{\frac1p}[u]_{W^{s,p}(B_{R})} +
    \boldsymbol{\mathfrak T}(u;R)\Big].
\end{align*}
This finishes the proof of Theorem~\ref{thm:W1q-p<2}.
\end{proof}

By Morrey embedding the $W^{1,q}$-regularity from Theorem~\ref{thm:W1q-p<2} immediately implies that $(s,p)$-harmonic functions are Hölder continuous for any Hölder exponent $\gamma\in(0,1)$. This is exactly the content of Theorem~\ref{thm:Hoelder-p<2}.

\begin{proof}[\textbf{\upshape Proof of Theorem~\ref{thm:Hoelder-p<2}}]
From Theorem~\ref{lem:W1q-p<2} we know that $u\in W^{1,q}_{\rm loc}(\Omega)$ for any $q\ge p$. Therefore, by Morrey embedding, Lemma~\ref{Lem:morrey-classic}, we conclude that $u\in C^{0,\gamma}_{\rm loc}(\Omega)$ for any $\gamma\in(0,1)$. Now, fix some $\gamma\in(0,1)$ and consider a ball $B_R\equiv B_R(x_o)\Subset\Omega$. Applying in turn Lemma~\ref{Lem:morrey-classic} with the choice $q=\frac{N}{1-\gamma}$ and Theorem~\ref{lem:W1q-p<2}, we obtain the quantitative estimate 
\begin{align*}
    [u]_{C^{0,\gamma}(B_{\frac12 R})}
    &=
    [u]_{C^{0,1-\frac{N}{q}}(B_{\frac12 R})}
    \le 
    C\|\nabla u\|_{L^{q}(B_{\frac12 R})} \\
    &\le
    C R^{\frac{N}{q}-1} 
    \Big[R^{s-\frac{N}{p}} (1{-}s)^{\frac1p}[u]_{W^{s,p}(B_{R})} +
    \boldsymbol{\mathfrak T}(u;R)\Big].
\end{align*}
Recalling the choice of $q$, we conclude the claimed inequality.
\end{proof}

\section{Fractional differentiability of the gradient}\label{sec:fracdiff}

In this section we prove, as stated in Theorem~\ref{thm:grad-frac}, that the gradient of $(s,p)$-harmonic functions exhibits a certain fractional differentiability  at each $L^q$-scale with $q\ge 2$. Particularly, for $q=2$ we have  $\nabla u\in W^{\alpha ,2}_{\rm loc}(\Omega)$ for any $0<\alpha <\min\{\frac12 sp, 1-(1-s)p\}$ with a quantitative local estimate which remains stable as $s\uparrow 1$. It is straightforward to check that the upper bound of $\al$ tends to $1$ as $s\uparrow1$.
Therefore,  we formally recover the $W^{2,2}$-regularity obtained in \cite{Manfredi-Weitsman} and \cite{Acerbi-Fusco} for $p$-harmonic functions in the singular case $p\in (1,2)$. 

We start with two preparatory lemmata. The first one ensures that $\nabla u\in W^{\alpha,q}_{\loc}$ for any $q\ge p$ and $\alpha<\frac{sp}{q}$.

\begin{lemma}\label{lem:tauhh-inproved-basic}
Let $p\in (1,2]$, $s\in (0, 1)$, $q\in[2,\infty)$, $\varep\in (0,\min\{1-s,\frac12 sp\}]$, and $\tilde\varep\in (0,\frac12 sp)$.
Whenever
$u$ is a locally bounded $(s,p)$-harmonic function in the sense of Definition~\ref{def:loc-sol}, we have
$$
    \nabla u\in W^{\be,q}_{\loc}(\Om)\quad\text{where}\quad
    \be:=\frac{sp-\varep-\tilde\epsilon}{q}\in (0,1).
$$
Moreover, there exists $C=C(N,p,s,q,\tilde\varep)$, such that for every ball
$B_{R}\equiv B_{R}(x_o)\Subset\Omega$ and $r\in(0,R)$, we have 
\begin{align*}
     [\nabla u]_{W^{\beta,q}(B_r)}^q
    &\le \frac{C}{R^{(1+\be) q}}\Big(\frac{R}{R-r}\Big)^{N+q+1}\boldsymbol{\mathfrak K}^q,
\end{align*}
where
\begin{align*}
    \boldsymbol{\mathfrak K}^q
    &:=
    R^{N+(1-\varep) q}[u]^q_{C^{0,1-\varep}(B_{R})} +
    R^q  \|\nabla u \|^q_{L^q(B_R)} +
    R^N \boldsymbol{\mathfrak T}(u;R)^q .
\end{align*}
\end{lemma}

\begin{proof}
By Theorem~\ref{thm:W1q-p<2} we have $u\in W^{1,q}_{\rm loc}(\Omega)$. Moreover, from Theorem \ref{thm:Hoelder-p<2} we know that $u\in C_{\rm loc}^{0,1-\varep}(\Omega)$ for any $\varep\in(0,1)$. For $\varep$ as in 
the statement, we set $\gamma:=1-\varep$ for ease of notation. 
Fix $0<r<R$ and introduce $\tilde r=\frac17(5r+2R)$,
$\widetilde R=\frac17(r+6R)$, and $d=\frac14(\widetilde R-\tilde r)=\frac17(R-r)$. 
First of all, by Hölder's inequality, we have
\begin{equation}\label{eq-holder}
    \int_{B_{R}} |\nabla u|^{q+p-2}\,\dx
    \le 
    C(N)R^{\frac{N(2-p)}{q}}\bigg[\int_{B_{R}} |\nabla u|^{q}\,\dx\bigg]^{\frac{q+p-2}{q}}.
\end{equation}
Next, we apply Corollary~\ref{cor:energy-q-1} with $(\tilde r, \widetilde R,\vartheta,q+p-2,1-\epsilon)$ in place of $(r,R,\theta,q,\gamma)$. Note that $\widetilde R+d=R$ and that
$\frac{\widetilde R}{\widetilde R-\tilde r}$
and $\frac1{\widetilde R}$ can be bounded in terms of
$\frac{R}{R-r}$ and $\frac1R$ apart from a multiplicative factor, use~\eqref{eq-holder}, and the definition of $\boldsymbol{\mathfrak K}$ to obtain that 
\begin{align*}
    \mathbf{I}
    &:= 
    \big[
    V_\frac{q}{2}(\btau_hu)
    \big]^2_{W^{\gamma-\frac{1}{2}(\gamma-s)p,2}(B_{\tilde r})}
    \\
    &\ \le 
    \frac{C}{(1-s)R^{sp}} \Big(\frac{|h|}{R}\Big)^{q+p-2}
    \Big(\frac{R}{R-r}\Big)^{N+3}
    [u]_{C^{0,\gamma}(B_{R})}^{2-p} \\
    &\ \phantom{\le\,}\cdot
    \Bigg[ R^{\frac{N(2-p)}{q}+q+p-2}  
    \bigg[\int_{B_{R}} |\nabla u|^{q}\,\dx\bigg]^{\frac{q+p-2}{q}} +
    R^N \boldsymbol{\mathfrak T}(u;R)^{q+p-2}
    \Bigg] \\
    &\ \le 
    \frac{C}{(1-s) R^{sp}} \Big(\frac{|h|}{R}\Big)^{q+p-2}
    \Big(\frac{R}{R-r}\Big)^{N+3}
    [u]_{C^{0,\gamma}(B_{R})}^{2-p} 
    R^{\frac{N(2-p)}{q}} \boldsymbol{\mathfrak K}^{q+p-2}.
\end{align*}
Estimating $[u]_{C^{0,\gamma}(B_{R})}$ in terms of $R^{-(\frac{N}{q}+\gamma)}\boldsymbol{\mathfrak K}$, we obtain 
\begin{align}\label{est:I-fract}
    \mathbf{I}
    &\le 
    \frac{C}{(1-s)R^{sp+\gamma(2-p)}} \Big(\frac{|h|}{R}\Big)^{q+p-2}
    \Big(\frac{R}{R-r}\Big)^{N+3}
    \boldsymbol{\mathfrak K}^{q}.
\end{align}
Here, the constant $C$ has the form $C=\widetilde C(N,p,q)/s$.
Let us apply Lemma~\ref{lem:N-FS} to $w=V_\frac{q}{2}(\btau_hu)$ on $B_{\tilde r}$ with $(q,d,\gm)$ in that lemma replaced by $(2,d=\frac17 (R-r),\widetilde\gm)$, where
\begin{equation}\label{def:tildegm}
   \widetilde\gm= \gamma-\tfrac{1}{2}(\gamma-s)p
   =
   \big( 1-\tfrac12 p\big)\gm +\tfrac12 ps\in [\tfrac12 ps,\gm ],
\end{equation}
and deduce that
\begin{align*}
    \int_{B_{\tilde r-d}} &
    \big|\btau_\lm \big(V_{\frac{q}{2}}(\btau_hu)\big) \big|^2 \,\dx \\
    &\le 
    C(N)|\lm |^{2\widetilde\gamma} 
    \Bigg[(1-\widetilde\gamma)\mathbf I +
    \bigg(\frac{\tilde r^{2(1-\widetilde\gamma)}}{d^2}
    +\frac{1}{\widetilde\gamma d^{2\widetilde\gamma }}
    \bigg)  
    \int_{B_{\widetilde r}} \big|V_{\frac{q}{2}}(\btau_hu)\big|^2 \,\dx\Bigg]
\end{align*}
for any $0<|\lm |\le d$. Since $\varep <1-s$  we have
\begin{equation}\label{est-tildegm}
    1-\widetilde\gamma 
    =
    \tfrac12p(1-s)+(1-\tfrac12p)\epsilon
    \le 
    1-s.
\end{equation}
Moreover, we observe $\tilde r-d=r+d$. Therefore,  the last inequality yields
\begin{align*}
    \int_{B_{r+d}} &
    \big|\btau_\lm \big(V_{\frac{q}{2}}(\btau_hu)\big) \big|^2 \,\dx
    \\
    &\le 
    C(N)|\lm |^{2\gamma -(\gm-s)p} 
    \bigg[(1-s)\mathbf I+
    \frac{1}{s R^{sp+\gamma(2-p)}} \Big(\frac{R}{R-r}\Big)^2
    \int_{B_{\widetilde r}} |\btau_hu|^q \,\dx\bigg].
\end{align*}
The integral on the right-hand side is estimated by Lemma \ref{lem:diff-quot-2}, that is
\begin{equation*}
    \int_{B_{\widetilde r}} |\btau_hu|^q \,\dx
    \le \Big(\frac{|h|}{R}\Big)^qR^q\int_{B_{ R}} |\nabla u|^q \,\dx
    \le \Big(\frac{|h|}{R}\Big)^{q+p-2}\boldsymbol{\mathfrak K}^{q},
\end{equation*}
while $\mathbf I$ is bounded by~\eqref{est:I-fract}. Consequently, we obtain
\begin{align*}
    \int_{B_{r+d}} 
    \big|\btau_\lm \big(V_{\frac{q}{2}}(\btau_hu)\big) \big|^2 \,\dx
    \le 
    \frac{C}{s}
    \Big(\frac{|\lm|}{R}\Big)^{sp +\gamma(2-p)} 
    \Big(\frac{|h|}{R}\Big)^{q+p-2} 
    \Big(\frac{R}{R-r}\Big)^{N+3} 
    \boldsymbol{\mathfrak K}^{q} ,
\end{align*}
where $C=C(N,p,q)$. 
In the above estimate, we choose $\lambda=h$. Furthermore, since $q\ge 2$, we have
\begin{align*}
    \big|\btau_h\big(V_{\frac{q}{2}}(\btau_hu) \big)\big| 
    \ge 
    |\btau_h(\btau_h u)|^{\frac{q}{2}}
    \qquad \mbox{in $B_{r+d}$.}
\end{align*}
Therefore, for any $0<|h|\le d$ we arrive at
\begin{align*}
    \int_{B_{r+d}} |\btau_h(\btau_hu)|^q \,\dx 
    &\le
    \frac{C}{s} \Big(\frac{|h|}{R}\Big)^{q+p-2 +sp+\gamma(2-p)}
    \Big(\frac{R}{R-r}\Big)^{N+3} 
    \boldsymbol{\mathfrak K}^q \\
    &=
    \frac{C}{s} \Big(\frac{|h|}{R}\Big)^{q +sp-\epsilon(2-p)}
    \Big(\frac{R}{R-r}\Big)^{N+3} 
    \boldsymbol{\mathfrak K}^q \\
    &\le 
    \frac{C}{s} \Big(\frac{|h|}{R}\Big)^{q +sp-\epsilon-\frac12\tilde\epsilon}
    \Big(\frac{R}{R-r}\Big)^{N+3} 
    \boldsymbol{\mathfrak K}^q,
\end{align*}
where $C=C(N,p,q)$. 
Apply Lemma~\ref{lem:2nd-Ni-FS} with $\gm$ replaced by 
$$
    \widetilde\beta
    =
    \frac{sp-\varep-\frac12\tilde\epsilon}{q}
$$
and with
\begin{align*}
    M^q 
    &= \frac{C}{s} \frac{1}{R^{q(1+\widetilde\beta)}}
    \Big(\frac{R}{R-r}\Big)^{N+3}
    \boldsymbol{\mathfrak K}^q
\end{align*} 
and $\frac14 d$ in place of $d$. 
Moreover, in the application we fix $\beta$  by
\begin{align*}
    \beta:= \frac{sp-\varep-\tilde\epsilon}{q},
\end{align*}
such that $\beta \in (0,\widetilde\beta)$ and $\widetilde\beta -\beta =\frac{\tilde\epsilon}{2q}$. In this setup, Lemma~\ref{lem:2nd-Ni-FS} yields
that $\nabla u\in W^{\be,q}(B_r)$ with the
quantitative estimate
\begin{align}\label{add}  
[\nabla u]_{W^{\beta,q}(B_r)}^q
    &\le
    \frac{C(N,q)d ^{q(\widetilde\beta -\beta)}}{(\widetilde\beta -\beta) \widetilde\beta ^q (1-\widetilde\beta )^q}
    \bigg[ M^q + \frac{(r+d )^{q}}{\beta d ^{q(1+\widetilde\beta )}} 
    \int_{B_{r+d}}|\nabla u|^q\,\dx
    \bigg].
\end{align}
To simplify the estimate \eqref{add}, we note that $r+d\le R$ and use the elementary inequalities
\begin{equation*}
    \widetilde\beta\ge \frac{sp}{4q},
    \quad
    1-\widetilde \beta\ge \frac{q-sp+\frac12\tilde\epsilon}{q}
    \ge
    \frac{\tilde\epsilon}{2q},
\end{equation*}
so that the denominator of the first factor in \eqref{add} can be bounded from below by
\begin{align*}
    (\widetilde\beta -\beta) \widetilde\beta ^q (1-\widetilde\beta )^q
    &\ge
    \frac{\tilde\eps}{q}\Big(\frac{sp}{4q}\Big)^q\Big(\frac{\tilde\epsilon}{2q}\Big)^q
    =
    C(p,q) s^q \tilde\epsilon^{q+1}.
\end{align*}
Inserting this in the $W^{\beta,q}$-estimate \eqref{add}  will result in 
\begin{align*}
 [\nabla u]_{W^{\beta,q}(B_r)}^q
    &\le \frac{C}{R^{(1+\be) q}}\Big(\frac{R}{R-r}\Big)^{N+q+1}\boldsymbol{\mathfrak K}^q,
\end{align*}
where the constant $C$ has the structure
\[
    C
    =
    \frac{\widetilde C(N,p,q)}{s^{q+1}\tilde\varep^{q+1} }.
\]
This is the desired estimate. Note that
the constant $\widetilde C$ blows up as $p\downarrow 1$, but remains stable for $p=2$. Moreover, the constant remains stable as $s\uparrow 1$.
\end{proof}

In contrast to Lemma~\ref{lem:tauhh-inproved-basic}, in the next result we additionally assume that $u\in W^{1+\theta,q}_{\rm loc}(\Omega)$ for some $\theta\in(0,1)$. If $\theta$ is large enough and $s>\frac{q(p-1)}{p}$, this leads to an improvement in fractional differentiability. Nevertheless, we state the lemma for the larger range $s\in (\frac{p-1}{p}, 1)$. In this larger range, it is guaranteed that the parameter $\beta$, which expresses the fractional differentiability, is positive.

\begin{lemma}\label{lem:tauhh-inproved}
Let $p\in (1,2]$, $s\in (\frac{p-1}{p}, 1)$, $q\in[2,\infty)$, $\theta\in(0,1)$, $\varep\in (0,\min\{1-s,\frac12[1-(1-s)p]\}]$, and $\tilde\epsilon\in(0,\frac12\theta]$.
Whenever
$u$ is a locally bounded $(s,p)$-harmonic function in the sense of Definition~\ref{def:loc-sol} that satisfies
$$
    u\in W^{1+\theta,q}_{\rm loc}(\Omega),
$$
we have
$$
\nabla u\in W^{\be,q}_{\loc}(\Om)\quad\text{where}\quad
\be:=\frac{1-(1-s)p+\theta-\varep-\tilde\epsilon}{q}\in (0,1).
$$
Moreover, there exists $C=C(N,p,s,q,\theta,\tilde\varep)$, such that for every ball
$B_{R}\equiv B_{R}(x_o)\Subset\Omega$ and $r\in(0,R)$, we have 
\begin{align*}
     [\nabla u]_{W^{\beta,q}(B_r)}^q
    &\le \frac{C}{R^{(1+\be) q}}\Big(\frac{R}{R-r}\Big)^{N+q+1}\boldsymbol{\mathfrak K}^q,
\end{align*}
where
\begin{align*}
    \boldsymbol{\mathfrak K}^q
    &:=
    R^{N+(1-\varep) q}[u]^q_{C^{0,1-\varep}(B_{R})} +
    R^{(1+\theta)q}[\nabla u]^q_{W^{\theta,q}(B_{R})} +
    R^q \|\nabla u\|^q_{L^{q}(B_{R})} +
    R^N \boldsymbol{\mathfrak T}^q 
\end{align*}
and
$\boldsymbol{\mathfrak T}:= \boldsymbol{\mathfrak T}(u;R)$.
\end{lemma}

\begin{proof}
By Theorem \ref{thm:Hoelder-p<2} we have $u\in C_{\rm loc}^{0,1-\varep}(\Omega)$. To keep some of the exponents as simple as possible, we use the abbreviations $\gamma :=1-\varep$ and $\vartheta:=\theta-\frac12\tilde\epsilon$.
Let $0<r<R$ and abbreviate $\tilde r=\frac17(5r+2R)$,
$\widetilde R=\frac17(r+6R)$, and $d=\frac14(\widetilde R-\tilde r)=\frac17(R-r)$. 
First of all, we note that for $p\in (1,2)$ and since $\vartheta\in(0,\theta)$, by H\"older's inequality and Lemma \ref{int-sing} we have
\begin{align*}
    &[\nabla u]^{q+p-2}_{W^{\vartheta,q+p-2}(B_{R})}=\iint_{K_{R}}\frac{|\nabla u(x)-\nabla u(y)|^{q+p-2}}{|x-y|^{N+\vartheta(q+p-2)}}\dx\dy\\
    &\quad\le 
    \bigg[
    \iint_{K_{R}}\frac{|\nabla u(x)-\nabla u(y)|^{q}}{|x-y|^{N+\theta q}}\dx\dy
    \bigg]^{\frac{q+p-2}{q}}
    \bigg[
    \iint_{K_{R}}\frac{\dx\dy}{|x-y|^{N-\frac{q(\theta - \vartheta)(q+p-2)}{2-p}}}
    \bigg]^{\frac{2-p}{q}}\\
    &\quad\le 
    \frac{C(N,q)}{\tilde\epsilon} R^{\frac{N(2-p)}{q}+(\theta-\vartheta)(q+p-2)} [\nabla u]^{q+p-2}_{W^{\theta,q}(B_{R})}.
\end{align*}
If $p=2$, the argument is simpler, since $|x-y|^{q(\theta-\vartheta)}$ is bounded by $[2R]^{q(\theta-\vartheta)}$, and thus the exponent of $|x-y|$ can be increased from $N+\vartheta q$ to $N+\theta q$. In any case, the last two inequalities are valid for any $p\in (1,2]$. Apply Proposition~\ref{prop:energy-q-2} with $(\tilde r, \widetilde R,\vartheta,q+p-2$ in place of $(r,R,\theta,q)$ and with $\gm =1-\varep$, note that $\widetilde R+d=R$ and 
$\frac{\widetilde R}{\widetilde R-\widetilde r}$
and $\frac1{\widetilde R}$ can be bounded in terms of
$\frac{R}{R-r}$ and $\frac1R$ apart from a multiplicative factor. Subsequently, we use Hölder's inequality as in~\eqref{eq-holder}, the above estimate, and the definition of $\boldsymbol{\mathfrak K}$ to obtain 
\begin{align*}
    \mathbf{I}
    &:= 
    \big[
    V_\frac{q}{2}(\btau_hu)
    \big]^2_{W^{\gamma-\frac{1}{2}(\gamma-s)p,2}(B_{\tilde r})}
    \\
    &\ \le 
    \frac{C}{(1-s)R^{sp}} \Big(\frac{|h|}{R}\Big)^{q-1+\vartheta}
    \Big(\frac{R}{R-r}\Big)^{N+sp+1}
    [u]_{C^{0,\gamma}(B_{R})}^{2-p} \\
    &\ \phantom{\le\,}
    \cdot
    \Bigg[R^{(1+\vartheta)(q+p-2)}[\nabla u]^{q+p-2}_{W^{\vartheta,q+p-2}(B_{R})} +
    \frac{R^{q+p-2}}{\vartheta} \int_{B_{R}} |\nabla u|^{q+p-2}\,\dx +
    R^N \boldsymbol{\mathfrak T}^{q+p-2}
    \Bigg] \\
    &\ \le 
    \frac{C}{(1-s)R^{sp}} \Big(\frac{|h|}{R}\Big)^{q-1+\vartheta}
    \Big(\frac{R}{R-r}\Big)^{N+3}
    [u]_{C^{0,\gamma}(B_{R})}^{2-p} \\
    &\ \phantom{\le\,}\cdot
    \Bigg[ \frac{R^{\frac{N(2-p)}{q}+(1+\theta)(q+p-2)}}{\tilde\epsilon} [\nabla u]^{q+p-2}_{W^{\theta,q}(B_{R})}\\
    &\ \qquad\,
    +
    \frac{R^{\frac{N(2-p)}{q}+q+p-2}}{\tilde\epsilon}  \bigg[\int_{B_{R}} |\nabla u|^{q}\,\dx\bigg]^{\frac{q+p-2}{q}} +
    R^N \boldsymbol{\mathfrak T}^{q+p-2}
    \Bigg] \\
    &\ \le 
    \frac{C}{(1-s)\tilde\epsilon R^{sp}} \Big(\frac{|h|}{R}\Big)^{q-1+\vartheta}
    \Big(\frac{R}{R-r}\Big)^{N+3}
    [u]_{C^{0,\gamma}(B_{R})}^{2-p} 
    R^{\frac{N(2-p)}{q}} \boldsymbol{\mathfrak K}^{q+p-2}.
\end{align*}
In turn, we used $\vartheta=\theta-\frac12\tilde\epsilon\ge2\tilde\epsilon-\frac12\tilde\epsilon = \frac32\tilde\epsilon$. 
Keeping in mind that  $[u]_{C^{0,\gamma}(B_{R})}$ can be estimated by $R^{-[\frac{N}{q}+\gamma]}\boldsymbol{\mathfrak K}$, we obtain 
\begin{align}\label{est:I-q/2}
    \mathbf{I}
    &\le 
    \frac{C}{(1-s)\tilde\epsilon}
    \frac{1}{R^{sp+\gamma(2-p)}} \Big(\frac{|h|}{R}\Big)^{q-1+\vartheta}
    \Big(\frac{R}{R-r}\Big)^{N+3}
    \boldsymbol{\mathfrak K}^{q}.
\end{align}
The constant $C$ has the structure $C=\widetilde C(N,p,q)/s$ and $\widetilde C$ is stable as $p\uparrow 2$ and blows up as $p\downarrow 1$; see Proposition~\ref{prop:energy-q-2}.

Let us apply Lemma~\ref{lem:N-FS} to $w=V_\frac{q}{2}(\btau_hu)$ on $B_{\widetilde r}$ with $(q,d,\gm)$ replaced by $(2,d=\frac17 (R-r),\widetilde\gm)$, where $\tilde\gamma$ is defined in~\eqref{def:tildegm} and deduce that
\begin{align*}
    \int_{B_{\tilde r-d}} &
    \big|\btau_\lm \big(V_{\frac{q}{2}}(\btau_hu)\big) \big|^2 \,\dx \\
    &\le 
    C(N)|\lm |^{2\widetilde\gamma} 
    \Bigg[(1-\widetilde\gamma)\mathbf I +
    \bigg(\frac{\tilde r^{2(1-\widetilde\gamma)}}{d^2}
    +\frac{1}{\widetilde\gamma d^{2\widetilde\gamma }}
    \bigg)  
    \int_{B_{\tilde r}} \big|V_{\frac{q}{2}}(\btau_hu)\big|^2 \,\dx\Bigg]
\end{align*}
for any $0<|\lm |\le d$. Since $\varep <1-s$ we have $1-\widetilde\gamma \le 1-s$; see~\eqref{est-tildegm}. Moreover, $\tilde r-d=r+d$. 
Therefore,  the last inequality yields
\begin{align*}
    \int_{B_{r+d}} &
    \big|\btau_\lm \big(V_{\frac{q}{2}}(\btau_hu)\big) \big|^2 \,\dx
    \\
    &\le 
    C(N)|\lm |^{2\gamma -(\gm-s)p} 
    \bigg[(1-s)\mathbf I+
    \frac{1}{s R^{sp+\gamma(2-p)}} \Big(\frac{R}{R-r}\Big)^2
    \int_{B_{\widetilde r}} |\btau_hu|^q \,\dx\bigg].
\end{align*}
The integral on the right-hand side is estimated by Lemma \ref{lem:diff-quot-2}, that is
\begin{equation*}
    \int_{B_{\widetilde r}} |\btau_hu|^q \,\dx
    \le \Big(\frac{|h|}{R}\Big)^qR^q\int_{B_{ R}} |\nabla u|^q \,\dx
    \le \Big(\frac{|h|}{R}\Big)^{q-1+\vartheta}\boldsymbol{\mathfrak K}^{q},
\end{equation*}
while $\mathbf I$ is bounded by~\eqref{est:I-q/2}. We obtain
\begin{align*}
    \int_{B_{r+d}} 
    \big|\btau_\lm \big(V_{\frac{q}{2}}(\btau_hu)\big) \big|^2 \,\dx
    \le 
    \frac{C}{s\tilde\epsilon}
    \Big(\frac{|\lm|}{R}\Big)^{sp +\gamma(2-p)} 
    \Big(\frac{|h|}{R}\Big)^{q-1+\vartheta} 
    \Big(\frac{R}{R-r}\Big)^{N+3} 
    \boldsymbol{\mathfrak K}^{q} ,
\end{align*}
where $C=C(N,p,q)$. 
Here we choose $\lambda=h$. Since $q\ge 2$, we have
\begin{align*}
    \big|\btau_h\big(V_{\frac{q}{2}}(\btau_hu) \big)\big| 
    \ge 
    |\btau_h(\btau_h u)|^{\frac{q}{2}}
    \qquad \mbox{in $B_{r+d}$.}
\end{align*}
Therefore, for any $0<|h|\le d$ we get
\begin{align*}
    \int_{B_{r+d}} |\btau_h(\btau_hu)|^q \,\dx 
    \le
    \frac{C}{s\tilde\epsilon} \Big(\frac{|h|}{R}\Big)^{q-1+\vartheta +sp+\gamma(2-p)}
    \Big(\frac{R}{R-r}\Big)^{N+3} 
    \boldsymbol{\mathfrak K}^q,
\end{align*}
for a constant $C=C(N,p,q)$. Recalling $\gamma =1-\epsilon$ and  $\vartheta=\theta-\frac12\tilde\epsilon$, the above estimate can be re-written as
\begin{align*}
    \int_{B_{\widetilde R-d}} |\btau_h(\btau_hu)|^q \,\dx 
    \le
    \frac{C}{s\tilde\varep} \Big(\frac{|h|}{R}\Big)^{q+1 - (1-s)p+\theta -(2-p)\varep -\frac12\tilde\epsilon}
    \Big(\frac{R}{R-r}\Big)^{N+3}\boldsymbol{\mathfrak K}^q. 
\end{align*}
As before, we have $C=C(N,p,q)$.
Apply Lemma~\ref{lem:2nd-Ni-FS} with $\gm$ replaced by 
$$
    \widetilde\beta
    =
    \frac{1-(1-s)p+\theta-(2-p)\varep-\frac12\tilde\epsilon}{q}
$$
and
\begin{align*}
    M^q 
    &= \frac{C}{s\tilde\varep} \frac{1}{R^{q(1+\widetilde\beta)}}
    \Big(\frac{R}{R-r}\Big)^{N+3}\boldsymbol{\mathfrak K}^q
\end{align*} 
and $\bar d=\frac14 d$. 
Moreover, we fix $\beta$ in the application by
\begin{align*}
    \beta:= \frac{1-(1-s)p+\theta-\varep-\tilde\epsilon}{q}.
\end{align*}
Note that $\beta \in (0,\widetilde\beta)$ and $\widetilde\beta -\beta =\frac{(p-1)\varep+\frac12\tilde\epsilon}{q}>\frac{\tilde\varep}{2q}$. Lemma~\ref{lem:2nd-Ni-FS} yields
$\nabla u\in W^{\be,q}(B_r)$ with the
quantitative estimate
\begin{align}\label{add1}
     [\nabla u]_{W^{\beta,q}(B_r)}^q
    &\le
    \frac{C(N,q)d ^{q(\widetilde\beta -\beta)}}{(\widetilde\beta -\beta) \widetilde\beta ^q (1-\widetilde\beta )^q}
    \bigg[ M^q + \frac{(r+d )^{q}}{\beta d ^{q(1+\widetilde\beta )}} 
    \int_{B_{r+d}}|\nabla u|^q\,\dx
    \bigg].
\end{align}
Noting that $r+d\le R$ and using the elementary inequalities
\begin{equation*}
    \widetilde\beta> \beta\ge \frac{\theta}{2q},
    \quad
    1-\widetilde \beta\ge \frac{q-1-\theta}{q},
\end{equation*}
the denominator of the first factor in \eqref{add1} can be bounded from below by
\begin{align*}
    (\widetilde\beta -\beta) \widetilde\beta ^q (1-\widetilde\beta )^q
    \ge
    \frac{\tilde\eps}{2q}\Big(\frac{\theta}{2q}\Big)^q\Big(\frac{q-1-\theta}{q}\Big)^q 
    =
    \frac{\tilde\varep\theta^q(q-1-\theta)^q}{C(q)}.
\end{align*}
Inserting this in the $W^{\beta,q}$-estimate \eqref{add1}  will result in 
\begin{align*}
 [\nabla u]_{W^{\beta,q}(B_r)}^q
    &\le \frac{C}{R^{(1+\be) q}}\Big(\frac{R}{R-r}\Big)^{N+q+1}\boldsymbol{\mathfrak K}^q,
\end{align*}
where the constant $C$ has the structure
\[
    C
    =
    \frac{\widetilde C(N,p,q)}{s\tilde\varep^2\theta^{q+1}(q-\theta-1)^q }.
\]
This is the desired estimate. 
\end{proof}

\begin{remark}\upshape
   We stress that the constant $C$ in the previous lemmas eventually blows up as $\tilde\varep\downarrow 0$. However, it is independent of $\varep$ which is required to be less than $1-s$. Therefore, $C$ remains stable as $s\uparrow1$.
\end{remark}

Now, we have all prerequisites at hand to give the proof of Theorem~\ref{thm:grad-frac}.

\begin{proof}[\textbf{\upshape Proof of Theorem~\ref{thm:grad-frac}}]
Note that
\begin{equation}\label{precise:beta}
    \beta=
    \left\{
    \begin{array}{cl}
    \displaystyle \frac{sp}{q},&\mbox{if $\displaystyle s\le \frac{q(p-1)}{p}$,}\\[8pt]
    \displaystyle\frac{1-(1-s)p}{q-1}, &\mbox{if $\displaystyle s> \frac{q(p-1)}{p}$.}
    \end{array}
    \right.
\end{equation}
We proceed in three steps.

\textit{Step 1: Almost $W^{\frac{sp}{q},q}$-regularity.} This step involves Lemma~\ref{lem:tauhh-inproved-basic} only and applies to all $s\in (0,1)$.
To begin with, let $\alpha\in(0,\frac{sp}{q})$. From Theorem~\ref{lem:W1q-p<2} we have  $u\in W^{1,q}_{\rm loc}(\Omega)$ for any $q\ge p$, while from  Theorem~\ref{thm:Hoelder-p<2} we know that $u\in C^{0,\gamma}_{\rm loc}(\Omega)$ for any $\gamma\in(0,1)$.
This allows  to fix in Lemma~\ref{lem:tauhh-inproved-basic} the parameters $\varep$ and $\tilde\varep$ by
$$
    \epsilon
    =
    \min\big\{1-s,\tfrac12(sp-\alpha q)\big\}
    \in 
    \big(0,\min\{1-s,\tfrac12sp\big\}\big]
$$
and 
$$
    \tilde\epsilon 
    =
    \tfrac12(sp-\alpha q)
    \in 
    \big(0,\tfrac12sp\big].
$$
The application of the Lemma~\ref{lem:tauhh-inproved-basic} results in the fractional differentiability of $\nabla u$ in the sense that
$$
    \nabla u\in W^{\beta,q}_{\loc}(\Om),
    \quad\mbox{where $\beta=\frac{sp-\varep-\tilde\epsilon}{q}\in (\alpha,1)$.}
$$
The lower and upper bounds for $\beta$ are a direct consequence of the choices of $\varep$ and $\tilde\varep$.
Moreover, we have the quantitative estimate
\begin{align*}
     [\nabla u]_{W^{\beta,q}(B_r)}^q
    \le 
    \frac{C}{R^{(1+\beta) q}}\Big(\frac{R}{R-r}\Big)^{N+q+1}
    \boldsymbol{\mathfrak K}^q,
\end{align*}
for every ball
$B_{R}\equiv B_{R}(x_o)\Subset\Omega$ and $r\in(0,R)$ and with a constant $C=C(N,p,s,q,\alpha)$ of the form 
\[
    C
    =
    \frac{\widetilde C(N,p,q)}{s^{q+1}(\frac{sp}{q}-\alpha)^{q+1} }.
\]
Here we have used the abbreviation 
\begin{align*}
    \boldsymbol{\mathfrak K}^q
    :=
    R^{N+(1-\varep)q}[u]^q_{C^{0,1-\varep}(B_{R})} +
    R^q \|\nabla u\|^q_{L^{q}(B_{R})} +
    R^N \boldsymbol{\mathfrak T}(u;R)^q .
\end{align*}
Recall that $\beta\ge\alpha$, so that $\nabla u\in W^{\alpha,q}_{\loc}(\Om)$ and 
\begin{align}\label{frac-basic}
    [u]_{W^{\al,q}(B_{r})}^q
    &\le 
    2^{(\beta-\al)q} 
    r^{(\beta-\alpha)q}
    [u]_{W^{\beta,q}(B_{r})}^q \nonumber\\
    &\le 
    \frac{C}{R^{(1+\alpha) q}}\Big(\frac{R}{R-r}\Big)^{N+q+1}
    \boldsymbol{\mathfrak K}^q.
\end{align}
Note that \eqref{frac-basic} holds for any $s\in(0,1)$. In particular, we have proved the qualitative assertion of Theorem~\ref{thm:grad-frac} in the range $s\le \frac{q(p-1)}{p}$, in which  we have $\beta=\frac{sp}{q}$, cf.~\eqref{precise:beta}.

\textit{Step 2: Almost $W^{\frac{1-(1-s)p}{q-1},q}$-regularity for $s>\frac{q(p-1)}{p}$.}
In this step we concentrate on the case $s>\frac{q(p-1)}{p}$ and employ Lemma~\ref{lem:tauhh-inproved} together with an iteration scheme to upgrade the differentiability to any number less than $\beta=\frac{1-(1-s)p}{q-1}$. 
To this end, let $\theta_o=\frac{sp}{2q}$ and consider $\alpha\in(\theta_o,\be)$.
Moreover, we fix $\varep$ and $\tilde \varep$ by
\begin{equation*}
    \varep
    =
    \min\big\{1-s,\tfrac13(\beta-\alpha)(q-1)\big\}
    \in
    \big( 0, \min\big\{1-s, \tfrac12[1-(1-s)p]\big\}\big]
\end{equation*}
and 
\begin{equation*}
    \tilde\varep
    =
    \min\big\{\tfrac12 \theta_o,\tfrac13(\beta-\alpha)(q-1)\big\}
    \in
    \big( 0, \tfrac12\theta_o\big],
\end{equation*}
such that the requirements on $\varep$ and $\tilde\epsilon$ imposed in Lemma~\ref{lem:tauhh-inproved} are satisfied.
For $i\in\N_0$ introduce 
\begin{equation*}
    \left\{
    \begin{array}{c}
       \displaystyle
       \theta_{i+1}=\frac{1-(1-s)p+\theta_i-\varep-\tilde\epsilon}{q}, \\[8pt]
        \displaystyle
        \rho_i=r +\frac{R-r}{2^{i+1}}, 
        \qquad
        B_i=B_{\rho_i}.
    \end{array}
    \right.
\end{equation*}
By iteration we have
\begin{align*}
    \theta_i
    &=
    \bigg( 1+\frac1{q} +\dots +\frac{1}{q^{i-1}}\bigg)\frac{1-(1-s)p-\varep-\tilde\epsilon}{q}+\frac{\theta_o}{q^i}\\
    &=\frac{1}{q^i}\bigg[\theta_o-\frac{1-(1-s)p-\varep-\tilde\epsilon}{q-1}\bigg]
    +
    \frac{1-(1-s)p-\varep-\tilde\epsilon}{q-1}.
\end{align*}
Since
\begin{align*}
    \theta_o-\frac{1-(1-s)p-\varep-\tilde\epsilon}{q-1}
    \le 
    \theta_o-\beta+\tfrac{2}{3}(\beta-\alpha)
    =
    \theta_o-\tfrac{1}{3}\beta-\tfrac{2}{3}\alpha
    <
    \tfrac{1}{3}(\theta_o-\beta)
    <0,
\end{align*}
the sequence $(\theta_i)_{i\in\N_0}$ is strictly increasing and 
\begin{equation*}
    \lim_{i\to\infty}\theta_i
    =
    \frac{1-(1-s)p-\varep-\tilde\epsilon}{q-1}
    >
    \alpha .
\end{equation*}
From {\it Step~1}, in particular from inequality~\eqref{frac-basic} applied with the choice $\frac{sp}{2q}=\theta_o$ for $\alpha$, we already know that
$\nabla u\in W^{\theta_o,q}_{\rm loc}(\Omega)$,
with the quantitative estimate
\begin{align}\label{est:quant-theta_o}
    R^{(1+\theta_o) q}[\nabla u]_{W^{\theta_o, q}( B_{r})}^{q}
    &\le
    C
    \Big(\frac{R}{R-r}\Big)^{N+q+1}
    \boldsymbol{\mathfrak K}^q .
\end{align}
The constant $C$ has the structure 
$$
    C=\frac{\widetilde C(N,p,q)}{s^{2q+2}}.
$$
Therefore, we are allowed to apply Lemma~\ref{lem:tauhh-inproved} with $(\theta_o,\theta_1,\rho_1,\rho_o)$ in place of $(\theta,\be,r,R)$ and with $\varep$ as above in order to improve the fractional differentiability of $\nabla u$ to order $\theta_1$, i.e.~$\nabla u\in W^{\theta_1,q}_{\rm loc}(\Omega)$.

For the {\bf induction step} $i\to i+1$ let us now assume that we have already established $\nabla u\in W^{\theta_i,q}(B_i)$ for some $i\in\N_0$. Since $\varep \le\frac12[1-(1-s)p]$ and $\tilde\epsilon\le\frac12\theta_o\le\frac12\theta_i$ all assumptions of Lemma~\ref{lem:tauhh-inproved} are fulfilled. The application with
$(\theta_i,\theta_{i+1},\rho_{i+1},\rho_i)$ in place of $(\theta,\beta,r,R)$ and with the above choices of $\varep$ and $\tilde\epsilon$ yields
\begin{align*}
     \rho_{i+1}^{(1+\theta_{i+1}) q}[\nabla u]_{W^{\theta_{i+1},q}(B_{i+1})}^q
    &\le C_i\Big(\frac{\rho_i}{\rho_i-\rho_{i+1}}\Big)^{N+q+1}
    \Big[
      \rho_i^{(1+\theta_i)q}[\nabla u]^q_{W^{\theta_i,q}(B_{i})}  +  \boldsymbol{\mathfrak K}_i^q \Big],
\end{align*}
where
\begin{align*}
    \boldsymbol{\mathfrak K}_i^q
    &:=
    \rho_i^{N+(1-\varep) q}[u]^q_{C^{0,1-\varep}(B_{i})} +
    \rho_i^q \|\nabla u\|^q_{L^q(B_{i})}  +
    \rho_i^N \boldsymbol{\mathfrak T}(u;\rho_i)^q.
\end{align*}
The constant $C_i$ has the structure
\begin{equation*}
    C_i
    =
    \frac{\widetilde C(N,p,q)}{s\tilde\varep^2\theta_i^{q+1}(q-\theta_i-1)^q }
    \le 
    \frac{\widetilde C(N,p,q)}{s^{q+2}\tilde\varep^{q+2}}
    \le 
    \frac{\widetilde C(N,p,q)}{s^{2q+4}(\beta-\widetilde\alpha)^{q+2}}
    =: 
    C_*.
\end{equation*}
To obtain the last inequality we used $\frac{sp^2}{4q}\le \theta_i<\frac{1-(1-s)p-\varep-\tilde\epsilon}{q-1}$. 
To further simplify the previous recursive estimate, observe that
$$
    R
    \ge
    \rho_i
    =
    \frac{R}{2^{i+1}}+ r\Big( 1-\frac{1}{2^{i+1}}\Big)
    >
    \frac{R}{2^{i+1}},
$$
and 
$$
    \rho_i-\rho_{i+1}
    =
    \frac{R-r}{2^{i+2}},
$$
such that we have
$$
\Big(\frac{\rho_i}{\rho_i-\rho_{i+1}}\Big)^{N+q+1}
     \le
     C(N,q)2^{(N+q+1)i} \Big(\frac{R}{R-r}\Big)^{N+q+1}.
$$
In addition, we have
\begin{align*}
    \rho_i^{(1-\varep) q+N }[u]^q_{C^{0,1-\varep}(B_{i})} 
    &\le 
    R^{(1-\varep)q+N}[u]^q_{C^{0,1-\varep}(B_{R})},
\end{align*}
and
\begin{align*}
   \rho_i^q \int_{B_{i}} |\nabla u|^q\,\dx
   \le 
   R^q\int_{B_R} |\nabla u|^q\,\dx.
\end{align*}
Moreover, Remark~\ref{rem:t} yields
\begin{align*}
    \rho_i^{N}\boldsymbol{\mathfrak T}(u;\rho_i)^q
    \le
    \rho_i^{N}
    \bigg[ 
    C(N)
    \Big(\frac{R}{\rho_i}\Big)^{N}\bigg]^\frac{q}{p-1}\boldsymbol{\mathfrak T}(u;R)^q 
    \le
    C(N,p,q)2^{\frac{Nq}{p-1}i}R^{N}\boldsymbol{\mathfrak T}(u;R)^q.
\end{align*} 
Consequently, $\boldsymbol{\mathfrak K}_i$ is estimated by $\boldsymbol{\mathfrak K}$ with a multiplicative factor thanks to the last three estimates, and we arrive at a simpler recursive estimate which reads as
\begin{align*}
     \rho_{i+1}^{(1+\theta_{i+1}) q}[\nabla u]_{W^{\theta_{i+1},q}(B_{i+1})}^q
     \le 
     C_*\boldsymbol{b}^{i} \Big(\frac{R}{R-r}\Big)^{N+q+1}
    \Big[ 
    \rho_i^{(1+\theta_i)q}[\nabla u]^q_{W^{\theta_i,q}(B_{i})}
    +
    \boldsymbol{\mathfrak K}^q \Big].  
\end{align*} 
Here, we denoted $\boldsymbol{b}:=2^{N+q+1+\frac{Nq}{p-1}}$.
Iterating this inequality and taking into account that $\boldsymbol b\ge 1$ and $C_\ast\ge 1$ we obtain
\begin{align*}
    \rho_i^{(1+\theta_i)q}&[\nabla u]^q_{W^{\theta_i,q}(B_{i})}\\
    &\quad\le  
    i\boldsymbol{b}^{1+2+\dots+i} C_*^{i} \Big(\frac{R}{R-r}\Big)^{i(N+q+1)}
    \Big[
    R^{(1+\theta_o)q}[\nabla u]^q_{W^{\theta_o,q}(B_{R})}
    +
    \boldsymbol{\mathfrak K}^q\Big] \\
    &\quad\le  i\boldsymbol{b}^{1+2+\dots+i} C_*^{i} \Big(\frac{R}{R-r}\Big)^{(i+1)(N+q+1)}
    \boldsymbol{\mathfrak K}^q ,
\end{align*}
where from the second to last line we used inequality~\eqref{est:quant-theta_o}. 
In order to conclude, we only need to take $i_o\in \N$ to be the first index satisfying 
\begin{equation*}
    \theta_{i_o}
    \ge 
    \frac{1-(1-s)p-\frac{3}{2}(\varep+\tilde\epsilon)}{q-1}
    \ge 
    \alpha,
\end{equation*}
i.e.
\begin{equation*}
    i_o
    =
    \Bigg\lceil\frac{\ln\Big(1+\frac{\theta_o(q-1)-[1-(1-s)p]}{\epsilon+\tilde\epsilon}\Big)}{\ln q}\Bigg\rceil.
\end{equation*}
Note that $i_o$ only depends on $p, s, q$, and $\alpha$. It is stable as $s\uparrow1$ and blows up as $s\downarrow 0$ and $\alpha\uparrow\beta$. 
Hence,
\begin{align}\label{frac-iter}
    r^{(1+\alpha)q} [u]_{W^{\al,q}(B_{r})}^q
    &\le 
    2^{(\theta_{i_o}-\al)q} 
    r^{(1+\theta_{i_o})q}
    [u]_{W^{\theta_{i_o},q}(B_{\rho_{i_o}})}^q \nonumber\\
    &\le 
    2^q i_o\boldsymbol{b}^{1+2+\dots+i_o} C_*^{i_o} \Big(\frac{R}{R-r}\Big)^{(i_o+1)(N+q+1)}
    \boldsymbol{\mathfrak K}^q.
\end{align}

\textit{Step 3: Almost $W^{\beta,q}$-regularity.}
By combining \textit{Step~1} and \textit{Step~2}, in particular~\eqref{frac-basic} and~\eqref{frac-iter}, we have  established that $\nabla u\in W^{\alpha,q}_{\rm loc}(\Omega)$ for every $\alpha\in(0,\beta)$ together with the quantitative estimate 
\begin{align*}
    &[u]_{W^{\al,q}(B_{\frac12 R})}^q \\
    &\quad\le 
    \frac{C}{R^{(1+\alpha)q}}
    \bigg[R^{N+(1-\varep)q}[u]^q_{C^{0,1-\varep}(B_{\frac34 R})} +
    R^q \int_{B_{\frac34 R}} |\nabla u|^q\dx +
    R^N \boldsymbol{\mathfrak T}\big(u; \tfrac34 R\big)^q\bigg].
\end{align*}
Finally, we use Theorem~\ref{thm:W1q-p<2}, Theorem~\ref{thm:Hoelder-p<2}, and Lemma~\ref{lem:t} to further estimate the terms in the bracket on the right-hand side. After slightly adjusting the radii we obtain 
\begin{align*}
    [u]_{W^{\al,q}(B_{\frac12 R})}^q
    &\le 
    \frac{C R^{N}}{R^{(1+\alpha)q}}
    \Big[R^{s-\frac{N}{p}}(1-s)^{\frac1p}[u]_{W^{s,p}(B_{R})} +
    \boldsymbol{\mathfrak T}(u; R)\Big]^q.
\end{align*}
This finishes the proof of the theorem.
\end{proof}

\end{document}